\documentclass[11pt,oneside]{article} 
\usepackage[english]{babel}
\usepackage{amsmath,amssymb,amsthm,amscd,bbold,mathdots,epsfig,eucal,enumitem}
\usepackage{mathabx}
\usepackage{txfonts}
\usepackage{appendix}
\usepackage{wrapfig}
\usepackage{makeidx}
\usepackage{wasysym}
\usepackage{psfrag}
\usepackage{graphicx}  
\usepackage{comment}
\usepackage{color}
\usepackage{chngcntr}

\theoremstyle{plain}
\newtheorem{theorem}{Theorem}
\newtheorem{lemma}[theorem]{Lemma}
\newtheorem{proposition}[theorem]{Proposition}
\newtheorem{corollary}[theorem]{Corollary}
\newtheorem{assumption}[theorem]{Assumption}

\theoremstyle{definition}

\newtheorem{definition}[theorem]{Definition}
\newtheorem{notations}[theorem]{Notations}
\newtheorem{mainhypothesis}[theorem]{Main hypothesis}

\DeclareMathOperator*{\argmax}{arg\,max}

\begin{document}

\newcommand{\rot}{\operatorname{rot}}
\newcommand{\dive}{\operatorname{div}}
\newcommand{\grad}{\operatorname{grad}}
\newcommand{\vvec}{\overrightarrow}
\newcommand{\Var}{{\rm Var}}
\newcommand{\Cov}{{\rm Cov}}
\newcommand{\Card}{{\rm Card}}
\newcommand{\Rel}{{\rm Re}}
\newcommand{\Img}{{\rm Im}}
\newcommand{\codim}{{\text{\rm codim\,}}}
\newcommand{\Pos}{\textrm{\rm Pos}}
\newcommand{\Sym}{\textrm{\rm Sym}}
\newcommand{\stab}{\textrm{\rm stab}}
\newcommand{\GL}{\textrm{\rm GL}}
\newcommand{\Per}{\textrm{\rm Per}}
\newcommand{\card}{\text{\rm card}}
\newcommand{\Vect}{\text{\rm span}}
\newcommand{\diag}{{\rm diag}}
\newcommand{\trace}{\text{\rm trace}}
\newcommand{\End}{\text{\rm End}}
\newcommand{\Leb}{\text{\rm Leb}}
\newcommand{\Orb}{\textrm{\rm Orb}}
\newcommand{\Orth}{\textrm{\rm O}}
\newcommand{\range}{\textrm{\rm Im}}
\newcommand{\Graph}{\textrm{\rm Graph}}
\newcommand{\rank}{\textrm{\rm rank}}
\newcommand{\one}{\mathbb{1}}
\newcommand{\Id}{{\textrm{\rm Id}}}
\newcommand{\dist}{\textrm{\rm dist}}
\newcommand{\Grass}{\textrm{\rm Grass}}
\newcommand{\coGrass}{\textrm{\rm coGrass}}
\newcommand{\ortho}{\mathord{\perp}}
\newcommand{\flag}[1]{{\textcolor{red}{\textbf{#1}}}}
\newcommand{\Dstar}{D_{\textrm{\tiny SVG\,}}}
\newcommand{\Gstar}{D_{\textrm{\tiny FI\,}}}
 
\title{Explicit bounds for separation  between \\  Oseledets subspaces} 
\author{Anthony  Quas  \footnote{Department of Mathematics and Statistics, 
University of Victoria, Victoria, BC, Canada, V8W 3R4, aquas@uvic.ca}, 
Philippe Thieullen \footnote{ Institut de Math\'ematiques de Bordeaux, 
Universit\'e Bordeaux 1, 33405 Talence cedex, France,  
philippe.thieullen@u-bordeaux.fr}, 
Mohamed Zarrabi 
\footnote{ Institut de Math\'ematiques de Bordeaux, Universit\'e Bordeaux 1, 
33405 Talence cedex, France, Mohamed.Zarrabi@u-bordeaux.fr}
}
\date{May 27th 2018}
\maketitle 

\counterwithin{theorem}{section}

\begin{abstract} 
We consider a two-sided sequence of  bounded operators in a Banach space 
which are not necessarily injective and satisfy two properties 
(SVG) and (FI). The {\it singular value gap} (SVG) property says that  
two successive  singular values of the  cocycle at some index $d$  admit a 
uniform exponential gap; the {\it fast  invertibility} (FI) property  says that the 
cocycle is uniformly invertible on the fastest $d$-dimensional direction. We prove  
the existence of a uniform equivariant splitting of the Banach space into a fast 
space of dimension $d$  and a slow space of codimension $d$. We compute 
an explicit constant lower bound on the angle between these two spaces using 
solely the constants defining the properties (SVG) and (FI). 
We extend the results obtained by Bochi and Gourmelon in the finite-dimensional
case for bijective operators  and the results obtained by Blumenthal and Morris in
the infinite dimensional case for injective norm-continuous cocycles, in the 
direction that the operators are not required to be globally injective, that no dynamical system is involved and no compactness of the underlying 
system or smoothness of the cocycle is required. Moreover we give quantitative 
estimates of the angle between the fast and slow spaces that are new even in the case of
finite-dimensional bijective operators in Hilbert spaces.
\end{abstract}

\section{Introduction}
Let $X$ be a real Banach space and $(A_k)_{k\in\mathbb{Z}}$ be a bi-infinite sequence
of bounded operators of $X$ which are not required to be  injective. 
The \emph{cocycle} associated to $(A_k)_{k\in\mathbb{Z}}$ is the sequence of 
iterated operators
\begin{equation*}
A(k,n) := A_{k+n-1} \cdots A_{k+1}A_k, \quad\forall k\in\mathbb{Z}
\ \ \text{and} \ \ \forall n\geq0,
\end{equation*} 
with the convention $A(k,0):=\Id$. Our main objective is to find simple conditions 
on the sequence $(A_k)_{k\in\mathbb{Z}}$ which guarantee the existence of 
constants $d\geq1$, $K_d\geq1$ and $\tau>0$,  and a uniform equivariant 
splitting $X = E_k \oplus F_k$ of fast/slow subspaces  satisfying the following condition:
\begin{itemize}
\item $\forall k\in\mathbb{Z}, \quad \dim(E_k) = d$, \ $(A_k|E_k)$ is injective,
\item $\forall k\in\mathbb{Z}, \quad  A_k E_k = E_{k+1}$  \  and 
\ $A_k F_k \subset F_{k+1}$, \quad (the equivariance property),
\item $\inf_{k\in \mathbb{Z}}\gamma (E_k,F_k) >0$, \quad (the uniform minimal gap property), 
\item $\displaystyle{\forall k \in \mathbb{Z}, \ \forall n\geq1,
\ \frac{\| A(k,n)  |   F_k \| }{ \| (A(k,n)  |  E_k)^{-1} \| ^{-1} } \leq K_d e^{-n\tau}}$,
\quad (the slow/fast ratio property)
\end{itemize}
where $\gamma(E_k,F_k)$ denotes the minimal gap between $E_k$ and $F_k$ 
(a notion of minimal angle between two complementary spaces, see definition 
\ref{definition:minimumGap}), 
\[
\gamma(E_k,F_k) := \inf \{ \dist(u,F_k) : u \in E_k, \ \|u\|=1 \},
\] 
and $\| (A(k,n)  |  E_k)^{-1} \| ^{-1}$ and $\| A(k,n)  |  F_k \|$ denote respectively 
the lowest and largest expansion of the cocycle restricted to $E_k$ and $F_k$,
\begin{gather*}
\| A(k,n)  |  F_k \| := \sup \{ \| A(k,n)v \| : v \in F_k, \ \|v\|=1 \}, \\
\| (A(k,n)  |  E_k)^{-1} \| ^{-1} :=  \inf \{ \|A(k,n)u \| : u \in E_k, \ \|u\|=1 \}.
\end{gather*}
(The notation $\| (A | E)^{-1}\|^{-1}$ will be used only when $\dim(E) < +\infty$ and $A:E\to  X$ is injective). In order to distinguish the two equivariant subspaces in this exponential dichotomy, 
we will use the terminology {\it fast space} for $E_k$ and {\it slow space}  for 
$F_k$ although both operators $A(k,n) : E_k \to E_{k+n}$ and 
$A(k,n) : F_k \to F_{k+n}$ may be expanding or contracting. The index 
$k$ denotes the position of the cocycle  and  $n$ represents the order of 
iteration. We interpret  $A(k,n)$ as an operator acting from a space above 
$k$ to a space above $k+n$; in particular the dual operator $A(k,n)^*$ 
acts on the dual space as an operator from a space above $k+n$ to a space above $k$.

Our main assumption is related to the existence of a uniform gap in the singular value 
decomposition at index $d$. The notion of singular values for an operator in  a 
general Banach space is not well defined.   We define the \emph{singular value} of index 
$d\geq 1$ of an operator $A$, to be the number 
\begin{equation*}
\sigma_d(A) := \sup_{\dim (E)=d} \ \inf_{u \in E\setminus\{0\}} \frac{\| A u \|}{\| u \|}.
\end{equation*}
 Equivalent definitions $\sigma_d'(A), \sigma_d''(A)$ are given in \ref{definition:SingularValueBis} and 
 \ref{definition:singularValueTer}. In the Hilbert case, all these definitions are equal. To simplify the notations, we  use 
\[
\sigma_d(k,n) := \sigma_d(A(k,n)).
\] 
The top singular value is $\sigma_1(k,n) = \| A(k,n) \|$ and, in the particular case  
$\dim X=d$ and $A(k,n)$ is  invertible, the bottom singular value is  
$\sigma_d(k,n) = \| A(k,n)^{-1} \|^{-1}$.

\begin{mainhypothesis}
Let $X$ be a real  Banach space and $(A_k)_{k \in \mathbb{Z}}$ be a 
sequence of bounded operators (not necessarily injective nor surjective). 
We assume  there exist an integer $d\geq1$ and  constants $\Dstar,\Gstar \geq1$,  
$\tau,\mu>0$ such that

\begin{itemize}
\item the sequence admits a uniform {\it  singular value gap}  at  index $d$

$\text{(SVG)} \quad\quad\quad \forall k \in \mathbb{Z}, \forall n\geq0, 
\quad \left\{\begin{array}{l} \displaystyle  
\frac{ \sigma_{d+1}(k,n) \| A_{k+n}\|}{\sigma_d(k,n+1)}  \leq 
\Dstar e^{-n\tau} \\ \\ \displaystyle 
\frac{ \| A_{k}\| \sigma_{d+1}(k+1,n)}{\sigma_d(k,n+1)}  
\leq \Dstar e^{-n\tau} \end{array}\right.$

\noindent (We implicitly assume that $\sigma_d(k,n)>0$ for every 
$k\in\mathbb{Z}$ and $n\geq0$),

\item  the sequence is {\it d-dimensionally fast invertible}

$\displaystyle \text{(FI)} \quad\quad\quad \displaystyle  \forall m\geq0, \quad \inf_{k\in\mathbb{Z}, \ n\geq0} \prod_{i=1}^d
\frac{\sigma_i(k-m,m+n) }{ \sigma_i(k-m,m) \sigma_i(k,n) } \geq \Gstar ^{-1} e^{-m\mu}$.
\end{itemize}
\end{mainhypothesis}

Property (FI) is a new property that can be used as a substitute for  
uniform invertibility along $d$-dimensional spaces. It is an asymmetric 
property  with respect to forward and backward iterations related to the fact that the fast space (respectively the slow space) has dimension $d$ (respectively codimension $d$). We will show, 
thanks to  the super-multiplicative property of a similar quotient, that (FI) 
is equivalent to a seemingly weaker property with $m=1$,
\[
\text{(FI)} \ \Longleftrightarrow \ \text{(FI)}_{\text{weak}} \quad   \quad e^{-\nu} := \inf_{k\in\mathbb{Z}, \ n\geq0}
\ \prod_{i=1}^d\frac{\sigma_i(k-1,1+n) }{ \sigma_i(k-1,1) \sigma_i(k,n) } > 0.
\]
We have chosen the other form to quantify precisely the minimal gap between the fast and slow spaces in our main theorem \ref{theorem:main} in the Banach spaces case. In the Hilbert spaces case we may choose  $\Gstar =1$ and $\nu=\mu$.

Property (FI)  is used as a sufficient and necessary hypothesis in a bootstrap argument. 
Our main result actually shows that the cocycle  must satisfy a
 stronger property $\text{(FI)}_{\text{strong}}$ with a uniform lower bound independent of $m$,
\begin{itemize}
\item[] $\displaystyle{\text{ (FI)}_{\text{strong}} \quad\quad\quad 
\inf_{m\geq0} \ \inf_{k\in\mathbb{Z}, \ n\geq0} \quad  
\prod_{i=1}^d\frac{\sigma_i(k-m,m+n) }{ \sigma_i(k-m,m) \sigma_i(k,n) }>0}$.
\end{itemize}
We will show
\[
\text{(SVG) \ and \ (FI)} \quad\Longrightarrow \quad \text{ (FI)}_{\text{strong}}.
\]
Notice  that we do not assume that the norm of the operators  
$A_k$ is uniformly bounded from above. Notice also that $A_k$ may not be invertible.

If the cocycle  is {\it uniformly invertible} (UI) in the sense 
\begin{itemize}
\item[] $\text{(UI)} \quad\quad \sup_{k\in\mathbb{Z}} \|A_k\| \leq 
M^* \ \ \text{and} \ \ \inf_{k\in\mathbb{Z}} \| A_k^{-1} \|^{-1} \geq M_*$
\end{itemize}
for some constants $M^*,M_*>0$,  property (FI) is automatically true with 
$\Gstar =1$ and $\mu := d\log(M^*/M_*)$. In that case our main result implies
\begin{gather*}
\text{(UI)} \quad\Longrightarrow\quad \text{(FI)}, \quad\quad\quad
\text{(SVG) \ and \ (UI)} \quad \Longrightarrow \quad \text{(FI)}_{\text{strong}}.
\end{gather*}

The  singular value gap property (SVG) admits a weaker form. This weaker form 
is actually equivalent to the strong one for uniformly invertible cocycles  and 
was introduced by Bochi and Gour\-me\-lon in \cite{BochiGourmelon2009_1} 
for the first time,

\begin{itemize}
\item[] $\displaystyle  \text{(SVG)}_{\text{weak}} 
\quad\quad\quad  \forall k \in \mathbb{Z}, \forall n\geq0, \quad  
\frac{ \sigma_{d+1}(k,n)}{\sigma_d(k,n)} \leq \Dstar e^{-n\tau}$.
\end{itemize}

The strong form (SVG) was introduced by Blumenthal and Morris in 
\cite{BlumenthalMorris2015_1} in order to extend the results of 
Bochi and Gourmelon to the infinite-dimensional case.
They nevertheless assume the cocycle to be 
norm-continuous over a compact dynamical system and each operator 
$A_k$ to be injective.  Our property (FI) is used instead of the injectiveness  
assumption.  Moreover we do not assume that the cocycle  is defined 
over a dynamical system, nor do we require regularity conditions as in 
\cite{BochiGourmelon2009_1,BlumenthalMorris2015_1}. Our main 
objective is to obtain an effective splitting of the Banach space into a 
fast and a slow space, equivariant under the cocycle,  for which the angle 
between the two spaces can be explicitly bounded from below using only the 
constants $(\Dstar,\Gstar , \tau,\mu)$ while avoiding the use of compactness of the 
underlying dynamical system and  regularity assumptions on the cocycle.

Our estimates depend on a constant $K_d$ which is only a function of 
the dimension $d$  and the Banach space. For a Hilbert 
space $K_d=1$, for a general Banach space, $K_d$ is explicitly computed 
given a volume distortion $\Delta_d(X)$ (see definition 
\ref{definition:volumeDistortion}) which measures the distortion of 
the unit Banach ball to the best fitted Euclidean ball. We have that $\Delta_d(X)\le \sqrt d$
for Banach spaces and $\Delta_d(X)=1$ for Hilbert spaces. We give an 
estimate of $\Delta_d(X)$ in proposition 
\ref{proposition:simpleDistortionEstimate} when $X=\ell_d^p$ is the 
space of dimension $d$  equipped  the $p$-norm. We do not intend to 
undertake a systematic study of $\Delta_d(X)$.  We have chosen to give a unified proof for both Banach and Hilbert spaces in such a way the constants appearing in the estimates become optimal in the Hilbert case.

Our main result is the following.

\begin{theorem} \label{theorem:main}
Let $X$ be a Banach space, $d\geq1$, and $(A_k)_{k \in \mathbb{Z}}$ 
be a sequence of bounded operators satisfying the two assumptions 
\text{\rm(SVG)} and \text{\rm(FI)}  at the index $d$, for some constants 
$\Dstar, \Gstar \geq1$ and $\tau,\mu>0$. Then  there exist  a constant 
$K_d$  depending only on the dimension $d$ and the Banach norm such that,
\begin{enumerate}
\item \label{item:main_1} there exists an equivariant splitting
$X= E_k \oplus F_k$  satisfying for every $k \in \mathbb{Z}$,
\begin{itemize}
\item  $\dim(E_k) = d$, \ $A_k(E_k) = E_{k+1}$, \ $A_k(F_k) \subset F_{k+1}$,
\item $\displaystyle \gamma(E_k,F_k) \geq \frac{1}{5 K_d \Gstar } 
\Big[ \frac{(3d+7)^{-2}}{2  K_d\Gstar } \frac{1-e^{-\tau}}{\Dstar e^{\tau}}  
\Big]^{\frac{\mu(\mu+4\tau)}{2\tau^2}}$,
\end{itemize}
\item \label{item:main_2}  $\text{\rm (FI)} \Leftrightarrow 
\text{\rm (FI)}_{\text{\rm strong}}$. More precisely for every $k\in\mathbb{Z}$, $m,n\geq1$,
\end{enumerate}
\begin{gather*}
\prod_{i=1}^d \frac{\sigma_i(k-m-n,m+n)}{\sigma_i(k-m,m)\sigma_i(k,n)} \geq 
\frac{3}{25  K_d \Gstar ^3} \Big[ \frac{(3d+7)^{-2}}{2  K_d \Gstar } 
\frac{1-e^{-\tau}}{\Dstar e^{\tau}}  \Big]^{\mu(\mu^2+5\mu\tau+8\tau^2)/2\tau^3},
\end{gather*}
\begin{enumerate}
\addtocounter{enumi}{2}
\item \label{item:main_3} The spaces $E_k$ and $F_k$ are called the \emph{fast} and 
\emph{slow} spaces respectively and satisfy:
for every $k\in\mathbb{Z}$ and $n$ such that,
\[
n \geq \Big(1+\frac{\mu(\mu+4\tau)}{2\tau^2} \Big)\frac{1}{\tau} \log
\Big( \frac{\Dstar e^\tau}{1-e^{-\tau}}2 (3d+7)^2K_d \Big),
\]
\begin{itemize}
\item $\| (A(k,n) | E_k)^{-1} \|^{-1} \geq \frac{3}{5} K_d^{-1}    \gamma(E_k,F_k)  \sigma_d(k,n)$,
\item $\| A(k,n) | F_k \| \leq 3 K_d    \gamma(F_{k+n},E_{k+n})^{-1} \sigma_{d+1}(k,n)$.
\end{itemize}
\end{enumerate}
Using the definition of $\bar\Delta_d(X)$ in equation 
\eqref{equation:simplifiedVolumicDistortion}, and the constants $C_{0,d}$ 
and $\widehat C_{0,d}$ in theorems \ref{theorem:PolarDecomposition} and 
\ref{theorem:codimensionOneSingularValueDecomposition}, with $\epsilon=0$, 
we obtain
\[
K_d :=  \widehat C_{0,d}^7 C_{0,d}^{8d+5} \bar\Delta_2(X)^{4d}\bar\Delta_d(X)^{8d} \leq  (2d)^{2000d^3}.
\]
If $X$ is a real Hilbert space then $K_d=1$ and $\Gstar $ may be chosen equal to 1 in  \text{\rm (FI)}.
\end{theorem}

Our main result extends the results of Bochi and Gourmelon \cite{BochiGourmelon2009_1}
in the case   $X=\mathbb{R}^d$ in three ways: we do not assume the 
cocycle to be invertible, we do not introduce a dynamical system, 
we do not assume either $C^0$ regularity or compactness. The proof used in 
\cite{BochiGourmelon2009_1} requires all these assumptions and actually
needs the ergodic Oseledets theorem for invariant probability measures.  
We have chosen to work in two directions: a direction which gives explicit 
estimates, especially for the lower bound of the angle, with respect to the 
initial data, and a direction which gives an unified proof for Banach and 
Hilbert spaces. In order not to introduce artificial constants in the 
Banach setting, we found it necessary to develop in appendix 
\ref{appendix:basicResultsBanachSpaces} a theory of volume distortion 
$\bar\Delta_d(X)$ which enables us to quantify on each $d$-dimensional 
space the distortion of the Banach norm with respect to the best fitted 
Euclidean norm. The volume distortion  $\bar\Delta_d(X)$ is 1 in the 
Hilbert case. We express all estimates in terms of a constant $K_d$ 
that is only a function of $\bar\Delta_d(X)$ and satisfies $K_d=1$ in the Hilbert case.  

In item \ref{item:main_1}  we obtain an explicit lower bound of the angle 
between the fast and slow spaces depending only on $\Dstar,\Gstar ,\tau,\mu$ 
and the dimension $d$. We have chosen to give a uniform estimate for 
every $k\in\mathbb{Z}$ instead of an asymptotic estimate as $k\to\pm\infty$. 
This choice has led to additional computation. 

In item \ref{item:main_2} we prove the strong form $\text{(FI)}_{\rm strong}$. 
This is actually a simple consequence of  lemma \ref{lemma:boundBelowFBI} and 
the uniform bound $\inf_{k\in\mathbb{Z}} \gamma(E_k,F_k)>0$. We nevertheless give 
a precise  estimate valid for all iterates $m,n$ and not just  for $m,n\to+\infty$. In the Hilbert case, the estimate is simpler with $K_d=1$ and $\Gstar =1$ in  (FI).

In item \ref{item:main_3} we show that the two equivariant splittings 
correspond indeed to the fast and slow spaces; we again made the decision 
to give explicit but not optimal estimates. The singular value of index $d$ 
of the cocycle  restricted to the fast space is comparable up to a factor 
given by the minimal gap $\gamma(E_k,F_k)$ to the original $d$-dimensional 
singular value.  A similar result is obtained for the slow space. For large $n$ and in the Hilbert case, the two constants $\frac{3}{5}K_d^{-1}$ and $3K_d$ may be replaced by 1.

The proof of our main result is divided into 3 parts. In section \ref{section:fastSlowSpaces}, we show how property (SVG) implies the existence of two fast and  slow spaces that may not be complementary. This mechanism is standard since Raghunathan \cite{Raghunathan1979_1} in finite dimension, Ruelle \cite{Ruelle1982_1} in Hilbert spaces,  Blumenthal-Morris \cite{BlumenthalMorris2015_1} in Banach spaces, and Gonz\'alez-Tokman-Quas \cite{GonzalezTokmanQuas2014_1} for a shorter proof. Our proof quantifies precisely the speed of convergence of the approximate spaces.  In section \ref{section:proofMainResult_1}, we show how property (FI) implies that the two fast and slow spaces give a splitting of the ambient space. This part is the heart of the proof and is new. In section \ref{section:proofMainResult_2_3}, we show that (FI) is a necessary and sufficient condition and actually equivalent to a stronger condition $\text{\rm (FI)}_{\text{\rm strong}}$. In appendix \ref{appendix:basicResultsBanachSpaces}, we recall basic definitions of the geometric theory of Banach spaces. We recall different notions of distance between subspaces, several notions of singular values, some facts about the projective norm on the exterior product. The main purpose of this appendix is to recall without proofs the standard approximate singular value decomposition theorem \ref{theorem:PolarDecomposition}. 

\section{Construction of the fast and slow spaces}
\label{section:fastSlowSpaces}

The proof of our main result is based on a version of the singular value 
decomposition (SVD) theorem for a single bounded operator in the 
Banach setting. The (SVD) theorem is well known for compact operators 
in a Hilbert space (see \cite{Pietsch1987}). We did not find a version 
of the (SVD) theorem adapted to our needs in the literature. Appendix 
\ref{appendix:basicResultsBanachSpaces} fills in this missing piece. The 
main interest of Appendix \ref{appendix:basicResultsBanachSpaces} is 
theorem \ref{theorem:PolarDecomposition} which shows the existence of 
approximate singular spaces at every index $d$. The singular spaces may 
not be exact because of the non compactness of the operators and are 
thus non canonical. They depend for instance on an arbitrarily small  
constant $\epsilon>0$ coming  from the fact that, in the case of infinite 
Banach or Hilbert spaces,  the norm of an operator may not be attained 
by a vector of the unit sphere. Notice that we shall not use the (FI) condition in this section.

The following theorem is a special version of theorem \ref{theorem:PolarDecomposition} 
applied to each operator $A(k,n)=A_{k+n-1}\cdots A_{k+1}A_k$. We fix 
$\epsilon>0$ and the index $d\geq1$. We show there exist a pair of 
complementary spaces $X=U(k,n) \oplus V(k,n)$ of the source space 
and a pair of complementary spaces $X=\tilde U(k+n,n) \oplus \tilde V(k+n,n)$ 
of the target space that are related by $A(k,n)$ and $A(k,n)^*$. We 
replace the usual notion of orthogonality by a weaker notion using 
$C$-Auerbach families (see definition \ref{definition:CAuerbachBasis} 
for more details). We show that the  two splittings are $C_{\epsilon,d}$-orthogonal 
in the sense of the following definition.

\begin{definition}
Let $X$ be a Banach space, $d\geq1$, $C\geq1$.
\begin{itemize}
\item We say that a family of vectors $(u_1,\ldots,u_d)$ is {\it $C$-Auerbach} if 
\[
\forall j=1,\ldots,d, \quad C^{-1} \leq \dist(u_j,\Vect(u_i:i\not=j)) \leq \|u_j\| \leq C.
\]
\item We say a splitting $X=U\oplus V$ with $\dim(U)=d$ is {\it $C$-orthogonal}
if there exist a $C$-Auerbach basis $(e_1,\ldots,e_d)$ spanning $U$ and a 
$C$-Auerbach basis $(\phi_1,\ldots,\phi_d)$ spanning $V^\perp$ in the dual 
space $X^*$  which are dual to each other, that is $\langle \phi_i | e_j \rangle = \delta_{i,j}, \ \forall i,j = 1,\ldots,d$.
\end{itemize}
\end{definition}

If $V\subset X$ is a subspace of $X$, the \emph{annihilator} of $U$ is the subspace 
in the dual space, $U^\perp := \{ \phi \in X^* : \langle \phi | u \rangle =0, 
\ \forall u \in U\}$. If $H \subset X^*$, the \emph{pre-annihilator} of $H$ 
is the subspace in $X$,  $H^\Perp := \{ v \in X : \langle \eta | v \rangle = 0,
\ \forall \eta \in H \}$.

\begin{theorem}[\bf Approximate singular value decomposition] 
\label{theorem:aprroximatedSingularSpaces}
Let $X$ be a Banach space, $d\geq1$, $\epsilon>0$, and 
$(A_k)_{k\in\mathbb{Z}}$ be a sequence of bounded operators. 
Then there exists a constant $K_d \geq1$ depending only on the Banach norm 
and $d$, such that for every $k\in\mathbb{Z}$, $n\geq1$, and  
$C_{\epsilon,d}:=(1+\epsilon)K_d$, 

\begin{enumerate}
\item \label{item:aprroximatedSingularSpaces_1} there exist  two 
$C_{\epsilon,d}$-orthogonal splittings:
\begin{itemize}
\item $X = U(k,n) \oplus V(k,n)$, \quad  $X=\tilde U(k,n) \oplus \tilde V(k,n)$, 
\item $\dim(U(k,n))=\dim(\tilde U(k,n))=d$,
\item $A(k,n) U(k,n) = \tilde U(k+n,n)$, \quad $A(k,n) V(k,n) \subset \tilde V(k+n,n)$, 
\item $A(k,n)^* \tilde U(k+n,n)^\perp \subset U(k,n)^\perp$, \quad 
$A(k,n)^* \tilde V(k+n,n)^\perp = V(k,n)^\perp$,
\end{itemize}

\item \label{item:aprroximatedSingularSpaces_2} the singular values of 
$A(k,n)$ and $A(k,n)^*$ restricted to this splitting are comparable to 
those of $A(k,n)$ on $X$: for every $1 \leq i \leq d$,
\begin{itemize}
\item $\sigma_i(k,n) \geq \sigma_i(A(k,n) | U(k,n)) \geq \sigma_i(k,n)/C_{\epsilon,d}$, 
\item $\sigma_i(k,n) \geq \sigma_i(A(k,n)^* | \tilde V(k+n,n)^\perp) \geq 
\sigma_i(k,n)/C_{\epsilon,d}$, 
\item $\sigma_{d+1}(k,n) \leq \|A(k,n) | V(k,n)\| \leq \sigma_{d+1}(k,n) C_{\epsilon,d}$,
\item $\sigma_{d+1}(k,n) \leq \|A(k,n)^* | \tilde U(k+n,n)^\perp\| 
\leq \sigma_{d+1}(k,n) C_{\epsilon,d}$,
\end{itemize}

\item \label{item:aprroximatedSingularSpaces_3} the minimal gap of the 
two splittings is uniformly bounded from below,
\begin{gather*}
\gamma(U(k,n), V(k,n))  \geq 1/C_{\epsilon,d}, \quad 
\gamma(V(k,n),U(k,n))  \geq 1/C_{\epsilon,d}, \\
\gamma(\tilde U(k,n), \tilde V(k,n))  \geq 
1/C_{\epsilon,d}, \quad \gamma(\tilde V(k,n),\tilde U(k,n))  \geq 1/C_{\epsilon,d},
\end{gather*}

\item \label{item:aprroximatedSingularSpaces_4} there exits a pair of  
$C_{\epsilon,d}$-Auerbach families  of (the source space) $X$, $X^*$,
\[
(e_1(k,n),\ldots,e_d(k,n)), \quad (\phi_1(k,n),\ldots,\phi_d(k,n)) 
\]
and a pair of  $C_{\epsilon,d}$-Auerbach families  of (the target space) $X$, $X^*$,
\[
(\tilde e_1(k+n,n),\ldots,\tilde e_d(k+n,n)), \quad (\tilde \phi_1(k+n,n),\ldots,\tilde \phi_d(k+n,n))
\]
satisfying
\begin{itemize}
\item $\langle \phi_i(k,n) | e_j(k,n) \rangle = \delta_{i,j}, \quad 
\langle \tilde \phi(k,n) | \tilde e_j(k,n) \rangle = \delta_{i,j}$,
\item $A(k,n) e_i(k,n) = \sigma_i(k,n) \tilde e_i(k+n,n)$,
\item $A(k,n)^*\tilde \phi_i(k+n,n) = \sigma_i(k,n) \phi_i(k,n)$,
\item $U(k,n) = \Vect(e_1(k,n),\ldots,e_d(k,n))$, 
\item $V(k,n) = \Vect(\phi_1(k,n),\ldots,\phi_d(k,n))^\Perp$, 
\item $\tilde U(k+n,n) = \Vect(\tilde e_1(k+n,n), \ldots, \tilde e_d(k+n,n))$,
\item $\tilde V(k+n,n) = \Vect(\tilde \phi_1(k+n,n), \ldots, \tilde \phi_d(k+n,n))^\Perp$.
\end{itemize}
\item Moreover $K_d=1$ if $X$ is a Hilbert space and $\epsilon$ may be chosen to
be zero if $X$ is finite-dimensional.
\end{enumerate}
\end{theorem}

We call $U(k,n)$ and $V(k,n)$, the \emph{approximate fast and slow forward spaces} above $k$. 
Similarly we will call $\tilde U(k,n)$ and $\tilde V(k,n)$, defined using $A(k-n,n)$,
the \emph{approximate fast and slow backward spaces} above $k$. 
Since the approximate forward spaces are built using the sequence of operators 
$(A_k,A_{k+1}, \ldots,A_{k+n-1})$ and the approximate backward spaces are built 
using $(A_{k-n},A_{k-n+1},\ldots,A_{k-1})$, the two splittings above $k$, 
$X=U(k,n) \oplus V(k,n)$ and $X = \tilde U(k,n) \oplus \tilde V(k,n)$, need not be 
closely related.

We first consider the construction of the slow spaces $(F_k)_{k\in\mathbb{Z}}$ using the forward cocycle $(A_n)_{n=k}^{+\infty}$ and their approximate slow forward spaces $V(k,n)$.

The following lemma shows an exponential contraction between the two 
approximate  slow forward spaces. The maximal gap $\delta(V,W)$ 
between  $V$ and $W$  is a standard notion of distance  between 
two subspaces (see  definition \ref{definition:maximalGap} and 
equivalent formulations -- note the asymmetry in the definition).
\[
\delta(V,W) = \sup \{ \dist(v,W) : v \in V, \ \|v\|=1 \}.
\]

\begin{lemma}[\bf  Raghunathan estimate I] \label{lemma:RaghunathanEstimateI}
Suppose that the sequence of operators $(A_k)$ satisfies (SVG). 
Then for every $k\in\mathbb{Z}$ and $n\geq1$,
\begin{gather}\label{eq:RagI}
\begin{split}
&\delta(V(k,n),V(k,n+1))  \leq  C_{\epsilon,d}^2 \Dstar e^{-n\tau}, \\
&\delta(V(k,n+1),V(k,n))  \leq   C_{\epsilon,d}^2  \Dstar e^{-n\tau}
/(1-C_{\epsilon,d}^2 \Dstar e^{-n\tau}).
\end{split}
\end{gather}
\end{lemma}

\begin{proof}
Let $v\in V(k,n)$ and $\phi \in V(k,n+1)^\perp$ be of norm 1. 
Choose $\tilde \phi \in  \tilde V(k+n+1,n+1)^\perp$ such that 
$\phi = A(k,n+1)^*\tilde \phi$. Using item 
\ref{item:aprroximatedSingularSpaces_2} of theorem  
\ref{theorem:aprroximatedSingularSpaces} one obtains   on the one hand
\[
\| \phi \|=  \| A(k,n+1)^* \tilde  \phi \| \geq 
\frac{ \sigma_d(k,n+1) }{ C_{\epsilon,d}}\| \tilde \phi \|,
\]
and on the other hand
\begin{align*}
\langle \phi | v \rangle &= \langle \tilde \phi | A(k,n+1) v \rangle \\
&\leq \| \tilde  \phi \| \ \| A(k,n+1) v \| \leq \| \tilde \phi \| \| A_{k+n} \| \| A(k,n) v \| \\
&\leq   C_{\epsilon,d} \| A_{k+n} \|  \sigma_{d+1}(k,n) \| \tilde \phi \| \|v\| \\
&\leq  C_{\epsilon,d}^2 \frac{\| A_{k+n} \| \sigma_{d+1}(k,n)}{\sigma_d(k,n+1)} \| \phi \| \|v\|\\
&\leq C_{\epsilon,d}^2 \Dstar e^{-n\tau},
\end{align*}
where the last line follows from (SVG). The first estimate in \eqref{eq:RagI} then
follows from \eqref{eq:maxgap}.
The second estimate is obtained using  equation \eqref{equation:symmetryMaximalGap},
\[
\delta(V(k,n+1),V(k,n)) \leq  \frac{\delta(V(k,n),V(k,n+1))}{1 - \delta(V(k,n),V(k,n+1)) }. \qedhere
\] 
\end{proof}

The previous lemma shows that the gap between two successive 
$V(k,n)$ is exponentially small. This implies in particular that 
$(V(k,n))_{n\geq1}$ is a Cauchy sequence and that $V(k,n) \to F_k$
uniformly in $k$ to a subspace $F_k$ of codimension $d$ that we 
will call the  slow space. We will need a more precise statement where 
$F_k$ is understood as a graph over a fixed splitting uniformly in $k$ 
(see definition \ref{definition:Graph}).  The reference splitting will be given by $X=U(k,N_*)\oplus V(k,N_*)$ for some $N_*$ chosen sufficiently large. An initial choice of $N_*$ is made in the following lemma and will be subsequently tightened in lemma \ref{lemma:expandingConeEstimate}, \ref{lemma:secondCrucialStep}, and finally in Assumption  \ref{assumption:tau_N}. It will be convenient to choose at each step of the proof  $N_*$ depending on  a parameter $\theta_* \in (0,1)$   as in  \eqref{equation:condition_N},  \eqref{equation:condition_N_bis} and \eqref{equation:condition_N_ter}. 

\begin{lemma}[\bf Existence of the slow space] \label{lemma:existenceSlowSpace}
Let $\theta_* \in (0,1)$ and $N_*$ satisfy
\begin{equation} \label{equation:condition_N}
\Dstar e^{-N_*\tau}  \leq \theta_*(1-\theta_*)^6 \frac{1-e^{-\tau}}{C_{\epsilon,d}^4}.
\end{equation}
Then for every $k\in\mathbb{Z}$, for every $n\geq N_*$, the following 5 items are satisfied.
\begin{enumerate}
\item \label{item1:existenceSlowSpace} $V(k,n) = \Graph(\Theta(k,n))$  
for some $\Theta(k,n) \in \mathcal{B}( V(k,N_*), U(k,N_*))$ 
\begin{gather*}
\delta(V(k,N_*), V(k,n)) \leq \| \Theta(k,n) \|  \leq \theta_*, \quad 
\delta(V(k,n), V(k,N_*)) \leq \theta_*.
\end{gather*}
\item \label{item2:existenceSlowSpace} $(\Theta(k,n))_{n\geq N_*}$ is  
a Cauchy sequence, for every $n\geq1$
\[
\| \Theta(k,n+1) -\Theta(k,n) \| \leq \theta_*e^{-(n-N_*)\tau}(1-e^{-\tau}).
\]
\item \label{item3:existenceSlowSpace} Let $ \Theta_k(N_*) := 
\lim_{n\to+\infty}\Theta(k,n)$ and  $F_k := \Graph(\Theta_k(N_*))$. Then
\begin{gather*}
\delta(V(k,N_*), F_k) \leq \| \Theta_k(N_*) \| \leq \theta_*, \quad   
\delta( F_k, V(k,N_*)) \leq \theta_*.
\end{gather*}
$F_k$ is called the slow space of index $d$; $F_k$ is independent of the choice of $N_*$.
\item \label{item4:existenceSlowSpace} $V(k,n)^\perp = 
\Graph(\Theta^\perp(k,n))$ for the bounded operator 
\begin{gather*}
\Theta^\perp(k,n) = -\pi(k,N_*)^*\Theta(k,n)^*\rho (k,N_*)^* \in 
\mathcal{B}( V(k,N_*)^\perp, U(k,N_*)^\perp),
\end{gather*}
where $\pi(k,n)$ is the projection onto $V(k,n)$ parallel to $U(k,n)$ and 
$\rho(k,n)$ is the inclusion operator $U(k,n) \hookrightarrow  X$. Moreover
\begin{gather*}
\Theta^\perp_k(N_*) := \lim_{n\to+\infty} \Theta^\perp(k,n) \quad \text{exists}, \\
F_k^\perp = \Graph(\Theta^\perp_k (N_*)), \quad \|\Theta^\perp(k,n)\| \leq 
\theta_*, \quad  \| \Theta_k^\perp(N_*) \| \leq \theta_*. 
\end{gather*} 
\item \label{item5:existenceSlowSpace}  
$\| (A(k,n) | U(k,N_*)^{-1} \|^{-1}/\sigma_d(k,n)$ is uniformly bounded from below,
\begin{itemize}
\item $X= U(k,N_*) \oplus F_k$, 
\item $\forall  u \in  U(k,N_*), \quad  \| A(k,n) u \|  \geq  
C_{\epsilon,d}^{-2} (1-\theta_*)^2 \sigma_d(k,n) \|u\|$,
\item $\gamma(U(k,N_*),F_k) \geq C_{\epsilon,d}^{-1} (1-\theta_*)^2$.
\end{itemize}
\end{enumerate}

\end{lemma}

\begin{proof}
In order to simplify the notations, fix $k$ and  denote
\[
V_n := V(k,n), \ V_* := V(k,N_*), \ \ U_* := U(k,N_*).
\]
We want to apply lemma \ref{lemma:ConvergenceGraph} for the initial 
splitting $X=U_* \oplus V_*$ where $V_*$ plays the role of $U_0$. An additional  complication comes from the fact that the minimal angle is not symmetric. We  
shall show  by induction for every $n\geq N_*$
\begin{itemize}
\item $\| \Theta_{n} - \Theta_{n-1} \| \leq \theta_{n-1} (1-\theta_*)$, \quad 
($\Theta_{N_*-1}=0$ by convention),
\item $V_n = \Graph(\Theta_n)$ for some   $\Theta_n \in\mathcal{B} (V_*,U_*)$ 
with $\| \Theta_{n} \| \leq \theta_* (1-\theta_*) \gamma(U_*,V_*)$,
\item $\delta(V_n,V_*) \leq \theta_* \gamma(U_*,V_*)$, 
\end{itemize}
where $\theta_n := \theta_* e^{-(n-N_*)\tau}(1-e^{-\tau}) \gamma(U_*,V_*) \leq \theta_*$.

Suppose that the above conditions are satisfied for the index $n$. We  first claim that the choice of $N_*$ implies
\[
\delta(V_{n+1},V_n) \leq \theta_n(1-\theta_n)(1-\theta_*)^2 \gamma(U_*,V_n) \leq \theta_n.
\]
To see this, on the one hand, from equation \eqref{equation:Kato4_4_5}, we have
\begin{align*}
\gamma(U_*,V_n) &\geq \frac{\gamma(U_*,V_*)-\delta(V_n,V_*)}{1+\delta(V_n,V_*)} \\
&\geq \frac{(1-\theta_*) \gamma(U_*,V_*)}{1+\theta_* \gamma(U_*V_*)} 
\geq (1-\theta_*)^2\gamma(U_*,V_*).
\end{align*}
On the other hand,  from the definition of $N_*$ we have 
\begin{align*}
C_{\epsilon,d}^2 \Dstar e^{-n\tau} &\leq \theta_* (1-\theta_*)^6 
e^{-(n-N_*)\tau}(1-e^{-\tau}) \gamma(U_*,V_*)^2, \\
&\leq \theta_n(1-\theta_*)^6 \gamma(U_*,V_*).
\end{align*}
Combining both estimates,  lemma \ref{lemma:RaghunathanEstimateI} and 
equation \eqref{equation:symmetryMaximalGap}, one obtains
\begin{gather*}
\delta(V_n,V_{n+1})  \leq C_{\epsilon,d}^2 \Dstar e^{-n\tau} \leq \theta_n  
(1-\theta_*)^4 \gamma(U_*,V_n) \leq \theta_n \leq \theta_* , \\
\delta(V_{n+1},V_n) \leq \frac{\theta_n(1-\theta_*)^4\gamma(U_*,V_n)}{1-\theta_*} 
\leq \theta_n(1-\theta_n) (1-\theta_*)^2 \gamma(U_*,V_n).
\end{gather*}
The claim is proved. We now show the three conditions  for the index $n+1$. From item \ref{item:{lemma:ConvergenceGraph}_4} 
of lemma \ref{lemma:ConvergenceGraph}, $V_{n+1} = \Graph(\Theta_{n+1})$ 
for some $\Theta_{n+1} \in \mathcal{B}(V_*,U_*)$ and
\begin{gather*}
\| \Theta_{n+1} - \Theta_n \| \leq 
\frac{\delta(V_{n+1},V_n)}{\gamma(U_*,V_n) - \delta(V_{n+1},V_n)} 
\ \frac{\gamma(U_*,V_*)}{\gamma(U_*,V_*) - \delta(V_n,V_*)} 
\leq \theta_n(1-\theta_*),  \\
\delta(V_*,V_{n+1}) \leq \| \Theta_{n+1} \| \leq 
\sum_{k=N_*}^{n}\theta_k(1-\theta_*) \leq \theta_* (1-\theta_*) \gamma(U_*,V_*), \\
\delta(V_{n+1}, V_*) \leq \frac{\delta(V_*,V_{n+1})}{1-\delta(V_*,V_{n+1})} \leq \theta_* \gamma(U_*,V_*). 
\end{gather*}
The induction is complete and the three first items are proved.  

The fact that 
$F_k$ is independent of the initial choice $N_*$ is proved in the following way. 
Let $w \in F_k$,  $w= v+\Theta_k(N_*)v$ for some $v \in V(k,N_*)$. Then
\begin{gather*}
w -[v + \Theta(k,n)v] =  [\Theta_k(N_*)v-\Theta(k,n)v], \\
\dist(w,V(k,n)) \leq \| \Theta_k(N_*) - \Theta(k,n) \| \, \|v\| \leq 
\frac{\| \Theta_k(N_*) - \Theta(k,n) \|}{\gamma(V_*,U_*)} \, \|w\|, \\
\delta(F_k,V(k,n)) \leq \frac{\| \Theta_k(N_*) - \Theta(k,n) \|}{\gamma(V_*,U_*)} 
\leq \theta_* e^{-(n-N_*)\tau} \frac{\gamma(U_*,V_*)}{\gamma(V_*,U_*)}.
\end{gather*}
Let $F'_k$ as in item \ref{item3:existenceSlowSpace} with another choice of 
$\theta'_*$ and $N'_*$. Using  the weak triangle inequality
\[
\delta(F_k,F'_k) \leq 2\delta(F_k,V(k,n)) + 2\delta(V(k,n),F'_k)
\]
and letting $n\to+\infty$, one obtains $\delta(F_k,F'_k)=0$ and $F_k=F'_k$.

Item \ref{item4:existenceSlowSpace} is a consequence of lemma \ref{lemma:GraphcoGraph}. Item \ref{item5:existenceSlowSpace} is a consequence of  item  \ref{item:aprroximatedSingularSpaces_2} 
of theorem \ref{theorem:aprroximatedSingularSpaces} and equation 
\eqref{equation:Kato4_4_5},
\begin{gather*}
\gamma(U_*,V_n)  \geq 
\frac{\gamma(U_*,V_*)-\delta(V_n,V_*)}{1+\delta(V_n, V_*)} 
\geq \gamma(U_*,V_*) \frac{1-\theta_*}{1+\theta_*} 
\geq \gamma(U_*,V_*) (1-\theta_*)^2,  \\
\gamma(U_*,F_k) \geq \gamma(U_*,V_*) (1-\theta_*)^2, \quad
\text{(by taking the limit $n\to+\infty$)}. 
\end{gather*}
Moreover  for every $u \in U_*$ such that $\|u\|=1$,
\begin{align*}
\| A(k,n) u \| &\geq \sup \{ \langle \tilde \phi | A(k,n) u \rangle : 
\tilde \phi \in \tilde V(k+n,n)^\perp, \ \| \tilde\phi\|=1 \} \\
& \geq \sup \{ \langle\phi | u \rangle : \phi \in  V_n^\perp, \ \|\phi\|=1 \} 
\inf \Big\{ \frac{\| A(k,n)^* \tilde \phi \|}{\|\tilde\phi\|} : \tilde\phi 
\in V(k+n,n)^\perp \Big\}  \\
&\geq \dist(u,V_n)\frac{\sigma_d(k,n)}{C_{\epsilon,d}} \geq 
\gamma(U_*,V_n)\frac{\sigma_d(k,n)}{C_{\epsilon,d}} \\
&\geq  \gamma(U_*,V_*) (1-\theta_*)^2 \frac{\sigma_d(k,n)}{C_{\epsilon,d}}. \qedhere
\end{align*}
\end{proof}

\begin{lemma}[\bf Equivariance of the slow space]
For every $k \in \mathbb{Z}$, 
\[
A_k F_k \subset F_{k+1}.
\]
\end{lemma}

\begin{proof}
Let $v\in V(k,n+1)$, and $\phi \in V(k+1,n)^\perp$. Then there 
exists $\tilde \phi \in \tilde V(k+n+1,n)^\perp$ such that 
$\phi = A(k+1,n)^*\tilde \phi$. On the one hand, item 
\ref{item:aprroximatedSingularSpaces_2} of theorem 
\ref{theorem:aprroximatedSingularSpaces} implies
\begin{align*}
\| \phi \| \geq \sigma_{d}(A(k+1,n)^* | \tilde V(k+n+1,n)^\perp) 
\| \tilde \phi \| \geq \frac{\sigma_d(k+1,n)}{C_{\epsilon,d}} \| \tilde \phi \|.
\end{align*}
On the other hand,  item \ref{item:aprroximatedSingularSpaces_2} also  shows
\begin{align*}
\langle \phi | A_k v \rangle &= \langle \tilde \phi | A(k,n+1) v \rangle 
\leq \| \tilde \phi \| \|A(k,n+1) | V(k,n+1)\| \, \|v\|,\\
&\leq C_{\epsilon,d}^2 \frac{\sigma_{d+1}(k,n+1)}{\sigma_d(k+1,n)} \| \phi \| \, \|v\| \\
&\leq C_{\epsilon,d}^2 \|A_k\| 
\frac{\| A_{k+1}\| \sigma_{d+1}(k+2,n-1)}{\sigma_d(k+1,n)} \| \phi \| \, \|v\|  \\
&\leq C_{\epsilon,d}^2 \|A_k\| \Dstar e^{-(n-1)\tau} \| \phi \| \, \|v\| .
\end{align*}
We have thus obtained for every $v\in V(k,n+1)$,
\begin{align*}
\dist(A_kv,V(k+1,n)) &= \sup \{ \langle \phi | A_k v \rangle : 
\phi \in V(k+1,n)^\perp, \ \|\phi\|=1\}, \\
&\leq C_{\epsilon,d}^2 \|A_k\| \Dstar e^{-(n-1)\tau} \|v\|.
\end{align*}
Let $\theta_*$ and  $N_*$ satisfy equation \eqref{equation:condition_N}.  
Assume  $n\geq N_*$. Let $v_* \in V(k,N_*)$ and 
$w_n := \Theta(k,n+1)v_*+v_*$. Then there exists $v'_n \in V(k+1,N_*)$ such that 
\[
w'_n := \Theta(k+1,n)v'_n +v'_n \ \ \text{satisfies} \ \  \| A_kw_n - w'_n \| \to 0.
\]
Since $w_n \to w:=\Theta_k(N_*)v_*+v_*$, the sequences 
$(A_kw_n)_n$, $(w'_n)_n$  and $(v'_n)_n$ are Cauchy sequences. 
We obtain therefore the convergence of $v'_n  \to v' \in V(k+1,N_*)$ 
and $A_k(\Theta_k(N_*)v_*+v_*) = \Theta_{k+1}(N_*)v'+v'$.
\end{proof}

We now consider the construction of the fast spaces $(E_k)_{k\in\mathbb{Z}}$ using the backward cocycle $(A_n)_{-\infty}^{n=k-1}$ and their approximate fast backward spaces $\tilde U(k,n)$. The following lemma is analogous to lemma \ref{lemma:RaghunathanEstimateI}.

\begin{lemma}[\bf  Raghunathan estimate II] \label{lemma:RaghunathanEstimateII}
For every $n\geq1$, $k\in\mathbb{Z}$,
\begin{gather}
\begin{split}
&\delta(\tilde U(k,n+1),\tilde U(k,n))  \leq  C_{\epsilon,d}^2 \Dstar e^{-n\tau}, \\
&\delta(\tilde U(k,n),\tilde U(k,n+1))  \leq  C_{\epsilon,d}^2 
\Dstar e^{-n\tau}/(1-C_{\epsilon,d}^2 \Dstar e^{-n\tau}).
\end{split}
\end{gather}
\end{lemma}

\begin{proof}
Let $\tilde u \in \tilde U(k,n+1)$ and $\tilde \phi \in \tilde U(k,n)^\perp$ of norm 1. 
On the one hand $\tilde u = A(k-n-1,n+1)u$ for some $u \in U(k-n-1,n+1)$ 
and item \ref{item:aprroximatedSingularSpaces_2} of theorem 
\ref{theorem:aprroximatedSingularSpaces} implies
\[
\| \tilde u \| \geq \sigma_d(k-n-1,n+1) \|u \| /C_{\epsilon,d}.
\]
On the other hand, item \ref{item:aprroximatedSingularSpaces_2} also  implies
\begin{align*}
\langle \tilde \phi | \tilde  u \rangle &= \langle \tilde \phi | A(k-n-1,n+1) u \rangle 
= \langle A(k-n-1,n+1)^*\tilde \phi | u \rangle \\
&\leq  \| A_{k-n-1} \| \| A(k-n,n)^* \tilde \phi \| \|u \| \\
&\leq \| A_{k-n-1} \| \sigma_{d+1}(k-n,n) C_{\epsilon,d} \| \tilde \phi \| \| u \| \\
&\leq C_{\epsilon,d}^2 \frac{\| A_{k-n-1} \|
\sigma_{d+1}(k-n,n)}{\sigma_d(k-n-1,n+1)} \| \tilde \phi \|  \|\tilde u \|.
\end{align*}
The second inequality is a consequence of equation \eqref{equation:symmetryMaximalGap}.
\end{proof}

The following lemma is analogous to lemma \ref{lemma:existenceSlowSpace}. 
We show that the sequence of subspaces $(\tilde U(k,n))_{n\geq1}$ is a 
Cauchy sequence  converging uniformly in $k$ to a subspace $E_k$ of 
dimension $d$. We see $E_k$  as a graph over $\tilde U(k,N_*)$ in the 
splitting $X= \tilde U(k,N_*) \oplus \tilde V(k,N_*)$ for some large 
$N_*$ defined in \eqref{equation:condition_N}.

\begin{lemma}[\bf Existence of the fast space] \label{lemma:existenceFastSpace}
Let  $\theta_* \in (0,1)$ and $N_*$ satisfy equation \eqref{equation:condition_N}. 
Then for every $k\in\mathbb{Z}$, for every $n\geq N_*$, the following 
4 items are satisfied.
\begin{enumerate}
\item \label{item:existenceFastSpace_1} $\tilde U(k,n) = 
\Graph(\tilde \Theta(k,n))$ for some  $\tilde\Theta(k,n) \in 
\mathcal{B}(\tilde U(k,N_*), \tilde V(k,N_*))$,
\begin{gather*}
\delta(\tilde U(k,N_*), \tilde U(k,n)) \leq \| \tilde\Theta(k,n) \| \leq 
\theta_*, \quad  \delta(\tilde U(k,n), \tilde U(k,N_*)) \leq \theta_*.
\end{gather*}
\item \label{item:existenceFastSpace_2} $(\tilde \Theta(k,n))_{n\geq N_*}$
is a Cauchy sequence, for every $n\geq1$
\[
\| \tilde \Theta(k,n+1) -  \tilde \Theta(k,n) \| \leq \theta_*e^{-(n-N_*)\tau}(1-e^{-\tau}), 
\]
\item \label{item:existenceFastSpace_3} Let $\tilde \Theta_k(N_*) := 
\lim_{n\to+\infty} \tilde \Theta(k,n)$ and $E_k := \Graph(\tilde \Theta_k(N_*))$.   Then
\begin{gather*}
\delta(\tilde U(k,N_*), E_k) \leq \| \tilde \Theta_k(N_*) \| 
\leq \theta_*, \quad \delta(E_k, \tilde U(k,N_*)) \leq \theta_*.
\end{gather*}
$E_k$ is called the fast  space of index $d$; $E_k$ is independent of the choice of $N_*$.
\item \label{item:existenceFastSpace_4} 
$\|(A(k-n,n)^*|\tilde V(k,N_*)^\perp)^{-1}\|^{-1}/\sigma_d(k-n,n)$ 
is  bounded from below,
 \begin{itemize}
\item $X = E_k \oplus \tilde V(k,N_*)$,
\item $\forall  \tilde\phi \in \tilde V(k,N_*)^\perp, \quad 
\| A(k-n,n)^*\tilde \phi \|  \geq C_{\epsilon,d}^{-2}  (1-\theta_*)^2  
\sigma_d(k-n,n) \|\tilde\phi\|$,
\item $\gamma(\tilde U(k,n), \tilde V(k,N_*)) \geq (1-\theta_*)^2 C_{\epsilon,d}^{-1}$.
\end{itemize}
\end{enumerate}
\end{lemma}

\begin{proof}
The proof of items 1 -- 3 is similar to the one in lemma 
\ref{lemma:existenceSlowSpace} by permuting the role of $U$ and $V$. 
For instance we also obtain by induction
\[
\delta(\tilde U(k,N_*), \tilde U(k,n)) \leq \theta_* \gamma(\tilde U(k,N_*),\tilde V(k,N_*)).
\] 
For the last item, we choose $\tilde \phi \in \tilde V(k,N_*)^\perp$, $\|\tilde\phi\|=1$, then using \eqref{equation:distanceToSubspaceEstimate},
\begin{align*}
\| A(k&-n,n)^*\tilde\phi\| \\
&\geq \sup \{ \langle \tilde\phi | A(k-n,n) u \rangle : u \in U(k-n,n), \ \|u\| = 1 \} \\
&\geq \sup \{ \langle \tilde\phi | \tilde u \rangle : \tilde u \in \tilde U(k,n),
\ \|\tilde u \|=1\} \inf \Big\{ \frac{\|A(k-n,n) u\|}{\|u\|} : u \in U(k-n,n) \Big\} \\
&\geq \dist(\tilde \phi , \tilde U(k,n)^\perp) 
\frac{\sigma_d(k-n,n)}{C_{\epsilon,d}} \geq
\gamma(\tilde V(k,N_*)^\perp,\tilde U(k,n)^\perp) \frac{\sigma_d(k-n,n)}{C_{\epsilon,d}},
\end{align*}
and by using equations \eqref{equation:Kato4_4_5} and 
\eqref{equation:Kato_4_4_2} one concludes
\begin{align*}
 \gamma(\tilde V(k,N_*)^\perp &,\tilde U(k,n)^\perp) =\gamma(\tilde U(k,n),\tilde V(k,N_*)) \\
&\geq \frac{\gamma(\tilde U(k,N_*), \tilde V(k,N_*))-
\delta(\tilde U(k,N_*),\tilde U(k,n))}{1+\delta(\tilde U(k,N_*),\tilde U(k,n))} \\
&\geq \frac{1-\theta_*}{1+\theta_*} \gamma(\tilde U(k,N_*), 
\tilde V(k,N_*)) \geq (1-\theta_*)^2 C_{\epsilon,d}^{-1}. \qedhere
\end{align*}
\end{proof}

\begin{lemma}[\bf Equivariance of the fast space]
For every $k \in \mathbb{Z}$, 
\[
A_k E_k=E_{k+1}.
\]
\end{lemma}

\begin{proof}
Let $\tilde u \in \tilde U(k,n)$ and $\tilde \phi  \in \tilde U(k+1,n+1)^\perp$. 
Then there exists $u \in U(k-n,n)$ such that $\tilde u = A(k-n,n)u$. On the one hand
\[
\| \tilde  u \| \geq \sigma_d(k-n,n) \|u \| /C_{\epsilon,d}.
\]
On the other hand
\begin{align*}
\langle \tilde \phi | A_k \tilde u \rangle &= \langle A(k-n,n+1)^* 
\tilde \phi | u \rangle \leq \| A(k-n,n+1)^*\tilde \phi \| \|u\| \\
&\leq C_{\epsilon,d} \sigma_{d+1}(k-n,n+1) \| \tilde\phi \| 
\|u\| \leq C_{\epsilon,d}^2 \frac{\sigma_{d+1}(k-n,n+1)}{\sigma_d(k-n,n)} 
\|\tilde\phi\| \,  \|\tilde u\| \\
&\leq  C_{\epsilon,d}^2 \|A_k\| \frac{\sigma_{d+1}(k-n,n-1) \|A_{k-1}\|}
{\sigma_d(k-n,n)} \|\tilde\phi\| \,  \|\tilde u\| \\
&\leq  C_{\epsilon,d}^2 \|A_k\| \Dstar e^{-(n-1)\tau} \|\tilde\phi\| \, \|\tilde u\|.
\end{align*}
We just have proved for every $\tilde u \in \tilde U(k,n)$,
\[
\dist(A_k \tilde u, \tilde U(k+1,n+1) \leq C_{\epsilon,d}^2 \|A_k\| 
\Dstar e^{-(n-1)\tau} \|  \tilde u \|.
\] 
Let $\theta_*,N_*$  as in equation \eqref{equation:condition_N}. Let 
$\tilde u_* \in \tilde U(k,N_*)$ and $w_n := \tilde u_* +
\tilde\Theta(k,n) \tilde u_*$. Then there exists $\tilde u_n' \in \tilde U(k+1,N_*)$ such that
\[
w'_n := \tilde u_n' + \tilde\Theta(k+1,n+1) \tilde u_n' 
\ \ \text{satisfies} \ \ \| A_k w_n - w'_n \| \to 0.
\]
Since $w_n \to \tilde u_* + \tilde\Theta_k(N_*) \tilde u_*$, $\tilde u'_n \to \tilde u'$, 
$w'_n \to w' = \tilde u' + 
\tilde\Theta_{k+1}(N_*) \tilde u'$. We have proved 
$A_k(\tilde u_* + \tilde\Theta_k(N_*) \tilde u_*) = 
\tilde u' + \tilde\Theta_{k+1}(N_*) \tilde u'$ and the equivariance of the fast space.
\end{proof}

\section{Proof of item \ref{item:main_1} of theorem \ref{theorem:main}}
\label{section:proofMainResult_1}

We present  the  proof of the bound from below (item \ref{item:main_1} 
of theorem \ref{theorem:main}) of  the angle between $E_k$ and 
$F_k$ uniformly in $k\in\mathbb{Z}$. We use for the first time the  
property (FI). Although there should exist a direct proof for any 
dimension $d$, we reduce our analysis to the case  $d=1$ by introducing 
the exterior product $\bigwedge^d X$. The  cocycle $A(k,n)$ admits 
a canonical extension to the exterior product that we denote 
\[
\widehat A(k,n) := {\textstyle \bigwedge}^d A(k,n). 
\]
The  approximate singular value decomposition obtained in theorem 
\ref{theorem:aprroximatedSingularSpaces} for the cocycle $A(k,n)$ can be 
extended to the cocycle $\widehat A(x,n)$ by applying theorem 
\ref{theorem:codimensionOneSingularValueDecomposition}  to each $A(k,n)$. 
We use definition  \ref{definition:extendedSplitting} for the notation 
$\widehat U$ and $\widecheck V$, for every subspace $U$ of dimension 
$d$ and $V$ of codimension $d$, respectively.  We obtain the following theorem.

\begin{theorem} \label{theorem:aprroximatedExteriorSingularSpaces}
Let $X$ be a Banach space, $d\geq1$, $\epsilon>0$, and 
$(A_k)_{k\in\mathbb{Z}}$ be a sequence of bounded operators. 
Let $X = U(k,n) \oplus V(k,n) = \tilde U(k,n) \oplus \tilde V(k,n)$ be 
the approximate singular value decomposition given in theorem  
\ref{theorem:aprroximatedSingularSpaces}  spanned respectively by the 
bases $(e_1,\ldots,e_d)$, $(\phi_1,\ldots,\phi_d)$, 
$(\tilde e_1,\ldots,\tilde e_d)$, $(\tilde \phi_1,\ldots,\tilde \phi_d)$. 
Then there exists a constant $\widehat K_d$ depending only on the 
Banach norm and $d$, such that, for every $k\in\mathbb{Z}$, 
$n\geq1$,  $\widehat C_{\epsilon,d}:=(1+\epsilon)\widehat K_d$, 

\begin{enumerate}
\item \label{item1:aprroximatedExteriorSingularSpaces} 
$\bigwedge^d X = \widehat U(k,n) \oplus \widecheck V(k,n)$,
\quad $\bigwedge^d X = \widehat{\tilde U}(k,n) \oplus \widecheck{\tilde V}(k,n)$,
\item \label{item2:aprroximatedExteriorSingularSpaces} 
$\widehat U(k,n) = \Vect({\textstyle  \bigwedge}_{i=1}^d e_i(k,n))$,  
\quad $ \widecheck V(k,n) = \Vect({\textstyle \bigwedge}^d_{i=1}\phi_i(k,n))^\Perp$,
\item \label{item3:aprroximatedExteriorSingularSpaces} 
$\widehat{\tilde U}(k,n) =  \Vect({\textstyle  \bigwedge}_{i=1}^d 
\tilde e_i(k,n))$, \quad $\widecheck{\tilde V}(k,n) = 
\Vect({\textstyle \bigwedge}^d_{i=1}\tilde \phi_i(k,n))^\Perp$,
\item $\dim(\widehat U(k,n)) = \dim(\widehat{\tilde U}(k,n)) = 1$,
\item \label{item4:aprroximatedExteriorSingularSpaces} 
$\widehat A(k,n )\widehat U(k,n) = \widehat{\tilde U}(k+n,n)$, 
\quad $\widehat A(k,n) \widecheck V(k,n) \subset \widecheck{\tilde V}(k+n,n)$,
\item \label{item5:aprroximatedExteriorSingularSpaces}  
$\widehat C_{\epsilon,d}^{-1} \ \prod_{i=1}^d \sigma_i(k,n) 
\leq \| \widehat A(k,n) | \widehat U(k,n) \| \leq 
\widehat C_{\epsilon,d}\prod_{i=1}^d \sigma_i(k,n)$,
\item \label{item6:aprroximatedExteriorSingularSpaces} 
$\widehat C_{\epsilon,d}^{-1} \ \prod_{i=1}^d \sigma_i(k,n)
\leq \| \widehat A(k,n)^* | \widecheck{\tilde V}(k+n,n)^\perp \| 
\leq \widehat C_{\epsilon,d} \ \prod_{i=1}^d \sigma_i(k,n)$,
\item \label{item7:aprroximatedExteriorSingularSpaces}  
$\| \widehat A(k,n) | \widecheck V(k,n) \| \leq \widehat C_{\epsilon,d} 
\ \sigma_1(k,n) \ \cdots\ \sigma_{d-1}(k,n)\sigma_{d+1}(k,n)$,
\item \label{item8:aprroximatedExteriorSingularSpaces}  $\gamma(\widehat U(k,n), 
\widecheck V(k,n)) \geq \widehat C_{\epsilon,d}^{-1}, \quad 
\gamma(\widecheck V(k,n), \widehat U(k,n)) \geq \widehat C_{\epsilon,d}^{-1}$.
\end{enumerate}
\end{theorem}
\noindent This theorem is a direct consequence of theorem 
\ref{theorem:codimensionOneSingularValueDecomposition}. We now recall 
some notations  introduced in   item \ref{item3:existenceSlowSpace} and 
\ref{item4:existenceSlowSpace}  of  lemma \ref{lemma:existenceSlowSpace}. We consider $E_k$ and $F_k$ as graphs over a fixed splitting $X= \tilde U(k,N_*) \oplus \tilde V(k,N_*)$ and $X= U(k,N_*) \oplus V(k,N_*)$ respectively.

\begin{notations} \label{notations:graphOperator}
Let $\theta_* \in (0,1)$ and $N_*$ satisfy equation  \eqref{equation:condition_N}. Then 
\begin{itemize}
\item  $E_k = \Graph(\tilde\Theta_k(N_*))$ for some $\tilde \Theta_k(N_*) : \tilde U(k,N_*) \to \tilde V(k,N_*)$, 
\item $F_k = \Graph( \Theta^\perp_k(N_*))^\Perp$ for some   $\Theta^\perp_k(N_*) :  V(k,N_*)^\perp \to  U(k,N_*)^\perp$,
\item $\widehat E_k = \Vect \big({\textstyle \bigwedge}_{i=1}^d
(\Id \oplus \tilde\Theta_k(N_*))\tilde e_i(k,N_*) \big)$, 
\item $\widecheck F_k := \Vect \big({\textstyle \bigwedge}^d_{i=1} 
(\Id \oplus \Theta^\perp_k(N_*) \phi_i(k,N_*) \big)^\Perp$,
\item $\widecheck F_k = \Graph(\widehat\Theta_k(N_*))$ for some  $\widehat\Theta_k(N_*) : \widecheck V(k,N_*) \to  \widehat U(k,N_*)$,
\item $\| \tilde \Theta_k(N_*)\| \leq \theta_*$, \ \ 
 $\| \Theta^\perp_k(N_*) \| \leq \theta_*$, \ \  $\| \widehat \Theta_k(N_*) \| \leq  C_{\epsilon,d}^{2d} K_d \theta_* (1+\theta_*)^{d-1}$,
(using lemma \ref{lemma:normExtendedGraph} for some constant $K_d = \bar\Delta_d(X)^d$  given by 
\eqref{equation:simplifiedVolumicDistortion}).
\end{itemize}
\end{notations}

The strategy of the proof is based on two steps. In the first step we show that, 
for some $N_*$ large enough,
\[
\forall k \in\mathbb{Z}, \quad \gamma(\widehat A(k-N_*,N_*)
\widehat U(k-N_*,N_*), \widecheck F_k) \geq c(N_*),
\]
with a constant that depends on $N_*$ (and goes to zero as $N_*\to+\infty$). 
This estimate may be 
considered as a bootstrap argument; this is the only place where property 
(FI) is used.  

In the second part, we  analyze the special backward  cocycle 
associated to the sequence of operators 
$(\widehat A(k-nN_*,N_*))_{n=1}^{+\infty}$. We improve the 
previous estimate and show that actually 
\[
\forall n\geq1, \ \forall k \in\mathbb{Z}, \quad  
\gamma(\widehat A(k-nN_*,nN_*)\widehat U(k-nN_*,nN_*), \widecheck F_k) 
\geq \text{constant}.
\]
The proof is complicated by the fact that we are in a Banach space and look for 
an explicit lower bound. The proof is also new in the finite dimensional setting. 
We conclude the proof by observing  
\[
\widehat A(k-nN_*,nN_*)\widehat U(k-nN_*,nN_*) = 
\widehat {\tilde U}(k,nN_*) \to \widehat E_k.
\]
We obtain a uniform bound from below of $\gamma(\widehat E_k,\widehat F_k)$ 
and therefore a uniform bound from below  of $\gamma(E_k,F_k)$ by using lemma 
\ref{lemma:comparisonAngleExteriorProduct}.

We show in the following lemma  that the smallest expansion of 
$\widehat A(k,n)$  on  $\widehat U(k,m)$  is bounded from below by 
$\prod_{i=1}^d\sigma_i(k,n)$ uniformly in 
$m,n$ large enough,
\begin{equation} \label{equation:SVGexteriorProductLower}
\forall k\in\mathbb{Z}, \ \forall m,n \geq N_*, \quad  
\| \widehat A(k,n) | \widehat U(k,m) \| \geq \text{constant} 
\  \big[ \prod_{i=1}^d\sigma_i(k,n) \big].
\end{equation}

We now choose $N_*$ satisfying a more restrictive condition than the one in \eqref{equation:condition_N}.

\begin{lemma} \label{lemma:expandingConeEstimate}
Let $\theta_* \in(0,1)$ and $N_*$ satisfy
\begin{equation} \label{equation:condition_N_bis}
\Dstar e^{-N_*\tau}  \leq \theta_*(1-\theta_*)^7 \frac{1-e^{-\tau}}{C_{\epsilon,d}^5}.
\end{equation}
Then for every $n, m  \geq N_*$ and $k\in\mathbb{Z}$,
\[
\forall u \in \hat U(k,m), \quad \|\hat A(k,n) u \| \geq C_{\epsilon,d}^{-4d} 
K_d^{-1} (1-\theta_*)^d \Big( \prod_{i=1}^d \sigma_i(k,n) \Big) \|u\|,
\]
where $K_d := \bar\Delta_d(X)^{3d}$.
\end{lemma}

\begin{proof}
{\it Part 1.} We prove in both cases, $n\geq m$ and $m \geq n$, that there 
exists an operator $\Theta^\perp : V(k,m)^\perp \to U(k,m)^\perp$ such 
that $V(k,n)^\perp = \Graph(\Theta^\perp)$ and 
$\| \Theta^\perp \| \leq \theta_*$. 

For $n \geq m$ the existence of 
$\Theta^\perp$ is a consequence of item \ref{item4:existenceSlowSpace} 
of lemma \ref{lemma:existenceSlowSpace} taking $N_*=m$.

For $m \geq  n$, let 
$\theta' := \theta_*(1-\theta_*)/C_{\epsilon,d}$, then
\begin{align*} 
\Dstar e^{-n\tau} &\leq \Dstar e^{-N_*\tau}  \leq \theta'(1-\theta')^6 
\frac{1-e^{-\tau}}{C_{\epsilon,d}^4}, \\
\delta(V(k,m),V(k,n)) &\leq \theta' \leq \theta_*(1-\theta_*) \gamma(V(k,m),U(k,m)).
\end{align*}
In particular,  from item 
\ref{item:ConvergenceGraph_2} of lemma \ref{lemma:ConvergenceGraph},
\begin{gather*}
\delta(V(k,m),V(k,n)) < \gamma(V(k,m),U(k,m)), \\
\delta(V(k,n)^\perp,V(k,m)^\perp) < \gamma(U(k,m)^\perp,V(k,m)^\perp), \\ 
V(k,n)^\perp = \Graph(\Theta^\perp), \quad\text{for some} \quad  \Theta^\perp : 
V(k,m)^\perp \to U(k,m)^\perp, \\
\| \Theta^\perp \| \leq \frac{\delta(V(k,n)^\perp,V(k,m)^\perp)}
{\gamma(U(k,m)^\perp,V(k,m)^\perp)-\delta(V(k,n)^\perp,V(k,m)^\perp)} \leq \theta_*.
\end{gather*}

{\it Part 2.} We now prove the relative rate of expansion of $\widehat A(k,n)$. 
From lemma \ref{lemma:boundFromBelow}, one obtains with $K'_d = \bar\Delta_d(X)^{2d}$,
\[
\det \big( \big[ \langle \phi_i(k,n) | e_j(k,m) \rangle \big]_{ij} \big) 
\geq (K'_d)^{-1} C_{\epsilon,d}^{-2d} (1-\theta_*)^d.
\]
As $A^*(k,n) \tilde \phi_i(k+n,n) = \sigma_i(k,n) \phi_i(k,n)$, using 
equations \eqref{equation:definitionDualityWedgeProduct} and
\eqref{equation:normDualityWedgeProduct}, one obtains
\begin{align*}
\det \big( \big[ \langle \phi_i(k,n) | e_j(k,m) \rangle \big]_{ij} \big) &= 
\frac{\det \big( \big[ \langle \tilde \phi_i(k+n,n) | 
A(k,n) e_j(k,m) \rangle \big]_{ij} \big)}{ \prod_{i=1}^d \sigma_i(k,n) }, \\
&\leq \Sigma_d(X) \frac{ \| {\textstyle \bigwedge}_{i=1}^d 
\tilde \phi_i(k+n,n) \| \ \| \hat A(k,n) {\textstyle \bigwedge}_{i=1}^d e_i(k,m) \| }
{  \prod_{i=1}^d \sigma_i(k,n)  }.
\end{align*}
From proposition \ref{proposition:MultiplicativityJacobian}, we have 
$\Sigma_d(X) \leq \bar \Delta_d(X)^d$. From the definition of the 
projective norm \eqref{equation:ProjectiveNorm}, we have 
\[
\| {\textstyle \bigwedge}_{i=1}^d \tilde\phi(k+n,n) \|  \leq C_{\epsilon,d}^d
 \ \ \text{and} \  \ \| {\textstyle \bigwedge}_{i=1}^d e_j(k,m) \| 
 \leq C_{\epsilon,d}^d. \qedhere
\]
\end{proof}

The next lemma gives a lower bound of the angle between the approximate 
fast space $\widehat W_k := \widehat A(k-N_*,N_*) \widehat U(k-N_*,m)$ 
and the  slow space $\widecheck F_k$ for  $m \geq N_*$. 
This estimate is non trivial as $\widehat W_k$ is defined using the operators 
$(A_{k-n})_{n\geq1}$ and $\widecheck F_k$ is defined using the operators 
$(A_{k+n})_{n\geq0}$. Property (FI) forces the two spaces to be complementary. It is the only place where  (FI) is used.

\begin{lemma}[\bf First crucial step] \label{lemma:firstCrucialStep}
Let $\theta_* \in(0,1)$, $N_*$ satisfy equation
\eqref{equation:condition_N_bis},  $k\in\mathbb{Z}$, and 
$m \geq N_*$.  Denote $\widehat W_k := \widehat A(k-N_*,N_*) 
\widehat U(k-N_*,m)$. Then 
\[
\gamma(\widehat W_k, \widecheck F_k) \geq \widehat C_{\epsilon,d}^{-3} 
C_{\epsilon,d}^{-4d} K_d^{-1} (1-\theta_*)^d \Gstar ^{-1} e^{-N_*\mu},
\]
where $K_d :=  \bar\Delta_d(X)^{3d}$.
\end{lemma}

\begin{proof}
As $\widecheck V(k,n) \to \widecheck F_k$ in the co-Grassmannian topology, 
it is enough to bound from below $\gamma(\widehat W_k, \widecheck V(k,n))$ 
for large  $n\geq m$. We first  show that $\widehat W_k $ is the graph  of some operator 
$\widehat \Gamma(k,n) : \widehat U(k,n) \to \widecheck V(k,n)$. 
We then give an upper bound for $\| \Id \oplus \widehat \Gamma(k,n) \|$; 
or equivalently a lower bound for the angle 
$\gamma(\widehat W_k, \widecheck V(k,n))$.  Let 
\[
w \in \widehat W_k, \ \ w= w'+w'', \ \ w' \in \widehat U(k,n) 
\ \ \text{and} \ \ w'' \in \widecheck V(k,n).
\]
On the one hand $w = \widehat A(k-N_*,N_*)u$ for some 
$u\in \widehat U(k-N_*,m)$. Then using lemma \ref{lemma:expandingConeEstimate} 
with $K_d = \bar\Delta_d(X)^{3d}$ and item 
\ref{item5:aprroximatedExteriorSingularSpaces} of theorem 
\ref{theorem:aprroximatedExteriorSingularSpaces}, one gets
\begin{align*}
\| \widehat A(k,n) w \| &= \| \widehat A(k-N_*,N_*+n)u \| \\
&\geq C_{\epsilon,d}^{-4d} K_d^{-1} (1-\theta_*)^d 
\textstyle{\prod_{i=1}^d} \sigma_i(k-N_*,N_*+n) \ \|u\|, \\
\|w\| &\leq \widehat C_{\epsilon,d} \textstyle{\prod_{i=1}^d} \sigma_i(k-N_*,N_*) \| u \|.
\end{align*}
Thus
\begin{gather*}
\| \widehat A(k,n) w \| \geq  \widehat C_{\epsilon,d}^{-1} C_{\epsilon,d}^{-4d} 
K_d^{-1} (1-\theta_*)^d  
\frac{\textstyle{\prod_{i=1}^d} \sigma_i(k-N_*,N_*+n) }{ \textstyle{\prod_{i=1}^d}
\sigma_i(k-N_*,N_*)} \| w \|.
\end{gather*}
On the other hand using  items \ref{item5:aprroximatedExteriorSingularSpaces} 
and \ref{item7:aprroximatedExteriorSingularSpaces} of theorem 
\ref{theorem:aprroximatedExteriorSingularSpaces},
\begin{align*}
\| \widehat A(k,n)w' \| &\leq \widehat C_{\epsilon,d} 
\big[\textstyle{\prod_{i=1}^d} \sigma_i(k,n) \big] \| w' \|, \\
\| \widehat A(k,n)w'' \| &\leq \widehat C_{\epsilon,d} 
\big[\textstyle{\prod_{i=1}^{d-1}} \sigma_i(k,n) \big] \sigma_{d+1}(k,n) \| w'' \|, \\
\| \widehat A(k,n) w \| &\leq \widehat C_{\epsilon,d}
\big[\textstyle{\prod_{i=1}^d} \sigma_i(k,n) \big] \Big[ \| w' \| + 
\frac{\sigma_{d+1}(k,n)}{\sigma_d(k,n)} \| w'' \| \Big].
\end{align*}
Property (FI)  implies
\[
\frac{\textstyle{\prod_{i=1}^d} \sigma_i(k-N_*,N_*+n) }{ \textstyle{\prod_{i=1}^d}
\sigma_i(k-N_*,N_*) \textstyle{\prod_{i=1}^d} \sigma_i(k,n)} \geq \Gstar ^{-1} e^{-N_*\mu}.
\]
Combining the two estimates of $\| \widehat A(k,n) w \|$ and using property (SVG), one obtains,
\begin{multline*}
\| (\Id \oplus \widehat \Gamma(k,n))w' \| = \|w\|  \leq \\ 
\widehat C_{\epsilon,d}^2 C_{\epsilon,d}^{4d} K_d 
(1-\theta_*)^{-d}\Gstar  e^{N_*\mu} \big[1+\Dstar e^{-n\tau}
\| \widehat \Gamma(k,n) \| \big] \|w'\|.
\end{multline*}
In particular $\| \widehat \Gamma(k,n) \|$ is uniformly bounded 
from above. Using lemma \ref{lemma:comparisonAngleGraph}  
and item \ref{item8:aprroximatedExteriorSingularSpaces} of theorem 
\ref{theorem:aprroximatedExteriorSingularSpaces}
\begin{align*}
\gamma(\widehat W, \widecheck V(k,n)) &\geq 
\frac{\gamma(\widehat U(k,n), \widecheck V(k,n))}{\| \Id \oplus \widehat \Gamma(k,n) \|} 
\geq \frac{\widehat C_{\epsilon,d}^{-1}}{\| \Id \oplus \widehat \Gamma(k,n) \|} \\
&\geq \widehat C_{\epsilon,d}^{-3} C_{\epsilon,d}^{-4d} K_d^{-1}
(1-\theta_*)^{d}\Gstar ^{-1} e^{-N_*\mu} \big[1+\Dstar e^{-n\tau}\| 
\widehat \Gamma(k,n) \| \big]^{-1}.
\end{align*}
We conclude by letting $n\to+\infty$.
\end{proof}

Similarly to lemma \ref{lemma:expandingConeEstimate}, we show that the 
largest expansion of  $\widehat A(k,n)$ restricted to  $\widecheck F_k$ 
is bounded from above by $[\prod_{i=1}^d \sigma_i(k,n)]e^{-n\tau}$ uniformly  for $n $ large enough,
\begin{equation} \label{equation:SVGexteriorProductUpper}
\forall k \in \mathbb{Z}, \ \forall n \geq N_*,  \quad \|\widehat A(k,n) | 
\widecheck F_k \| \leq \text{constant} \ \Big( \prod_{i=1}^d 
\sigma_i(k,n) \Big)   e^{-n\tau}.
\end{equation}
Equation \eqref{equation:SVGexteriorProductUpper} together with equation 
\eqref{equation:SVGexteriorProductLower} show that the  cocycle 
$\widehat A(k,n)$ satisfies property (SVG) at index 1. Estimate 
\eqref{equation:SVGexteriorProductUpper} is the main reason to 
introduce the exterior product. The simplest proof based on the original cocycle seems to require a comparison between the two ratios  $\sigma_d(k,n)/\sigma_1(k,n)$ 
and  $\sigma_{d+1}(k,n)/\sigma_d(k,n)$.

\begin{lemma} \label{lemma:contractingConeEstimate}
Let $\theta_* \in(0,1)$ and $N_*$ satisfy equation 
\eqref{equation:condition_N_bis}. Then for every $n \geq N_*$ and $k \in\mathbb{Z}$, 
\begin{gather*}
\| \widehat A(k,n) | \widecheck F_k \| \leq 2 \widehat C_{\epsilon,d}^2 
C_{\epsilon,d}^{2d} K_d \theta_* (1+\theta_*)^{d-1} 
\Big( \prod_{i=1}^d \sigma_i(k,n) \Big)  e^{-(n-N_*)\tau} ,
\end{gather*}
where $K_d = \bar\Delta_d(X)^d$.
\end{lemma}

\begin{proof}
Let  $F_k = \Graph(\Theta_k^\perp(n))^\Perp$ and 
$\widecheck F_k = \Graph(\widehat \Theta_k(n))$ as in notations 
\ref{notations:graphOperator}. We first notice
\[
\Dstar e^{-n\tau} \leq e^{-(n-N_*)\tau} \Dstar e^{-N_*\tau} 
\leq \theta' (1-\theta')^6 \frac{1-e^{-\tau}}{C_{\epsilon,d}^4}
\]
with $\theta' := \theta_* e^{-(n-N_*)\tau}$. Substituting $\theta'$ for 
$\theta_*$ and $n$ for $N_*$  in item \ref{item4:existenceSlowSpace} 
of lemma \ref{lemma:existenceSlowSpace}, one obtains 
$\| \Theta_k^\perp(n) \| \leq \theta'$. Then lemma 
\ref{lemma:normExtendedGraph} and proposition 
\ref{proposition:MultiplicativityJacobian} imply
\[
\| \widehat \Theta_k(n) \| \leq C_{\epsilon,d}^{2d} K_d\theta' (1+\theta_*)^{d-1}.
\]
 Let $w \in \widecheck F_k$, $w = w' + w''$,  $w'' \in \widecheck V(k,n)$ and  $w' = \widehat\Theta_k(n) w'' \in \widehat U(k,n)$. Then
\begin{align*}
\| w'' \| &\leq \| \pi_{\widecheck V(k,n) | \widehat U(k,n)} \| \|w\| \leq 
\widehat C_{\epsilon,d} \|w\|, \\
\| \widehat A(k,n) w' \| &\leq \widehat C_{\epsilon,d} \big[ 
\textstyle{\prod_{i=1}^d} \sigma_i(k,n) \big] \| \widehat\Theta_k(n) \| \|w''\|, \\
\| \widehat A(k,n) w'' \| &\leq  \widehat C_{\epsilon,d} \ \sigma_1(k,n) 
\ \cdots\ \sigma_{d-1}(k,n)\sigma_{d+1}(k,n) \| w'' \|, \\
\| \widehat A(k,n) w \| &\leq \widehat C_{\epsilon,d}^2 \big[ \textstyle{\prod_{i=1}^d} 
\sigma_i(k,n) \big] \Big[ \| \widehat\Theta_k(n) \| + 
\frac{\sigma_{d+1}(k,n)}{\sigma_d(k,n)} \Big] \|w\|. 
\end{align*}
We conclude using property (SVG),  
\[
\frac{\sigma_{d+1}(k,n)}{\sigma_d(k,n)} 
\leq \Dstar e^{-n\tau} \leq \theta' \leq C_{\epsilon,d}^{2d} K_d\theta' (1+\theta_*)^{d-1}. \qedhere
\]
\end{proof}

We now change notation and rewrite the cocycle 
$(\widehat A(k-nN_*,N_*))_{n=1}^{+\infty}$ as block matrices along the following splitting. Notice the small circumflex for the new notation. Define
\begin{itemize}
\item $\hat A_{-n} := \widehat A(k-nN_*,N_*)$, \ $\forall n\geq1$,
\item $\hat U_{-n} := \widehat U(k-nN_*,nN_*)$, \ $\hat V_{-n} := 
\widecheck V(k-nN_*,nN_*)$, \  $\forall n\geq1$, 
\item $\hat U_0 := \widehat{\tilde U}(k,N_*)$, \ $\hat V_0 := \widecheck{\tilde V}(k,N_*)$,
\item $\hat E_{-n} := \widehat E_{k-nN_*}$, \  $\hat F_{-n} := \widehat F_{k-nN_*}$, \ $\forall n\geq0$,
\item $\bigwedge^d X = \hat U_{-n} \oplus \hat F_{-n}$,\  $\forall n\geq0$.
\end{itemize}
Notice that the first crucial step, lemma \ref{lemma:firstCrucialStep}, 
implies that $\hat U_0=\hat A_{-1}\hat U_{-1}$ and $\hat F_0$ are indeed two complementary spaces. 
We consider the following block splitting
\begin{itemize}
\item $\hat p_{-n}$ the projector onto $\hat U_{-n}$ parallel to $\hat F_{-n}$, \ $\forall n\geq0$,
\item $\hat q_{-n}$ the projector onto $\hat F_{-n}$  parallel to $\hat U_{-n}$, \ $\forall n\geq0$,
\item $\hat A_{-n} := \begin{bmatrix} \hat a_{-n} & 0 \\ \hat c_{-n} & 
\hat d_{-n} \end{bmatrix}$, \ $\forall n\geq1$
\item $\hat a_{-n} = p_{-(n-1)} \circ (\hat A_{-n} | \hat U_{-n}) : \hat U_{-n} \to \hat U_{-(n-1)}$,
\item $\hat c_{-n} = q_{-(n-1)} \circ (\hat A_{-n} | \hat U_{-n})  :  \hat U_{-n} \to \hat F_{-(n-1)} $,
\item $\hat d_{-n} = (\hat A_{-n} | \hat F_{-n}) :  \hat F_{-n} \to \hat F_{-(n-1)} $.
\end{itemize}
By the equivariance of the slow space $\hat A_{-n} \hat F_{-n} \subset \hat F_{-(n-1)}$, we obtain
\begin{itemize}
\item $\hat A_{-n}^n := \hat A_{-1} \hat A_{-2} \ \cdots  \ \hat A_{-n}  = \widehat A(k-nN_*,nN_*)$,
\item $\hat a_{-n}^n := \hat a_{-1} \hat a_{-2}  \cdots \ \hat a_{-n} = 
\hat p_0 \circ (\widehat A(k-nN_*,nN_*) | \widehat U(k-nN_*,nN_*))$,
\item $\hat d_{-n}^n :=  \hat d_{-1} \hat d_{-2}  \cdots \ \hat d_{-n} = 
(\widehat A(k-nN_*,nN_*) | \widehat F_{k-nN_*})$.
\end{itemize}
Lemma \ref{lemma:firstCrucialStep} implies that $\hat A_{-n} \hat U_{-n}  
\ \text{and} \ \ \hat F_{-(n-1)}$
are complementary. In particular $\hat a_{-n} : \hat U_{-n} \to \hat U_{-(n-1)}$ 
is bijective. Define for $n\geq1$,
\begin{itemize}
\item $\hat A_{-n-1} \hat U_{-n-1}\ = \Graph(\hat \Gamma_{-n})$ for 
some operator $\hat \Gamma_{-n} : \hat U_{-n} \to \hat F_{-n}$, by 
convention, $\hat \Gamma_0 := 0$,
\item $\hat A_{-n}^n \hat U_{-n} = \Graph( \hat \Xi_{0}^n )$ for some 
operator $\hat\Xi_0^n : \hat U_0 \to \hat F_0$. Notice that the choice of 
$\hat U_0$ implies  $\hat\Xi_0^1 = 0$.
\end{itemize}

\begin{lemma} \label{lemma:normProjectorInductionStep}
Let $\theta_* \in(0,1)$ and $N_*$ satisfy equation 
\eqref{equation:condition_N_bis}. Then
\[
\forall n\geq1, \quad \| \hat q_{-n} \| \leq \widehat C_{\epsilon,d} 
C_{\epsilon,d}^{2d} K_d(1+\theta_*)^d,
\]
where $K_d = \bar\Delta_d(X)^d$.
\end{lemma}

\begin{proof}
From notations \ref{notations:graphOperator} one obtains 
$\hat F_{-n} = \Graph(\hat\Theta_{-n})$ for some operator 
$\hat\Theta_{-n} := \widehat\Theta_{k-nN_*}(nN_*) : \hat V_{-n} 
\to \hat U_{-n}$. Moreover
\begin{align*}
\hat q_{-n} &= (\Id \oplus \hat \Theta_{-n}) \circ \pi_{\hat V_{-n} | \hat U_{-n} }, \\
\| \hat\Theta_{-n} \| &\leq C_{\epsilon,d}^{2d} K_d \theta_* (1+\theta_*)^{d-1}, \\
\| \hat q_{-n} \| &\leq \widehat C_{\epsilon,d} (1+ \| \hat\Theta_{-n} \|) 
\leq \hat C_{\epsilon,d} C_{\epsilon,d}^{2d} K_d  (1+\theta_*)^d. \qedhere
\end{align*}
\end{proof}

\begin{lemma} \label{lemma:NormGammaEstimate}
Let $\theta_* \in(0,1)$ and $N_*$ satisfy equation \eqref{equation:condition_N_bis}. Then
\[
\forall n\geq1, \quad \| \hat\Gamma_{-n} \| \leq \widehat C_{\epsilon,d}^4 
C_{\epsilon,d}^{6d} K_d (1-\theta_*)^{-2d} \Gstar  e^{N_*\mu},
\]
where $K_d :=  \bar\Delta_d(X)^{4d}$.
\end{lemma}

\begin{proof}
Since $\hat\Gamma_{-n} = \hat q_{-n} (\Id \oplus \hat\Gamma_{-n})$, 
we obtain using lemmas \ref{lemma:comparisonAngleGraph},
\ref{lemma:normProjectorInductionStep} and  \ref{lemma:firstCrucialStep}
\begin{align*}
\| \hat\Gamma_{-n} \| &\leq \frac{\| \hat q_{-n} \|}{\gamma(\hat A_{-n-1} 
\hat U_{-n-1}, \hat F_{-n} )} \leq \| \hat q_{-n} \| 
\widehat C_{\epsilon,d}^3 C_{\epsilon,d}^{4d} K'_d (1-\theta_*)^{-d} \Gstar  e^{N_*\mu},
\end{align*}
with $K'_d = \bar\Delta_d(X)^{3d}$.
\end{proof}

We now show that the minimal gap between $\hat A_{-n}^n \hat U_{-n}$ and 
$\hat F_{0}$ is  bounded from below uniformly in $n$. Since 
$\hat A_{-n}^n \hat U_{-n}= \Graph(\hat\Xi_0^n)$ for some  
$\hat\Xi_0^n : \hat U_0 \to \hat F_0$, it is enough to bound from above 
$\|\Id \oplus \hat\Xi_0^n\|$. We show how to estimate  $\|\Id \oplus \hat\Xi_0^{n+1}\|$ 
in terms of $\|\Id \oplus \hat\Xi_0^n\|$. Since $\hat A_{-n}^n \hat U_{-n} = \widehat{\tilde U}(k,nN^*) \to \hat E_0$, 
we obtain a bound from below of $\gamma(\hat E_0,\hat F_0)$.

\begin{lemma}[\bf Second crucial step] \label{lemma:secondCrucialStep}
Let $\theta_* \in(0,1)$ and $N_*$ satisfy
\begin{equation} \label{equation:condition_N_ter}
\Dstar e^{-N_*\tau}  \leq \frac{\theta_*(1-\theta_*)^{3d-1}}
{2 \widehat C_{\epsilon,d}^7 C_{\epsilon,d}^{8d} K_d \Gstar } 
(1-\theta_*)^7 \frac{1-e^{-\tau}}{C_{\epsilon,d}^5},
\end{equation}
with $K_d := \bar\Delta_d(X)^{5d}$. Then for every $n\geq1$,
\[
\gamma(\hat A_{-n}^n \hat U_{-n},\hat F_0) \geq  \frac{ (1-\theta_*)^d 
\Gstar ^{-1}}{\widehat C_{\epsilon,d}^{3} C_{\epsilon,d}^{4d} K_d } 
e^{-N_*\mu} \prod_{k=0}^{n-2} \Big[ 1+ e^{N_*\mu} e^{-kN_*\tau} \Big]^{-1}.
\]
\end{lemma}

\begin{proof}
Define 
\[
\theta' := \frac{\theta_*(1-\theta_*)^{3d-1}}{2 \widehat C_{\epsilon,d}^7 
C_{\epsilon,d}^{8d} K_d \Gstar }.
\] 
 Notice that $N_*$ satisfies equation \eqref{equation:condition_N_bis} with 
 $\theta'$ instead of $\theta_*$ 
\[
\Dstar e^{-N_*\tau} \leq  \theta' (1-\theta')^7  \frac{1-e^{-\tau}}{C_{\epsilon,d}^5},
\]

{\it Part 1.}  We estimate the norms $\| (\hat a_{-n}^n)^{-1} \|$ and 
$\| \hat d_{-n}^n \|$. On the one hand, using item 
\ref{item5:aprroximatedExteriorSingularSpaces} of theorem 
\ref{theorem:aprroximatedExteriorSingularSpaces}, one gets
\begin{align*}
(\hat a_{-n}^n)^{-1} &= (\hat A_{-n}^n | \hat U_{-n})^{-1} 
\circ (\Id \oplus \hat\Xi_0^n), \\
\| (\hat a_{-n}^n)^{-1} \| &\leq \widehat C_{\epsilon,d} 
\big[ \prod_{i=1}^d \sigma_i(k-nN_*,nN_*) \big]^{-1} \|\Id \oplus \hat\Xi_0^n\|.
\end{align*}
On the other hand, using lemma \ref{lemma:contractingConeEstimate}, one gets
\begin{gather*}
\| \hat d_{-n}^n\| \leq 2\widehat C_{\epsilon,d}^2 C_{\epsilon,d}^{2d} 
K'_d \theta' (1+\theta')^{d-1}  \big[ \prod_{i=1}^d \sigma_i(k-nN_*,nN_*) \big] 
e^{-(n-1)N_*\tau},
\end{gather*}
with $K'_d = \bar\Delta_d(X)^d$.

{\it Part 2.} We bound from above $\|\Id \oplus \hat\Xi_0^{n+1}\|$ 
in terms of $\|\Id \oplus \hat\Xi_0^n\|$. Notice first that 
$\hat \Gamma_{-n} = \hat c_{-n-1} (\hat a_{-n-1})^{-1}$. Moreover
\begin{gather*}
\hat A_{-n-1}^{n+1} = 
\begin{bmatrix} \hat a_{-n-1}^{n+1} & 0 \\ \hat  c_{-n-1}^{n+1} &  
\hat d_{-n-1}^{n+1} \end{bmatrix} = \begin{bmatrix} \hat  a_{-n}^n & 0 \\
\hat c_{-n}^n &  \hat d_{-n}^n \end{bmatrix}
\begin{bmatrix} \hat  a_{-n-1} & 0 \\ \hat  c_{-n-1} &  \hat d_{-n-1} \end{bmatrix}, \\
\hat c_{-n-1}^{n+1} = \hat c_{-n}^n \hat a_{-n-1} + \hat d_{-n}^n \hat c_{-n-1}, \\
\hat c_{-n-1}^{n+1} (\hat a_{-n-1}^{n+1})^{-1} = \hat c_{-n}^{n} 
(\hat a_{-n}^{n})^{-1} + \hat d_{-n}^n \hat c_{-n-1}(\hat a_{-n-1})^{-1} 
(\hat a_{-n}^{n})^{-1}.
\end{gather*}
Since $\hat\Xi_0^n =  \hat c_{-n}^{n} (\hat a_{-n}^{n})^{-1}$, we obtain 
$(\Id \oplus \hat\Xi_0^{n+1}) = (\Id \oplus \hat\Xi_0^{n}) + 
\hat d_{-n}^n  \hat \Gamma_{-n}  (\hat a_{-n}^n)^{-1} $,
\[
\|\Id \oplus \hat\Xi_0^{n+1}\| \leq \|\Id \oplus \hat\Xi_0^{n}\| 
\Big( 1+ \frac{\| \hat d_{-n}^n \| \| \hat \Gamma_{-n} \| 
\|(\hat a_{-n}^n)^{-1} \|}{\|\Id \oplus \hat\Xi_0^{n}\| } \Big).
\]
Using the estimates of part 1 and $\theta'$ instead of $\theta_*$ in lemma 
\ref{lemma:NormGammaEstimate},  we obtain
\begin{align*}
\frac{\| \hat d_{-n}^n \| \| \hat \Gamma_{-n} \| \|(\hat a_{-n}^n)^{-1} \|}
{\|\Id \oplus \hat\Xi_0^{n}\| } &\leq  2 \widehat C_{\epsilon,d}^7 
C_{\epsilon,d}^{8d} K_d  \theta'(1-\theta')^{-3d+1} \Gstar  e^{N_*\mu}  
e^{-(n-1)N_*\tau} \\
&\leq e^{N_*\mu}   e^{-(n-1)N_*\tau}.
\end{align*}
Using $\| \Id \oplus \hat\Xi_0^1 \| =1$, one obtains
\begin{gather*}
\|\Id \oplus \hat\Xi_0^{n}\| \leq \prod_{k=0}^{n-2} \Big[ 1+ e^{N_*\mu} 
e^{-kN_*\tau} \Big]^{-1}.
\end{gather*}
Using the bound from below in lemma 
\ref{lemma:firstCrucialStep} for $\gamma(\hat U_0,\hat F_0)$ and the 
comparison estimate in lemma \ref{lemma:comparisonAngleGraph}, one gets
\begin{align*}
\gamma(\hat A_{-n}^n \hat U_{-n}, \hat F_0) &\geq 
\frac{\gamma(\hat U_0,\hat F_0)}{\| \Id \oplus \hat \Xi_0^n \|} 
\geq \frac{ (1-\theta_*)^d \Gstar ^{-1}}
{\widehat C_{\epsilon,d}^{3} C_{\epsilon,d}^{4d} K_d} 
e^{-N_*\mu} \prod_{k=0}^{n-2} \Big[ 1+ e^{N_*\mu} 
e^{-kN_*\tau} \Big]^{-1}. \qedhere
\end{align*}
\end{proof}

We now explain how to choose $\theta_*$ so that $N_*$ is the smallest possible. 
We use the following lemma whose proof is left to the reader. 
We will choose later $\alpha = 3d+6$.
\begin{lemma}
Let $\alpha >1$. Then
\begin{itemize}
\item $\theta_* := \frac{1}{1+\alpha} = \argmax \{\theta(1-\theta)^\alpha : 
1 < \theta < 1 \}$, 
\item $\theta_*(1-\theta_*)^\alpha  \geq \theta_*(1-\alpha\theta_*) = \frac{1}{(\alpha+1)^2}$.
\end{itemize}
\end{lemma}

We estimate the infinite product in lemma \ref{lemma:secondCrucialStep} 
using the following lemma. We will choose later $\rho=\mu/\tau$  and $a=e^{-N_*\tau}$.

\begin{lemma} \label{lemma:intermediateEstimate}
Let $a \in (0,1)$ and $\rho>0$. Then
\[
\prod_{n=0}^{+\infty} \big[1+a^{n-\rho}\big] \leq \exp
\Big( \frac{1+a}{1-a} \Big) \Big(\frac{1}{a} \Big)^{\rho(\rho+2)/2}.
\]
\end{lemma}

\begin{proof}
We choose $n_*$ such that $n_* \leq \rho < n_*+1$. We split the 
infinite product in two parts. On the one hand
\begin{align*}
\prod_{n=0}^{n_*} \big[1+a^{n-\rho}\big] &= \prod_{n=0}^{n_*} 
\big[ a^{\rho-n}+1\big] \Big(\frac{1}{a}\Big)^{\sum_{n=0}^{n_*} \rho-n}, \\
&\leq \exp\Big( \sum_{n=0}^{n_*} a^{\rho-n} \Big) 
\Big(\frac{1}{a}\Big)^{(n_*+1)\rho -n_*(n_*+1)/2} 
\leq \exp \Big( \frac{a^{\rho-n_*}}{1-a} \Big) 
\Big(\frac{1}{a}\Big)^{\rho(\rho+2)/2}.
\end{align*}
On the other hand
\begin{align*}
\prod_{n\geq n_*+1} \big[1+a^{n-\rho}\big] &
\leq \exp\Big( \sum_{n\geq n_*+1} a^{n-\rho} \Big) 
\leq \exp\Big(\frac{a^{n_*+1-\rho}}{1-a} \Big).
\end{align*}
Using the convexity of the function $\rho \in [n_*,n_*+1] \mapsto a^{n_*+1-\rho} + a^{\rho-n_*} $, we obtain  $a^{n_*+1-\rho} + a^{\rho-n_*} \leq 1+a$ and conclude the proof.
\end{proof}

\begin{assumption} \label{assumption:tau_N}
Let $\theta_*=\frac{1}{3d+7}$ and $N_*$  satisfy
\begin{gather}
\Dstar e^{-N_*\tau} \leq \theta_*(1-\theta_*)^{3d+6} 
\frac{1-e^{-\tau}}{2 \widehat C_{\epsilon,d}^7 
C_{\epsilon,d}^{8d+5} K_d \Gstar } < \Dstar e^{-N_*\tau} e^\tau,
\end{gather}
with $K_d := \bar\Delta_d(X)^{5d}$.
\end{assumption}

\begin{proof}[Proof of theorem \ref{theorem:main}, item \ref{item:main_1}]
Using the estimate 
\[
(1-\theta_*)^d \geq 1-\frac{d}{3d+7} = \frac{2d+7}{3d+7} \geq \frac{2}{3},
\]
the second crucial step \ref{lemma:secondCrucialStep}, 
lemma \ref{lemma:intermediateEstimate} with
\[
a := e^{-N_*\tau}, \ \ \rho := \frac{\mu}{\tau}, \ \  e^{-N_*\mu} = a^\rho,
\]  we obtain for every $n\geq0$,
\begin{align*}
\gamma(\hat A_{-n}^n \hat U_{-n}, \hat F_0)  \geq \tfrac{2}{3} 
\widehat C_{\epsilon,d}^{-3} C_{\epsilon,d}^{-4d} K_d^{-1} \Gstar ^{-1} 
\exp\Big(- \frac{1+a}{1-a} \Big) a^{\rho(\rho+4)/2}
\end{align*}
with $K_d = \bar \Delta_d(X)^{5d}$. Using $a \leq \tfrac{1}{2}\theta_*$, we obtain
\begin{gather}
\frac{1+a}{1-a} \leq \frac{6d+15}{6d+13} \leq \frac{15}{13}, \quad 
\frac{2}{3} \exp\Big(- \frac{1+a}{1-a} \Big) \geq \frac{1}{5}, \notag \\
\gamma(\hat A_{-n}^n \hat U_{-n}, \hat F_0)  
\geq \frac{ a^{\rho(\rho+4)/2}}{5 \widehat C_{\epsilon,d}^{3} 
C_{\epsilon,d}^{4d} K_d \Gstar } 
\label{equation:minimalAngleBound}.
\end{gather}
Using $\displaystyle a\geq \frac{(3d+7)^{-2}}{2 \widehat C_{\epsilon,d}^{7} 
C_{\epsilon,d}^{8d+5} K_d  \Gstar } \frac{1-e^{-\tau}}{\Dstar e^{\tau}} $, we obtain
\begin{align*}
\gamma(\hat A_{-n}^n \hat U_{-n}, \hat F_0) 
\geq \frac{1}{5 \widehat C_{\epsilon,d}^{3} 
C_{\epsilon,d}^{4d} K_d \Gstar } 
\Big[ \frac{(3d+7)^{-2}}{2 \widehat C_{\epsilon,d}^{7} 
C_{\epsilon,d}^{8d+5} K_d\Gstar } 
\frac{1-e^{-\tau}}{\Dstar e^{\tau}}  \Big]^{\frac{\rho(\rho+4)}{2}}. 
\label{equation:minimalAngleBoundBis}
\end{align*}
We conclude by using 
$\hat A_{-n}^n \hat U_{-n} \to \hat E_0$ and the comparison between 
the minimal gaps, $\gamma(E_0,F_0) \geq \gamma(\hat E_0,\hat F_0) / K'_d$ 
where the constant $K'_d = \bar \Delta_2(X)^{4d} \bar \Delta_d(X)^{3d}$ 
is given by lemma \ref{lemma:comparisonAngleExteriorProduct}.
\end{proof}

\section{Proof of items \ref{item:main_2} and \ref{item:main_3}  of theorem \ref{theorem:main}}
\label{section:proofMainResult_2_3}

We first show that property (FI) is related to a super-multiplicative  sequence \eqref{equation:supermultiplicativeSequence} $(f_m(k))_{m\geq0}$. We  use the notion of Jacobian  
of index $d$, introduced in definition in  \ref{definition:Jacobian} and denoted by  
$\Sigma_d(A)$. Proposition \ref{proposition:comparisonSingularValues} implies, 
\[
\prod_{i=1}^d \sigma_i(A) \leq \Sigma_d(A) = \prod_{i=1}^d \sigma_i''(A) \leq K_d \prod_{i=1}^d \sigma_i(A)
\]
where $K_d = \bar\Delta_d(X)^{2d^2}$. In the Hilbert case $K_d=1$ and  $\Sigma_d(A) = \prod_{i=1}^d \sigma_i(A)$.  Proposition 
\ref{proposition:MultiplicativityJacobian} shows that the Jacobian is sub-multiplicative, 
\[
\forall k \in \mathbb{Z}, \ \forall m_1,m_2 \geq 0, \quad 
\Sigma_d(k,m_1+m_2) \leq \Sigma_d(k,m_1) \Sigma_d(k+m_1,m_2),
\]
where $\Sigma_d(k,n) := \Sigma_d(A(k,n))$. 
We define for every $k\in\mathbb{Z}$ and $m\geq0$,
\begin{equation} \label{equation:supermultiplicativeSequence}
f_m(k) := \inf_{n\geq0} \ \frac{\Sigma_d(k-m,m+n)}{\Sigma_d(k-m,m) \Sigma_d(k,n)}.
\end{equation}
We have obviously $f_m(k) \leq K_d(X)^{-1} \leq 1$. We show in the following lemma that 
$f_m(k)$ is super-multiplicative and that the ratio appearing in 
property (FI) is comparable to $f_m(k)$.

\begin{lemma} \label{lemma:FBI_submultiplicative}
For every $k\in\mathbb{Z}$,
\begin{enumerate}
\item \label{item1:FBI_submultiplicative} $\forall m_1, m_2 \geq 0, 
\quad f_{m_1+m_2}(k) \geq f_{m_1}(k)  f_{m_2}(k-m_1)$ and $f_m(k) \leq 1$,
\item \label{item2:FBI_submultiplicative}  
$\displaystyle{K_d^{-2} \inf_{n\geq0}  \prod_{i=1}^d 
\frac{\sigma_i(k-m,m+n) }{\sigma_i(k-m,m) \sigma_i(k,n) } 
\leq f_m(k) \leq K_d \inf_{n\geq0}  \prod_{i=1}^d 
\frac{\sigma_i(k-m,m+n) }{\sigma_i(k-m,m) \sigma_i(k,n)}}$, 
\item \label{item3:FBI_submultiplicative} $\forall m,n \geq0, 
\quad \displaystyle{ \prod_{i=1}^d 
\frac{\sigma_i(k-m,m+n) }{\sigma_i(k-m,m) \sigma_i(k,n) } \leq K_d^2}$,
\end{enumerate}
with $K_d = \bar\Delta_d(X)^{2d^2}$,
\end{lemma}

\begin{proof}[Proof of item 1]
As $\Sigma_d(k-m_1-m_2,m_1+m_2) \leq \Sigma_d(k-m_1-m_2,m_2) \Sigma_d(k-m_1,m_1)$,
\begin{multline*}
\frac{\Sigma_d(k-m_1-m_2, m_1+m_2+n)}{\Sigma_d(k-m_1-m_2,m_1+m_2) \Sigma_d(k,n)} \\
\geq \frac{\Sigma_d(k-m_1-m_2, m_1+m_2+n)}
{\Sigma_d(k-m_1-m_2,m_2)\Sigma_d(k-m_1,m_1+n)} 
\ \frac{\Sigma_d(k-m_1,m_1+n)}{\Sigma_d(k-m_1,m_1) \Sigma_d(k,n)}.
\end{multline*}
The first quotient is bounded from below by $f_{m_2}(k-m_1)$, the second 
by $f_{m_1}(k)$. 

{\it Proof of item  2 and 3.} The proof follows the comparison between $\Sigma_d(k,n)$ and 
$\prod_{i=1}^d \sigma_i(k,n)$.
\end{proof}
In the following lemma we estimate a bound from below of $f_m(k)$ from  
partial information on $f_{mN_*}(k)$.

\begin{lemma} \label{lemma:partialRatioFBI}
Let $N_*\geq1$, $\alpha \geq 1$, and $(A_k)_{k\in\mathbb{Z}}$ be  
a sequence of operators satisfying property (FI). Then for every $k\in\mathbb{Z}$,
\[
\inf_{m\geq1} f_m(k) \geq K_d^{-1}\Gstar ^{-2}e^{-(1+\alpha)N_*\mu} 
\inf_{m\geq1, \ n\geq \alpha N_*} 
\frac{\Sigma_d(k-mN_*,mN_*+n) }{\Sigma_d(k-mN_*,mN_*) \Sigma_d(k,n)},
\]
where $K_d = \bar\Delta_d(X)^{8d^2}$.
\end{lemma}

\begin{proof}
We claim for every $m\geq1$,
\[
f_{mN_*}(k) \geq K_d^{-1}\Gstar ^{-1}e^{-\alpha N_*\mu} \inf_{n\geq \alpha N_*} 
\frac{\Sigma_d(k-mN_*,mN_*+n) }{\Sigma_d(k-mN_*,mN_*) \Sigma_d(k,n)}.
\]
It is enough to bound from below in the definition of $f_{mN_*}(k)$,
\[
\inf_{1 \leq n \leq \alpha N_*} \ \frac{\Sigma_d(k-mN_*,mN_*+n) }{\Sigma_d(k-mN_*,mN_*). \Sigma_d(k,n)}
\]
Consider $1 \leq n\leq \alpha N_*$ and choose   $p$ 
such that $ \alpha N_*  \leq p$. Then
\[
\Sigma_d(k-mN_*,mN_*+n) \Sigma_d(k+n,p-n) \geq \Sigma_d(k-mN_*,mN_*+p). 
\]
Dividing by $\Sigma_d(k-mN_*,mN_*) \Sigma_d(k,n)$ and rewriting in a different way, we obtain
\begin{multline*}
\frac{\Sigma_d(k-mN_*,mN_*+n) }{\Sigma_d(k-mN_*,mN_*) \Sigma_d(k,n)} \geq \\
\Big[\frac{\Sigma_d(k-mN_*,mN_*+p) }{\Sigma_d(k-mN_*,mN_*) 
\Sigma_d(k,p)} \Big] 
\ \Big[ \frac{\Sigma_d(k+n-n,p) }{\Sigma_d(k+n-n,n) \Sigma_d(k+n,p-n) } \Big].
\end{multline*}
The second bracket is bounded from below using property (FI) by
\[
f_{k+n}(n) \geq {K'_d}^{-1}\Gstar ^{-1}e^{-n\mu} \geq {K'_d}^{-1}\Gstar ^{-1}e^{-\alpha N_*\mu},
\]
where $K'_d=\bar\Delta_d(X)^{4d^2}$ is obtained from lemma 
\ref{lemma:FBI_submultiplicative}. The claim is proved. We conclude 
by using the super-multiplicative property
\[
\forall 0 \leq n \leq N_*, \  f_{mN_*+n}(k) 
\geq f_{mN_*}(k) f_n(k+mN_*) \geq f_{mN_*}(k) 
{K'_d}^{-1}\Gstar ^{-1} e^{-N_*\mu}. \qedhere
\]
\end{proof}

\bigskip
\begin{proof}[Proof of theorem \ref{theorem:main}, item \ref{item:main_2}]

{\it Step 1.} We use lemma \ref{lemma:boundBelowFBI} to bound from 
below the ratio in   property  (FI) by the angle between the fast and slow local spaces,
\[
\forall m,n\geq0, \quad \prod_{i=1}^d 
\frac{\sigma_i(k-m,m+n)}{\sigma_i(k-m,m)\sigma_i(k,n)} 
\geq \widehat C_{\epsilon,d}^{-3} \ \gamma(\widehat{\tilde U}(k,m), \widecheck V(k,n)).
\]
{\it Step 2.} We show for every $n\geq (1+\frac{\rho(\rho+4)}{2})N_*$ and $m\geq1$,
\[
\delta(\widecheck V(k,n), \widecheck F_k) 
\leq \frac{5}{2(3d+7)} \gamma(\widehat{\tilde U}(k,mN_*),\widecheck F_k).
\]
From the definition of $N_*$ in assumption \ref{assumption:tau_N}, we obtain 
\[
\Dstar e^{-n\tau} \leq \theta'(1-\theta')^6 \frac{1-e^{-\tau}}{C_{\epsilon,d}^4}, 
\quad \theta' := \theta_* e^{-(n-N_*)\tau} 
\frac{(1-\theta_*)^{3d}}{2 \widehat C_{\epsilon,d}^7 C_{\epsilon,d}^{8d+1} K_d\Gstar }
\]
with $K_d := \bar\Delta_d(X)^{5d}$. From notations \ref{notations:graphOperator} and lemma \ref{lemma:existenceSlowSpace} and \ref{lemma:normExtendedGraph},
\begin{gather*}
F_k^\perp = \Graph(\Theta_k(n)^\perp) \quad\text{for some}\quad \Theta_k(n)^\perp : V(k,n)^\perp \to U(k,n)^\perp, \\
\widecheck F_k = \Graph(\widehat \Theta_k(n)) \quad\text{for some}\quad \widehat\Theta_k(n) : \widecheck V(k,n) \to \widehat U(k,n), \\
\| \Theta_k(n)^\perp\| \leq \theta', \quad \| \widehat \Theta_k(n) \| \leq C_{\epsilon,d}^{2d} K'_d \theta'(1+\theta')^{d-1}.
\end{gather*}
with $K'_d := \bar\Delta_d(X)^d$. Using   $(1+\theta') \leq (1-\theta_*)^{-1}$ 
and lemma \ref{lemma:ConvergenceGraph}, we obtain
\begin{align*}
\delta(\widecheck V(k,n), \widecheck F_k) &
\leq \| \widehat\Theta_k(n) \| 
\leq  \theta_* (1-\theta_*)^{2d+1} e^{-(n-N_*)\tau} 
\frac{K'_d}{2\widehat C_{\epsilon,d}^7 C_{\epsilon,d}^{6d+1}K_d \Gstar }, \\
& \leq \frac{(3d+7)^{-1}}{2} \frac{K'_d\big( e^{-N_*\tau} \big)^{\rho(\rho+4)/2}}{\widehat C_{\epsilon,d}^{7} C_{\epsilon,d}^{6d+1}K_d \Gstar }.
\end{align*}
On the other hand, using equation \eqref{equation:minimalAngleBound},
\[
 \gamma(\widehat{\tilde U}(k,mN_*),\widecheck F_k) 
 \geq \frac{1}{5}  \frac{  \big( e^{-N_*\tau} \big)^{\rho(\rho+4)/2}}{\hat C_{\epsilon,d}^{3} 
 C_{\epsilon,d}^{4d} K_d \Gstar }
\]
and using the bound $K'_d \leq C_{\epsilon,d}$, we conclude the proof of the claim,
\[
\delta(\widecheck V(k,n), \widecheck F_k) \leq \frac{5}{2(3d+7)} 
\gamma(\widehat{\tilde U}(k,mN_*),\widecheck F_k).
\]
{\it Step 3.} We conclude the proof of  item \ref{item:main_2} of theorem \ref{theorem:main}. Equations 
\eqref{equation:Kato4_4_5} imply
\begin{align*}
\gamma(\widehat{\tilde U}(k,mN_*), \widecheck V(k,n)) &
\geq \frac{\gamma(\widehat{\tilde U}(k,mN_*), \widecheck F_k) - 
\delta(\widecheck V(k,n), \widecheck F_k)}
{1+\delta(\widecheck V(k,n), \widecheck  F_k)}, \\
 &\geq \frac{6d+9}{6d+19}
 \gamma(\widehat{\tilde U}(k,mN_*), \widecheck F_k) 
 \geq \tfrac{3}{5}\gamma(\widehat{\tilde U}(k,mN_*), \widecheck F_k).
\end{align*}
Using lemma \ref{lemma:partialRatioFBI} with $\alpha=1+\frac{\rho(\rho+4)}{2}$, one gets 
\begin{gather*}
\inf_{m\geq1} f_m(k) \geq \frac{3}{5} 
\inf_{m\geq1} \gamma(\widehat{\tilde U}(k,mN_*), \widecheck F_k) 
\frac{\big( e^{-N_*\mu} \big)^{2+\rho(\rho+4)/2}}{ \widehat C_{\epsilon,d}^{3} {K''_d}  \Gstar ^{2} },
\end{gather*}
where $K''_d = \bar\Delta_d(X)^{8d^2}$. Using 
\begin{gather*}
\mu = \tau\rho, \ \bar\Delta_d(X)^{8d^2+5d}\leq C_{\epsilon,d}^2, \\
\frac{\rho(\rho+4)}{2} + \rho \Big ( 2 + \frac{\rho(\rho+4)}{2} \Big) 
= \tfrac 12\rho( \rho^2 + 5\rho + 8),
\end{gather*}
and item \ref{item2:FBI_submultiplicative} of lemma  \ref{lemma:FBI_submultiplicative}, one obtains
\begin{multline*}
\inf_{m\geq0, n\geq0} \prod_{i=1}^d 
\frac{\sigma_i(k-m,m+n) }{\sigma_i(k-m,m) \sigma_i(k,n)} \geq  \\
\frac{3}{25 \widehat C_{\epsilon,d}^{6} C_{\epsilon,d}^{6d} \Gstar ^3} 
\Big[ \frac{(3d+7)^{-2}}{2 \widehat C_{\epsilon,d}^{7} 
C_{\epsilon,d}^{8d+5} K_d \Gstar } 
\frac{1-e^{-\tau}}{\Dstar e^{\tau}}  \Big]^{\rho(\rho^2+5\rho+8)/2}. \qedhere
\end{multline*}
\end{proof}

\begin{proof}[Proof  of theorem \ref{theorem:main},  item \ref{item:main_3}]
We assume $n\geq (1+\frac{\rho(\rho+4)}{2})N_*$ and write the assumptions 
\ref{assumption:tau_N} on $\theta_*$, $N_*$ in the form
\[
\Dstar e^{-n\tau} \leq \theta' (1-\theta')^6 \frac{1-e^{-\tau}}{C_{\epsilon,d}^4}, 
\quad \theta' = 
\frac{\theta_*(1-\theta_*)^{3d}}
{2 \widehat C_{\epsilon,d}^7 C_{\epsilon,d}^{8d+1}K_d \Gstar } e^{-(n-N_*)\tau}
\]
with $K_d := \bar\Delta_d(X)^{5d}$. Notice that 
$\frac{1}{2}\theta_*(1-\theta_*)^{3d}\leq \frac{1}{20}$.

{\it Part 1.} We first estimate $\gamma(E_k,V(k,n))$ by 
$\gamma(E_k,F_k)$. Equation \eqref{equation:Kato4_4_5} gives,
\[
\gamma(E_k,V(k,n)) \geq \frac{\gamma(E_k,F_k)-\delta(V(k,n),F_k)}{1+\delta(V(k,n),F_k)}.
\]
Item \ref{item1:existenceSlowSpace} of lemma 
\ref{lemma:existenceSlowSpace} and   $(n-N_*)\tau
\geq \frac{ \rho(\rho+4)}{2}N_*\tau$ gives
\begin{align*}
\delta(V(k,n),F_k) \leq \theta' \leq \tfrac{1}{20} 
\widehat C_{\epsilon,d}^{-7} C_{\epsilon,d}^{-8d-1}K_d^{-1}\Gstar ^{-1} 
(e^{-N_*\tau})^{\rho(\rho+4)/2}.
\end{align*}
By taking $n\to+\infty$ in equation \eqref{equation:minimalAngleBound} 
and by using lemma \ref{lemma:comparisonAngleExteriorProduct}, one obtains,
\begin{gather*}
\gamma(E_k,F_k) \geq {K'_d}^{-1} 
\gamma(\widehat E_k, \widecheck F_k) 
\geq  5^{-1}  \widehat C_{\epsilon,d}^{-3} 
C_{\epsilon,d}^{-4d} {K'_d}^{-1} K_d^{-1}\Gstar ^{-1} (e^{-N_*\tau})^{\rho(\rho+4)/2}.
\end{gather*}
where $K'_d = \bar \Delta_2(X)^{4d} \bar \Delta_d(X)^{3d}$.  
As  $K'_dK_d = \bar \Delta_2(X)^{4d} \bar \Delta_d(X)^{8d} 
\leq C_{\epsilon,d}$, we have,
\[
\delta(V(k,n),F_k) \leq \theta' \leq \tfrac{1}{4} \gamma(E_k,F_k), 
\quad \gamma(E_k,V(k,n)) \geq \tfrac{3}{5} \gamma(E_k,F_k).
\]
Using item \ref{item4:PolarDecomposition} of theorem 
\ref{theorem:PolarDecomposition}, we have for every $w \in E_k$,
\begin{align*}
\| A(k,n) w \| &\geq | \langle \tilde \phi | A(k,n) w \rangle |, 
\quad\quad (\forall\tilde\phi \in \tilde V(k+n,n)^\perp,  \ \|\tilde\phi \|=1) \\
\|A(k,n)^*\tilde \phi\| &\geq C_{\epsilon,d}^{-1} \sigma_{d}(k,n) \|, \quad\quad \text{(item \ref{item:aprroximatedSingularSpaces_2}  of theorem \ref{theorem:aprroximatedSingularSpaces})} \\
\|A(k,n) w \| &\geq \langle \frac{A(k,n)^*\tilde \phi}{\|A(k,n)^*\tilde \phi\|} | w \rangle \|A(k,n)^*\tilde \phi\| , \\
&\geq \sup\{ |\langle \phi | w \rangle| : 
\phi \in V(k,n)^\perp, \ \|\phi\|=1\} C_{\epsilon,d}^{-1} \sigma_d(k,n) \\
&\geq \gamma(E_k,V(k,n)) C_{\epsilon,d}^{-1} \sigma_d(k,n) \|w\|, \quad\quad\text{(equation \eqref{equation:minimalGap})} \\
&\geq \tfrac{3}{5} \gamma(E_k,F_k) C_{\epsilon,d}^{-1} \sigma_d(k,n) \|w\|.
\end{align*}

{\it Part 2.} We estimate $\gamma(F_k,\tilde U(k,n))$ by $\gamma(F_k,E_k)$. 
Using equation \eqref{equation:Kato4_4_5} and item 
\ref{item:existenceFastSpace_1} of lemma \ref{lemma:existenceFastSpace}, we have
\begin{align*}
\gamma(F_k,\tilde U(k,n)) &\geq \frac{\gamma(F_k,E_k)-
\delta(\tilde U(k,n),E_k)}{1+\delta(\tilde U(k,n),E_k)} \\
\delta(\tilde U(k,n),E_k) &\leq \theta' 
\leq \tfrac{1}{4}\gamma(E_k,F_k) \leq \tfrac{1}{2} \gamma(F_k,E_k) \\
\gamma(F_k,\tilde U(k,n)) &\geq \tfrac{1}{3} \gamma(F_k,E_k).
\end{align*}
Let $w \in F_k$, $w=u+v$ where $u\in U(k,n)$ and 
$v \in V(k,n)$. Then $\|v\| \leq C_{\epsilon,d} \|w\|$ thanks to 
item \ref{item:aprroximatedSingularSpaces_3}  of theorem \ref{theorem:aprroximatedSingularSpaces}, 
\begin{align*}
A(k,n)w &= \tilde u + \tilde v, \quad \tilde u \in \tilde U(k+n,n), \ \tilde v \in \tilde V(k+n,n), \\
\| \tilde v \| &\leq C_{\epsilon,d} \sigma_{d+1}(k,n) 
\| v \| \leq C_{\epsilon,d}^2 \sigma_{d+1}(k,n) \| w \|, \\
\| \tilde v \| &\geq \| A(k,n)w \| \,\gamma(F_{k+n},\tilde U(k+n,n)).
\end{align*}
Hence
\[
\| A(k,n)w\| \leq 3 C_{\epsilon,d}^2 
\gamma(F_{k+n}, E_{k+n})^{-1} \sigma_{d+1}(k,n) \|w\|. \qedhere
\]
\end{proof}

\appendix
\appendixpage
\addappheadtotoc
\numberwithin{equation}{section}
\counterwithin{theorem}{section}

The purpose of this appendix is to clarify the notion of 
{\it approximate singular value decomposition} of a bounded operator 
in a Banach space. We need two precise theorems
\ref{theorem:PolarDecomposition} and 
\ref{theorem:codimensionOneSingularValueDecomposition}.   The first theorem 
is usually stated for compact selfadjoint operators in an Hilbert space 
(see \cite{Pietsch1987}). In Hilbert spaces, for non compact operators, 
we did not find good references, although the results are certainly known 
by the specialists. In Banach spaces, we are not aware of any statements as in \ref{theorem:PolarDecomposition} and  \ref{theorem:codimensionOneSingularValueDecomposition}. Nevertheless quite similar ideas may be found in [1] and [8].

\section{Basic results in Banach spaces} \label{appendix:basicResultsBanachSpaces}

Let $(X, \|\cdot \|)$ be a real Banach space. We do not assume $X$ to be 
reflexive. We call $X^*$ the topological dual space and denote by 
$\langle \eta | u \rangle$ the duality between $\eta \in X^*$ and 
$u \in X$. If $X$ is an Hilbert space we identify $X^*=X$ and the 
duality $\langle \cdot | \cdot \rangle$ with the scalar product.  
If $U$ is a closed (vector) subspace of $X$, $U$ becomes a Banach 
space with the induced norm,  $U^*$ denotes  the corresponding 
{\it dual space}, and $U^\perp$ denotes the {\it annihilator} of $U$, the subspace of linear forms of 
$X^*$ vanishing on $U$. Conversely if $H \subset X^*$ is a subspace, 
the \emph{pre-annihilator} of $H$ is the subspace  
$H^\Perp := \{ u \in X : \langle \eta | u \rangle = 0, \ \forall \eta \in H \}$. 
Write $\mathcal{B}(X)$ for the space of bounded linear operators on $X$. 
If $(Y,\|\cdot\|)$ is another Banach space, write $\mathcal{B}(X,Y)$ for 
the space of bounded linear operators from $X$ to $Y$. If $U\subset X$ 
is a closed subspace of $X$, we denote by $A|U$ the restriction to $U$ of 
$A \in \mathcal{B}(X, Y)$. We say that a splitting $X = U \oplus V$ of 
two closed subspaces is topological if the projector $\pi_{U | V}$ onto 
$U$ parallel to $V$ (or equivalently $\pi_{V|U}$) is a bounded operator. For a Bounded operator $A \in \mathcal{B}(X,Y)$, we call $A^* \in \mathcal{B}(Y^*,X^*)$ the dual operator.

\subsection{Auerbach basis and distortion}

The purpose of  this section is to clarify the notion of a distortion of a 
Banach norm with respect to the best euclidean norm. We use the notion 
of Auerbach bases as a substitute for orthonormal bases. We begin by recalling  
the notion of Auerbach families. 

\begin{definition}
Let $X$  be a  Banach space, and $d\geq1$. 
\begin{itemize}
\item A family of vectors $(u_1,\ldots,u_d)$ in $X$ is said to be
{\it Auerbach} if 
\[
\forall j = 1,\ldots,d, \quad \|u_j\|=1 \ \ \text{and} \ \ \dist(u_j, \Vect(u_k : k\not=j))=1.
\]
\item If $(u_1,\ldots,u_d)$ are linearly independent in  $X$, a
{\it dual family} is any family of linear forms $(\eta_1,\ldots,\eta_d)$ of $X^*$ 
satisfying $\langle \eta_i | u_j \rangle = \delta_{ij}$. Similarly if 
$(\eta_1,\ldots,\eta_d)$ are linearly independent in  $X^*$, a 
{\it predual family} is any family of vectors $(u_1,\ldots,u_d)$ of $X$ satisfying 
$\langle \eta_i | u_j \rangle = \delta_{ij}$.
\end{itemize}
\end{definition}

If $\dim(X)=d$, dual bases and predual families do always exist and they are unique.  
We show in the following lemma that Auerbach families can be characterized by the 
existence of normalized dual families.

\begin{lemma} \label{lemma:characterizationAuerbachBasis}
Let $X$  be a  Banach space, and  $d\geq1$. 
\begin{enumerate}
\item A family of vectors $(u_1,\ldots,u_d)$ of $X$ is  Auerbach 
if and only if  $\|u_j\|=1$ for every $j=1,\ldots,d$ and  there exists a dual 
family $(\eta_1,\ldots,\eta_d)$ of $X^*$  satisfying $\|\eta_i\|=1$ for every $j=1,\ldots,d$.
\item \label{item2:characterizationAuerbachBasis} Suppose $\dim(X) = d$. 
A family of linear forms $(\eta_1,\ldots,\eta_d)$ of $X^*$ is an 
{\it Auerbach basis} if and only if $\|\eta_i\|=1$ and its unique predual 
family $(u_1,\ldots,u_d)$  of $X$ satisfies $\|u_j\|=1$ for every $j=1,\ldots,d$.
\end{enumerate}
\end{lemma}

If $\dim(X) = +\infty$, an Auerbach family in $X^*$ does not admit in general 
a predual Auerbach family. We will show in lemma \ref{lemma:predualAuerbachBasis} 
that such predual families do exist if we relax a little the notion of Auerbach family. 
If $X$ is an Hilbert space of finite dimension, an Auerbach family is an orthonormal 
family, and two families of vectors $(u_1,\ldots,u_d)$ and $(\eta_1,\ldots,\eta_d)$ 
are dual to each other if and only if they are equal.  

The following lemma shows that Auerbach families  exist in any Banach space. 
We will see that this notion is a key tool for the notion of singular values of bounded operators.

\begin{lemma} \label{lemma:existenceAuerbachBases}
Let  $X,Y$ be  Banach spaces, $\dim(X)= d\geq1$, $A \in \mathcal{B}(X,Y)$  
injective, and $\tilde X=AX$. Let $(u_1,\ldots,u_d)$ be vectors of $X$ and 
$(\tilde \eta_1,\ldots,\tilde\eta_d)$ be linear forms of $\tilde X^*$ realizing 
the supremum in
\[
\Sigma_d(A) := \sup 
\big\{ \det \big([\langle \tilde\eta_i | A u_j \rangle]_{1 \leq i,j \leq d} \big) \,
:\, \tilde\eta_i \in \tilde X^*, \ u_j \in X, \  \|\tilde\eta_i\|=\|u_j\|=1 \big \}.
\]
Let $\eta_i$ be a Hahn-Banach extension to $Y$ of $\tilde\eta_i$ with $\|\eta_i\|=1$.  
Then $(u_1,\ldots,u_d)$ is an Auerbach family of $X$,   $(\eta_1,\ldots,\eta_d)$ 
is an Auerbach family of ${Y}^*$, and
\[
\Sigma_d(A) = \sup \big\{ \det \big([\langle \zeta_i | A u_j \rangle]_{1 \leq i,j \leq d} \big)
\,:\, \zeta_i \in Y^*, \ u_j \in X, \ \|\zeta_i\|=\|u_j\|=1 \big \}.
\]
\end{lemma}
Notice in the previous lemma that,  in the case  $X=Y$ and $A=\Id$,  
$(\eta_1,\ldots,\eta_d)$ and $(u_1,\ldots,u_d)$ are not a priori dual 
to each other. We call   the particular constant 
$\Sigma_d(A)$ appearing in lemma \ref{lemma:existenceAuerbachBases} when $A=\Id$, the {\it projective distortion}
\begin{equation} \label{equation:projectiveDistortion}
\Sigma_d(X) := \sup \big\{ \det \big([\langle \eta_i |  u_j \rangle]_{1 \leq i,j \leq d} \big) 
\,:\, \eta_i \in X^*,  u_j \in X,  \|\eta_i\|=\|u_j\|=1 \big \}.
\end{equation}
The name ``projective distortion'' is related to the notion of projective norm 
introduced in \eqref{equation:ProjectiveNorm} and the estimate of the distortion 
of the canonical duality \eqref{equation:definitionDualityWedgeProduct} and 
\eqref{equation:normDualityWedgeProduct}.

A Banach norm introduces a distortion in the volume of unit balls of 
finite-dimensional subspaces. This distortion may depend on the 
dimension of the subspace. In order to obtain optimal estimates when $X$ is 
actually an Hilbert space, we introduce a notion of volume distortion that 
turn out to be trivial for Hilbert spaces.

\begin{definition} \label{definition:volumeDistortion}
Let $X$ be a Banach space  and  $d\geq1$. The {\it volume distortion} is
\begin{gather}
\Delta_d(X) := \sup \Big\{ \frac{\| \sum_{j=1}^d \lambda_ju_j \|}{ \big(
\sum_{j=1}^d |\lambda_j|^2 \big)^{1/2}} 
\,:\, \text{$u$ is an Auerbach family and } \lambda \not=0 \Big\}
\end{gather}
where the supremum is realized over every $u=(u_1,\ldots,u_d)$  Auerbach family
 of $X$ and every non-zero $\lambda=(\lambda_1,\ldots,\lambda_d) 
 \in  \mathbb{R}^d$.   If $X$ is a Hilbert space $\Delta_d(X)=1$. 
 In general we have $1\leq \Delta_d(X) \leq \sqrt{d}$. In order to 
 simplify the estimates, we will use instead a {\it simplified volume distortion}
\begin{equation} \label{equation:simplifiedVolumicDistortion}
\bar \Delta_d(X) := \max(\Delta_d(X),\Delta_d(X^*),\Delta_d(X^{**}).
\end{equation}
\end{definition}

Although we do not intend to compute this constant for different 
Banach spa\-ces, we give an exact estimate of $\Delta_d(X)$ for $X=\ell^p_d$
 the space $\mathbb{R}^d$  endowed with the norm  
 $\Vert x\Vert_p=\big(\sum_{n=1}^d \vert x_n\vert^p \big)^{1/p}$, 
 $x=(x_1,\ldots, x_d)$, with natural change for $p=+\infty$. Recall that 
 the {\it Banach-Mazur distance} between two isomorphic spaces $X$ and $Y$ is the number
\begin{equation*}
d_{BM}(X,Y) := \inf \{\Vert T\Vert \Vert T^{-1}\Vert,\ T:X\to Y
\ \mathrm{linear \ bounded \ isomorphism}\}.
\end{equation*}

\begin{proposition} \label{proposition:simpleDistortionEstimate}
For every $p \in [1,2]$,
$\Delta_d(\ell^p_d)=d_{BM}(\ell^p_d,\ell^2_d)=d^{\vert \frac{1}{p}- \frac{1}{2}\vert}$. Hence
\[
\lim_{p\to 2^-} \Delta_d(\ell_d^p) = 1.
\]
\end{proposition}

If $U\subset X$ is a subspace of $X$, then $\Delta_d(U) \leq \Delta_d(X)$. 
We have for instance $\Delta_d(X) \leq \Delta_d(X^{**})$. 
By extending any Auerbach family $(\eta_1,\ldots,\eta_d)$ of $U^*$ 
by Hahn-Banach while keeping $\|\eta_i \|=1$, we still  obtain an 
Auerbach family in $X^*$ and thus  $\Delta_d(U^*) \leq \Delta_d(X^*)$.
We show in the following lemma that $\Delta_d(X)$ and  $\Delta_d(X^*)$ 
admit equivalent definitions in the case $\dim(X)=d$.

\begin{lemma}
Let  be $d\geq1$ and $X$ be a Banach space of dimension $d$. Then
\begin{enumerate}
\item $\displaystyle{\Delta_d(X^*) = \sup \Big\{ 
\frac{ \big(\sum_{i=1}^d | \lambda_j|^2 \big)^{1/2} }{\|\sum_{j=1}^d \lambda_j u_j\|} 
\,:\,  \text{$u$ is an Auerbach basis of $X$}, \ \lambda\not=0\} \Big\}}$,
\item $\displaystyle{\Delta_d(X) = \sup \Big\{ 
\frac{ \big(\sum_{i=1}^d |\lambda_i|^2 \big)^{1/2}}{\| \sum_{i=1}^d \lambda_i \eta_i \| }
\,:\,  \text{$\eta$ is an Auerbach basis of $X^*$}, \ \lambda \not=0 \Big\}}$,
\item $\Delta_d(X) = \Delta_d(X^{**})$.
\end{enumerate}
\end{lemma}

In particular we obtain an ``explicit'' bound between the Banach norm and the Euclidean  
norm either in $U$ or in $U^*$.

\begin{corollary}
Let $d\geq1$ and $X$ be a Banach space of dimension $d$.
\begin{enumerate}
\item If $(u_1, \ldots,u_d)$ is an Auerbach basis of $X$, then  
\begin{gather*}
\forall \lambda \in \mathbb{R}^d, \ \frac{1}{\Delta_d(X^*)} 
\Big( \sum_{j=1}^d |\lambda_j|^2 \Big)^{1/2} \leq \big\| 
\sum_{j=1}^d \lambda_j u_j \big\| \leq \Delta_d(X) 
\ \Big( \sum_{j=1}^d |\lambda_j|^2 \Big)^{1/2}.
\end{gather*}
\item If  $(\eta_1,\ldots,\eta_d)$ is an Auerbach basis of $X^*$, then
\begin{gather*}
\forall \lambda\in \mathbb{R}^d, \ \frac{1}{\Delta_d(X)} 
\Big( \sum_{i=1}^d |\lambda_i|^2 \Big)^{1/2} \leq \big\| 
\sum_{i=1}^d \lambda_i\eta_i  \big\|  \leq \Delta_d(X^*) 
\ \Big( \sum_{i=1}^d |\lambda_i|^2 \Big)^{1/2}.
\end{gather*}
\end{enumerate}
\end{corollary}

Every subspace $U \subset X$ of finite dimension $d$  admits a topological 
complement (a closed subspace $V$ such that $X=U \oplus V$). For instance, 
if $(u_1, \ldots, u_d)$ is an Auerbach basis of $U$,  if $(\eta_1,\ldots,\eta_d)$ 
is an Auerbach basis in ${U}^*$ dual to $(u_1,\ldots_,u_d)$, that has been  
extended to $X$ by Hahn-Banach as linear forms of norm one, then 
$(\eta_1,\ldots,\eta_d)$ is again an  Auerbach family in $X^*$, and  
$V = \bigcap_{i=1}^d \ker(\eta_i)$ is a topological complement to $U$ 
where the projector $\pi_{U|V}$ onto $U$ parallel to $V$ is given by 
\begin{equation} \label{equation:canonicalProjector}
\pi_{U | V}(w) = \sum_{i=1}^d \langle \eta_i | w \rangle u_i, \quad \forall w \in X.
\end{equation}
Notice that if $(u_1,\ldots,u_d)$ and $(\eta_1,\ldots,\eta_d)$ are dual 
to  each other but not necessarily Auerbach, then in addition to \eqref{equation:canonicalProjector}, we have,
\begin{equation} 
\begin{split}
\pi_{V | U} &= \Id -\pi_{U|V} =  \pi_d \circ \cdots \circ \pi_1, 
\quad\text{where}\\
\pi_k(w) &= w - \langle \eta_k | w \rangle u_k, \quad \forall w \in X.
\end{split}
\end{equation}

\begin{definition} \label{definition:almostOrthogonalDecomposition}
Let $X$ be a Banach space, $d\geq1$, and $X=U \oplus V$ be a splitting  
such that  $\dim(U)=d$. We say that the splitting is {\it orthogonal} if there 
exist Auerbach families $(u_1,\ldots,u_d)$ of $X$ and $(\eta_1,\ldots,\eta_d)$ 
of $X^*$ dual to each other such that 
\[
U = \Vect(u_1,\ldots,u_d) \ \ \text{and} \ \ V = \bigcap_{i=1}^d \ker(\eta_i) =
\Vect(\eta_1,\ldots,\eta_d)^\Perp.
\] 
\end{definition}

If $X$ is a Hilbert space, we recover the usual notion of orthogonal complements. 
In particular the two  projectors $\pi_{V | U}$  and $\pi_{U|V}$  have  norm one.  
In general if $X$ is a Banach space, the norm of the projectors is not any more one. 
We give two results giving the bound of the norm of these projectors in terms 
of the volume distortion. We use the simplified volume distortion given in 
\eqref{equation:simplifiedVolumicDistortion}.

\begin{lemma} \label{lemma:projectorDistortionBasic}
Let $X$ be a Banach space, $u\in X$, $\eta \in X^*$, 
such that $\langle \eta | u \rangle = 1$, and $\| \eta \| = 1$.
Let $U = \Vect(u)$, $V = \ker(\eta)$, and   $K_d := \bar\Delta_2(X)^3$. Then 
\[
\| \pi_{U|V} \|= \|u\|, \ \ \text{and} \ \ \| \pi_{V|U} \| \leq K_d \|u\|.
\]
\end{lemma}

For any dimension, we obtain the following bound.

\begin{lemma} \label{lemma:projectorDistortion}
Let $X$ be a Banach space, $d\geq1$, $\dim(U)=d$, and $X=U \oplus V$ 
be an  orthogonal splitting. Let $K_d := \bar\Delta_2(X)^4 \bar\Delta_d(X)^2$. Then
\begin{equation*}
\forall u \in U, \ \forall v \in V, \quad \frac{1}{K_d}
\sqrt{\|u\|^2+\|v\|^2} \leq \| u+v \| \leq K_d \sqrt{\|u\|^2+\|v\|^2}
\end{equation*}
In particular $\|\pi_{U|V}\| \leq K_d$ and $\|\pi_{V | U}\| \leq K_d$. 
\end{lemma}

We are now able to extend  item \ref{item2:characterizationAuerbachBasis} of lemma 
\ref{lemma:characterizationAuerbachBasis} to Banach spaces of infinite dimension.

\begin{lemma} \label{lemma:predualAuerbachBasis}
Let $X$ be a Banach space   and $d\geq1$. Let be $K_d := \bar\Delta_2(X)^{3d}$. 
Then for every Auerbach family $(\eta_1,\ldots,\eta_d)$ of $X^*$, for every 
$\epsilon >0$, there exist a predual family $(u_1, \ldots,u_d)$ in $X$ satisfying
\begin{gather*}
1 \leq \dist(u_k, \Vect(u_l : l\not= k)) \ \ \text{and}\ \  \|u_k\| \leq 
(1+\epsilon) K_d, \quad \forall k=1,\ldots,d.
\end{gather*}
If $X$ is a Hilbert space,  $\epsilon=0$, $K_d=1$ and  
$(u_1,\ldots,u_d)=(\eta_1,\ldots,\eta_d)$.
\end{lemma}

The previous result suggests the following definition.

\begin{definition} \label{definition:CAuerbachBasis}
Let $X$ be a Banach space, $d\geq1$ and $C\geq1$. A family of vectors  
$(u_1,\ldots,u_d)$ is said to be a {\it $C$-Auerbach family} if 
\[
C^{-1} \leq   \dist(u_k,\Vect(u_l : l\not=k)) \ \ \text{and}
\ \  \|u_k\| \leq C, \quad \forall k=1,\ldots,d.
\]
A splitting $X=U \oplus V$ where $\dim(U)=d$, is said to be {\it $C$-orthogonal} 
if there exist $C$-Auerbach families $(u_1,\ldots,u_d)$ of $X$ and 
$(\eta_1,\ldots,\eta_d)$ of $X^*$ dual to each other such that 
$U=\Vect(u_1,\ldots,u_d)$ and $V = \Vect(\eta_1,\ldots,\eta_d)^\Perp$.
\end{definition}

Lemma \ref{lemma:predualAuerbachBasis} shows that, if $V$ is a subspace of 
$X$ of codimension $d$, and $\epsilon>0$, then there exists $U$ such that 
$X=U \oplus V$ is a $(1+\epsilon)K_d$-orthogonal splitting. 

If $X$ is a  Hilbert space, a $1$-Auerbach family corresponds to an 
orthonormal family, a $C$-Auerbach family represents a distorted orthonormal family. 
We give in the following lemma several equivalent characterizations of 
$C$-Auerbach bases in the case $X$ is a finite dimensional Hilbert space.

\begin{lemma}
Let $P=[P_{i,j}]_{1\leq i,j \leq d}$ be a real matrix and $C\geq1$. 
$\mathbb{R}^d$ is equipped  with the standard  euclidean norm 
$\|\cdot\|_2$. The following 3 conditions are equivalent.
\begin{enumerate}
\item The column vectors $\overrightarrow{C_j} :=  (P_{i,j})_{i=1}^d$ form a $C$-Auerbach basis.
\item The singular values of $P$ satisfy \ $C \geq \sigma_1 \geq \cdots \geq \sigma_d \geq 1/C$.
\item For every $(\lambda_1,\ldots,\lambda_d) \in \mathbb{R}^d$, 
\[
\frac{1}{C} \Big(\sum_{j=1}^d  |\lambda_j|^2 \Big)^{1/2} 
\leq \big\| \sum_{j=1}^d \lambda_j \overrightarrow{C_j} \big\|_2 
\leq C \Big(\sum_{j=1}^d  |\lambda_j|^2 \Big)^{1/2}.
\]
\end{enumerate}
In particular, since the singular values of $P$ and $P^*$ coincide, 
the 3 conditions are also  equivalent to
\begin{enumerate}
\setcounter{enumi}{3}
\item The row vectors $\overrightarrow{R_i} := (P_{i,j})_{j=1}^d$ form a $C$-Auerbach basis.
\item For every $(\lambda_1,\ldots,\lambda_d) \in \mathbb{R}^d$, 
\[
\frac{1}{C} \Big(\sum_{i=1}^d  |\lambda_i|^2 \Big)^{1/2} 
\leq \big\| \sum_{i=1}^d \lambda_i \overrightarrow{R_i} \big\|_2 
\leq C \Big(\sum_{i=1}^d  |\lambda_i|^2 \Big)^{1/2}.
\]
\end{enumerate}  
\end{lemma}

If $X$ is a Banach space, many previous results involving Auerbach families 
can be extended to $C$-Auerbach families. The volume distortion of  a 
$C$-Auerbach family can be expressed using the volume distortion 
defined in \ref{definition:volumeDistortion}.

\begin{lemma} \label{lemme:CAuerbachComparaison}
Let $X$ be a Banach space, $d\geq1$, and $C\geq1$. 
Define $K_d := \bar\Delta_d(X)^2$. If $(e_1,\ldots,e_d)$ is a 
$C$-Auerbach family, then for every $(\lambda_1,\ldots,\lambda_d) 
\in\mathbb{R}^d$,
\[
\frac{1}{C K_d} \Big( \sum_{j=1}^d |\lambda_j|^2 \Big)^{1/2} 
\leq \big\| \sum_{j=1}^d \lambda_j e_j \big\| \leq C K_d 
\ \Big( \sum_{j=1}^d |\lambda_j|^2 \Big)^{1/2}
\]
\end{lemma}

We extend lemma \ref{lemma:projectorDistortion} to $C$-Auerbach families.

\begin{lemma} \label{lemma:C-orthogonal}
Let $X$ be a Banach space, $d\geq1$ and $C\geq1$. Let $X=U \oplus V$ 
be a $C$-orthogonal splitting with $\dim(U)=d$. Define 
$K_d := \bar\Delta_2(X)^4 \bar\Delta_d(X)^4$. Then
\begin{equation*} 
\forall u \in U, \ \forall v \in V, \quad \frac{1}{C^2K_d}
\sqrt{\|u\|^2+\|v\|^2} \leq \| u+v \| \leq C^2 K_d \sqrt{\|u\|^2+\|v\|^2}.
\end{equation*}
\end{lemma}

We also extend lemma \ref{lemma:characterizationAuerbachBasis} to $C$-Auerbach families.

\begin{lemma} \label{lemma:predualAuerbachBasisBis}
Let $X$ be a Banach space, $C\geq1$, $d\geq 1$, and 
$K_d :=  \bar\Delta_2(X)^{3d} \bar\Delta_d(X)^2$. 
\begin{itemize}
\item If $(u_1,\ldots,u_d)$ is a $C$-Auerbach family of $X$, 
then there exists a $C$-Auerbach family $(\eta_1,\ldots,\eta_d)$ of 
$X^*$ dual to $(u_1,\ldots,u_d)$.
\item If $(\eta_1,\ldots,\eta_d)$ is a $C$-Auerbach family of $X^*$.  
Then for every $\epsilon>0$, there exists a 
$C K_d(1+\epsilon)$-Auerbach family of $X$ predual to $(\eta_1,\ldots,\eta_d)$.
\item If $U$ is a subspace of dimension $d$, $(\tilde\eta_1,\ldots,\tilde\eta_d)$ 
is a $C$-Auerbach basis of $U^*$ and $(\eta_1,\ldots,\eta_d)$ is some 
Hahn-Banach extension such that $\|\tilde\eta_i\|=\|\eta_i\|$, then 
$(\eta_1,\ldots,\eta_d)$ is again a $C$-Auerbach family and there 
exists a  $C$-Auerbach basis $(u_1,\ldots,u_d)$ in $U$ predual to 
$(\eta_1,\ldots,\eta_d)$.
\end{itemize}
\end{lemma}

\subsection{Grassmannian,  gaps, and  graphs} \label{appendix:Grassmannian}

The geometry of Grassmannian spaces is a well studied object in the case of 
Hilbert spaces. For Banach spaces, the notion of angle is not canonically 
well-defined and several equivalent definition could be used. The
\emph{$d$-dimensional Grassmannian space} is the set, $\Grass(d,X)$,  of all  
subspaces of $X$ of dimension $d\geq1$. The \emph{$d$-dimensional 
coGrassmannian space} is the set, $\coGrass(d,X)$, of all closed subspaces of $X$ of 
codimension $d$. We denote by $S_X$ the unit sphere of $X$. We first recall two estimates (see also Kato \cite{Kato1995}, chapter 4, section 2.3); for every closed non trivial subspace $N$ of $X$,

\begin{equation}
\label{equation:distanceToSubspaceEstimate}
\begin{split}
&\dist(u,N) = \sup \{ \langle \phi | u \rangle : \phi \in N^\perp, \ \|\phi\|=1\}, \quad \forall u \in X, \\
&\dist(\phi,N^\perp) = \sup \{ \langle \phi | u \rangle : u\in N, \ \|u\| = 1 \}, \quad \forall \phi \in X^*.
\end{split}
\end{equation}

\begin{definition} \label{definition:maximalGap}
Let $X$ be a Banach space and $M,N$ be two closed non-trivial subspaces of $X$. 
The \emph{maximal gap between $M$ and  $N$} is 
\begin{align}
\delta(M,N):=& \sup \big\{ \dist(u,N) : u \in M, \ \|u\|=1 \big\},\label{eq:maxgap} 
\\
=& \sup \big\{ \langle \phi | u \rangle : u \in M, \ \phi \in N^\perp, 
\ \|u\| = \|\phi\| = 1 \big\}. \notag
\end{align}
We also define another equivalent distance
\begin{equation}
d(M,N) := \sup \big\{ \dist(u,S_N) : u \in M, \ \|u\|=1\big\},
\end{equation}
and observe that $d$ satisfies the triangle inequality and the estimate
\begin{equation}
\delta(M,N) \leq d(M,N) \leq 2 \delta(M,N).
\end{equation}
\end{definition}

The notion of maximal gap between subspaces $\delta(M,N)$  was
introduced by Gohberg and Marcus \cite{GohbergMarcus1959}, 
(see also Kato \cite{Kato1995}, chapter 4, section 2.1), under the name  
{\it opening} or {\it aperture}. We use mainly  $\delta(M,N)$ in two cases: 
either for $\dim(M)=\dim(N) < +\infty$ or for $\codim(M)=\codim(N) < +\infty$. 
We recall the duality identity  (see equation (2.19)  in Kato \cite{Kato1995}, 
chapter 4, section 2.3)
\begin{equation} \label{equation:dualityMaximalGap}
\delta(M,N) = \delta(N^\perp,M^\perp),\quad \forall M,N \ \ \text{closed subspace of $X$}.
\end{equation}
In general the maximal gap is not symmetric, but for finite-dimensional 
subspaces of equal dimension we have (see \cite{Kato1958_01}, Lemma 213)
\begin{equation} \label{equation:symmetryMaximalGap}
\dim M = \dim N < + \infty \ \Rightarrow\ \delta(M,N) \leq \frac{\delta(N,M)}{1-\delta(N,M)}.
\end{equation}
We use another estimate which enables us to recover the standard estimate in the Hilbert case.

\begin{lemma} \label{lemma:symmetricGap}
Let $X$ be a Banach space and $d\geq1$. Define
\[
K_2 := \min (2,\Delta_2(X)^2\Delta_2(X^*)^2).
\]
For every subspaces $M,N$ of $X$, if  $\dim M = \dim N = d$, then
\[
\delta(M,N) \leq K_2 \delta(N,M).
\]
In particular, if $X$ is a Hilbert space, $\delta(M,N)=\delta(N,M)$.
\end{lemma}

For complementary subspaces we use another notion called the minimal gap 
(see Kato \cite{Kato1995}, chapter 4, section 4.1).

\begin{definition} \label{definition:minimumGap}
Let $X$ be a Banach space and $M,N$ be two closed non trivial subspaces of 
$X$. The{\it minimal gap} is
\begin{equation}
\label{equation:minimalGap}
\gamma(M,N) := \inf \big\{ \dist(u, N) : u \in M, \ \|u\|=1\big\}.
\end{equation}
A similar notion has been introduced in \cite{FroylandGonzalezTokmanQuas2014}
\begin{equation}
\ortho (M,N) := \inf \big\{ \|u-v\| : u \in M, \ v \in N, \ \|u\|=\|v\|=1 \big\}.
\end{equation}
The second definition is more symmetric and equivalent to the first one 
\begin{equation}
\gamma(M,N) \leq \ortho(M,N) \leq 2 \gamma(M,N).
\end{equation}
\end{definition}

The notion of minimal gap is equivalent to the notion of minimal angle 
$\theta(M,N)$ that is used in Gohberg and Krein \cite{GohbergKrein1969} 
(chapter VI, section 5.1) where
\[
\theta(M,N) := \arcsin \gamma(M,N), \quad \theta \in[0,\pi/2],
\]
We  use mainly the notion of minimal gap for complementary subspaces 
$X=M \oplus N$ where $M$ and $N$ are closed. The norm of the projector 
onto $M$ parallel to $N$ is not necessarily bounded. Whether it is bounded 
or not, we have (see equation (4.7) in Kato \cite{Kato1995}, chapter 4, section 4.1),
\begin{equation} \label{equation:comparisonAngleNorm}
X = M \oplus N \Rightarrow \gamma(M,N) = \|\pi_{M | N}\|^{-1}.
\end{equation}
Notice that lemma \ref{lemma:C-orthogonal} shows that, if the 
splitting $X=M \oplus N$, with $\dim(M)=d$, is $C$-orthogonal, 
then $\gamma(M,N) \geq 1/(C^2K_d)$. If $X$ is an Hilbert space, 
$\gamma(M,M^\perp)=1$. If  two closed subspaces $N$ and $N'$ are 
complementary with respect to the same $M$, $X=M \oplus N = M \oplus N' $, 
then their minimal gaps are comparable  (see equation (4.34) in Kato 
\cite{Kato1995}, chapter 4, section 4.5)  provided $\delta(N,N')$ is small enough
\begin{equation} \label{equation:Kato4_4_5}
\gamma(M,N') \geq \frac{\gamma(M,N)-\delta(N',N)}{1+\delta(N',N)}, \ \ 
\gamma(N',M) \geq \frac{\gamma(N,M) - \delta(N,N')}{1+\delta(N,N')}.
\end{equation}
The duality identity \eqref{equation:dualityMaximalGap} is also valid for the minimal gap (see equation (4.14)  
Kato \cite{Kato1995}, chapter 4, section 4.2)
\begin{equation} \label{equation:Kato_4_4_2}
X= M \oplus N \Rightarrow \gamma(N^\perp,M^\perp) = \gamma(M,N).
\end{equation}

The minimal gap can also be computed using  duality between subspaces 
of complementary dimension. Let $M\subset X$, $\Xi \subset X^*$, 
such that $\dim(M)= d$ and $\dim(\Xi)=d$. Define 
\begin{equation}
\langle \Xi | M \rangle := \sup \big\{ \det([\langle \xi_i | u_j \rangle]_{1 \leq i,j \leq d}) : 
\xi_i \in \Xi, \ u_j \in M, \ \|\xi_i\|=\|u_j\|=1 \big\} .
\end{equation}
Notice that 
\begin{equation*}
\Sigma_d(X) = \sup \{ \langle \Xi | M \rangle : M\subset X, \ \Xi \subset X^*, \ \dim(M)= \dim(\Xi)=d \}.
\end{equation*}

\begin{lemma}
Let $X$ be a Banach space, $d\geq1$, $M$ and $N$ be two 
closed subspaces such that $X=M \oplus N$ and $\dim M=d$. 
Define $K_d := \bar\Delta_d(X)^{2d}$ and 
$K'_d := \bar\Delta_2(X)^{3d^2} \bar\Delta_d(X)^{2d}$.
Then 
\[
(K'_d)^{-1} \gamma(M,N)^d \leq \langle N^\perp | M \rangle \leq K_d \, \gamma(M,N).
\]
\end{lemma}

The topology on the Grassmannian space $\Grass(d,X)$ and 
coGrassmannian space $\coGrass(d,X)$ is given by a fundamental 
system of open neighborhoods. 

\begin{definition}
Let $X$ be a Banach space and $V_0$ be a subspace of $X$ of finite 
dimension or codimension. The {\it basic neighborhood complementary to $V_0$} is
the subset
\begin{gather*}
\mathcal{N}(V_0) = \{ U \subset X : U \ \text{is a closed 
subspace and} \  X = U \oplus V_0 \ \text{is topological} \}.
\end{gather*}
\end{definition}

The set $\{ \mathcal{N}(V_0) :  \codim(V_0)=d \}$ defines a topology 
of $\Grass(d,X)$; similarly the set $\{ \mathcal{N}(U_0) :  \dim(U_0)=d \}$  
defines a topology of $\coGrass(d,X)$.

Each basic neighborhood is modeled on a Banach space. The following construction 
shows that   $\mathcal{N}(U_0)$ is bijectively mapped  to $\mathcal{B}(V_0,U_0)$. 

\begin{definition} \label{definition:Graph}
Let $X=U_0 \oplus V_0$ be a topological splitting of closed subspaces. 
\begin{enumerate}
\item If $\Theta \in \mathcal{B}(V_0,U_0)$, the {\it graph} of $\Theta$ is  the closed subspace
\[
\Graph(\Theta) :=  \{ v+ \Theta v : v \in V_0 \} \in \mathcal{N}(U_0).
\]
\item Conversely every $V \in \mathcal{N}(U_0)$ is the graph of
some operator $\Theta \in \mathcal{B}(V_0,U_0)$.  
\end{enumerate}
\end{definition}

 Notice that $V \in \mathcal{N}(U_0) $ if and only if  $V^\perp = 
 \Graph(\Theta^\perp) \in \mathcal{N}(U_0^\perp)$ for some 
 $\Theta^\perp \in \mathcal{B}(V_0^\perp,U_0^\perp)$.

\begin{lemma} \label{lemma:GraphcoGraph}
Let $X$ be a Banach space, $d\geq1$, and $X= U_0 \oplus V_0$ be a
splitting of closed subspaces of $X$ where $\dim(U_0)=d$. Assume $U_0 = 
\Vect(u_1,\ldots,u_d)$ and $V_0 = \Vect(\eta_1,\ldots,\eta_d)^\Perp$.  
Let $V \in \mathcal{N}(U_0)$,  $\Theta \in \mathcal{B}(V_0,U_0)$ 
such that $V = \Graph(\Theta)$, and $\Theta^\perp \in 
\mathcal{B}(V_0^\perp, U_0^\perp)$ such that 
$V^\perp = \Graph(\Theta^\perp)$. Then
\begin{itemize}
\item $\forall v \in V, \ \Theta(v) = -\sum_{i=1}^d \langle \Theta^\perp\eta_i | v\rangle u_i$,
\item $\Theta^\perp = -\pi_{V_0|U_0}^* \circ \Theta^* \circ \rho_{U_0}^*$
\end{itemize}
where $\rho_{U_0} : U_0 \to X$ is the canonical injection.
\end{lemma}

In the following lemma, we show that the norm of $\Id \oplus \Theta$ 
and the minimal gap $\gamma(U,V_0)$ are inverse proportional. We interpret  
\begin{equation}
\Id \oplus \Theta : U_0 \to U = \Graph(\Theta), \quad \Theta \in 
\mathcal{B}(U_0,V_0), \label{equation:canonicalIsomorphism}
\end{equation}
as an   isomorphism between $U_0$ and $U$  and call it the 
{\it canonical isomorphism between $U_0$ and $U$ parallel to $V_0$}. 
Notice that $(\Id \oplus \Theta)^{-1} = (\pi_{U_0 | V_0}  | U)$.

\begin{lemma} \label{lemma:comparisonAngleGraph}
Let $X$ be a Banach space and $X=U_0 \oplus V_0$ be a topological 
splitting  of $X$ of  subspaces of finite dimension or codimension. 
Then for every $U \in \mathcal{N}(V_0)$ and $\Theta \in 
\mathcal{B}(U_0,V_0)$ such that $U = \Graph(\Theta)$,
\[
\gamma(U_0,V_0) \leq \gamma(U,V_0) \|\Id \oplus \Theta \| \leq 1.
\] 
\end{lemma}

The following lemma shows that the maximal gap between two 
subspaces $U$ and $U'$ of $\mathcal{N}(V_0)$ sufficiently close 
to some fixed $U_0 \in \mathcal{N}(V_0)$  is equivalent to the distance $\| \Theta - \Theta'\|$.

\begin{lemma} \label{lemma:ConvergenceGraph}
Let $X$ be a Banach space, $X=U_0 \oplus V_0$ be a topological  
direct sum of  subspaces of $X$ of finite dimension or codimension. 
For every  $\Theta,\Theta' \in \mathcal{B}(U_0,V_0)$  define 
$ U := \Graph(\Theta)$ and $U' := \Graph(\Theta')$. Then
\begin{enumerate}
\item  \label{item:ConvergenceGraph_2}  if $\delta(U,U_0) < 
\gamma(V_0,U_0)$, then \quad $\| \Theta \| 
\leq \dfrac{\delta(U,U_0)}{\gamma(V_0,U_0)-\delta(U,U_0)}$,
\item  \label{item:ConvergenceGraph_3} \label{item:{lemma:ConvergenceGraph}_4} 
if $\delta(U,U_0) < \gamma(V_0,U_0)$ and $\delta(U',U) < \gamma(V_0,U)$, then 
\[
\| \Theta' - \Theta \| \leq \Big[ \frac{\gamma(V_0,U_0)}{\gamma(V_0,U_0)-\delta(U,U_0)}\Big] \frac{\delta(U',U)}{\gamma(V_0,U)-\delta(U',U)},
\]
\item \label{item:ConvergenceGraph_1} $\delta(U_0,U) \leq 
\| \Theta \|$, \quad $\Big[↓1+\dfrac{\delta(U,U_0)}{\gamma(V_0,U_0)} \Big] ^{-1}\delta(U,U') 
\leq 
\| \Theta-\Theta' \|$.
\end{enumerate}
\end{lemma}

Let  $X=U_0 \oplus V_0=U \oplus V$ be two splittings of $X$ by closed 
subspaces where  $\dim (U_0)=d$ and $\dim(U)=d$. Assume $U_0 \in \mathcal{N}(V)$ or 
$U \in \mathcal{N}(V_0)$. The following lemma shows that  the minimal 
gap $\gamma(U_0,V)$ or $\gamma(U,V_0)$ can be measured by a 
$d$-dimensional determinant adapted to $(V^\perp, U_0)$ or 
$(V_0^\perp,U)$ that are both of dimension $d$.

\begin{lemma} \label{lemma:boundFromBelow}
Let $X$ be a Banach space, $d\geq1$,  $C_0\geq1$, and $X=U_0 \oplus V_0$ 
be a $C_0$-orthogonal splitting  with $\dim U_0=d$.  Let 
$(e_1,\ldots,e_d)$ and $(\phi_1,\ldots,\phi_d)$ be $C_0$-Auerbach bases  
dual to each other generating $U_0$ and $V_0^\perp$. Let 
$K_d := \bar\Delta_d(X)^{2d}$. 
\begin{enumerate}
\item \label{item1:boundFromBelow} 
Let $\Theta^\perp \in \mathcal{B}(V_0^\perp,U_0^\perp)$, 
$\|\Theta^\perp\|\leq 1$, $V=\Graph(\Theta^\perp)^\Perp$  and 
$(\psi_1,\ldots,\psi_d)$  be a $C$-Auerbach basis of $V^\perp$. Then
\[
(C_0C)^d \langle V^\perp | U_0 \rangle 
\geq \big| \det([\langle \psi_i | e_j \rangle]_{ij}) \big| \geq \frac{1}{K_d} 
\Big( \frac{1-\|\Theta^\perp\|}{C_0 C} \Big)^d.
\]
\item \label{item2:boundFromBelow}  
Let $\Theta \in \mathcal{B}(U_0,V_0)$, $\| \Theta \|  \leq 1$,  
$U = \Graph(\Theta)$ and  $(f_1,\ldots,f_d)$ be a $C$-Auerbach basis of $U$. Then
\[
(C_0C)^d \langle V_0^\perp | U \rangle \geq \big| \det([\langle \phi_i | f_j \rangle]_{ij})  \big|
\geq \frac{1}{K_d} \Big( \frac{1-\|\Theta\|}{C_0 C} \Big)^d.
\]
\end{enumerate}
\end{lemma}

\subsection{Singular values decomposition} \label{appendix:SVD}

The notion of singular values for operators in Banach spaces is not canonically 
well-defined. Our starting definition is the following.

\begin{definition}
Let $X,Y$ be  Banach spaces, $A\in\mathcal{B}(X,Y)$, and $d\geq1$. 
We define the {\it singular value of $A$ of index $d$} by
\begin{equation*}
\sigma_d(A) := \sup_{\dim(U)=d} \inf 
\Big\{ \frac{\|Aw\|}{\|w\|} \,:\, w \in U \setminus \{0\} \Big\},
\end{equation*}
where the supremum is realized over every subspace $U$ of $X$ of dimension $d$.
\end{definition}

We recall some elementary properties.
\begin{lemma} \label{lemma:singularElementaryProperty}
Let $X,Y$ be  Banach spaces, $A \in \mathcal{B}(X,Y)$, and $d\geq1$. Then
\begin{enumerate}
\item $\sigma_d(A) \geq \sigma_{d+1}(A)$,
\item $\sigma_d(AB) \leq \| A \| \sigma_d(B)$, \ $\sigma_d(AB) \leq  \sigma_d(A) \| B \|$,
\item $\sigma_d(A)>0 \ \ \text{and} \ \ \sigma_{d+1}(A)=0 
\Longleftrightarrow \codim(\ker(A))=d$.
\end{enumerate}
\end{lemma}

Another definition could be used instead of $\sigma_d(A)$. It coincides 
with the first one when $X$ and $Y$ are Hilbert spaces.

\begin{definition} \label{definition:SingularValueBis}
Let $A\in \mathcal{B}(X,Y)$. For every $d\geq1$, define
\begin{equation*}
\sigma'_d(A) := \inf_{\codim(V)=d-1} \sup \Big\{ \frac{\|Aw\|}{\|w\|} 
\,:\, w \in V\backslash\{0\} \Big\},
\end{equation*}
where  the infimum is realized over every closed subspace $V$ of codimension $d-1$.
\end{definition}

It will be convenient to introduce a third notion of singular values using the notion of Jacobian.

\begin{definition} \label{definition:Jacobian}
Let $A\in \mathcal{B}(X,Y)$. The \emph{Jacobian of $A$ of index $d$} is defined by, 
\begin{equation*}
\quad \Sigma_d(A) := \sup \big\{ \det 
\big( [ \langle \zeta_i | Au_j \rangle ]_{1 \leq i,j \leq d} \big) \,:
\,  \zeta_i \in Y^*, \ u_j \in X, \  \|\zeta_i \| = \|u_j \| = 1  \big\},\quad
\end{equation*}
By convention $\Sigma_0(A)=1$.  Notice that, if $\dim(U)=d$, 
\[
\Sigma_d(A|U) = 0 \ \Leftrightarrow \ \dim(AU)<d
\ \Leftrightarrow A \ \ \text{is not injective on $U$}.
\]
\end{definition}

We may choose in the previous definition $\tilde\eta_i \in \overline{\range(A)}^{\ *}$ and take $\zeta_i$ an extension of $\tilde\eta_i$ to $Y^*$ by the Hahn-Banach theorem. 
If $U$ is a closed subspace of $X$, we define the \emph{Jacobian of $A$ restricted to 
$U$ of index $d$}, denoted $\Sigma_d(A  |  U)$,  to be
the Jacobian of $ A  |  U \in \mathcal{B}(U,Y)$. If $U$ has finite dimension and 
$A|U$ is injective,  the supremum is attained by vectors $u_j \in U$ and linear 
forms $\tilde\eta_i \in {\tilde U}^*$, $\tilde U = AU$, of norm one. 
Both $(u_1,\ldots,u_d)$ and $(\tilde \eta_1,\ldots,\tilde \eta_d)$ are Auerbach bases 
by lemma \ref{lemma:existenceAuerbachBases}.

The third definition of singular values is based on the notion of Jacobian. 
\begin{definition} \label{definition:singularValueTer}
Let $A\in \mathcal{B}(X,Y)$, define (assuming by convention  $\Sigma_0(A)=1$),
\begin{equation*}
\sigma''_d(A) := \frac{\Sigma_d(A)}{\Sigma_{d-1}(A)} 
\ \ \text{if} \  \ \Sigma_{d-1}(A) \not= 0, \quad 
\sigma''_{d}(A)=0 \ \ \text{if} \ \ \Sigma_{d-1}(A)=0.
\end{equation*}
If $U$ is a closed subspace of $X$, we define similarly  
$\sigma''_d(A  |  U)$ of the restriction of $(A  |  U) \in \mathcal{B}(U,Y)$.
\end{definition}

The three definitions $\sigma_d(A)$, $\sigma'_d(A)$ and 
$\sigma''_d(A)$ are comparable in Banach spaces, and equal in Hilbert spaces.

\begin{proposition} \label{proposition:comparisonSingularValues}
Let $X$, $Y$ be Banach spaces, $d\geq1$, and  
$K_d := [\Delta_d(Y^*)\Delta_d(X)]^d$. Then for every $A \in \mathcal{B}(X,Y)$,
\[
\sigma_d(A) \leq \sigma'_d(A) \leq \sigma''_d(A) \leq  K_d \ \sigma_d(A).
\]
\end{proposition}

It may not be true that the singular values of $A$ and $A^*$ coincide. On the other 
hand the Jacobian admits a very symmetric definition using the identity  
\[
\langle \tilde \eta | Au \rangle = \langle A^*\tilde \eta | u \rangle, \quad \forall u \in X, \ \forall \tilde\eta \in Y^*.
\]

Proposition \ref{proposition:comparisonSingularValues} and the 
following proposition shows that $\sigma_d(A)$ and $\sigma_d(A^*)$ 
are comparable modulo a constant depending only on the Banach norm of $X$. 
This constant is 1 for Hilbert spaces.

\begin{proposition}
Let $X,Y$ be Banach spaces,  $A\in\mathcal{B}(X,Y)$, $d\geq1$, and  $K_d := 
\max(\bar\Delta_d(X),\bar\Delta_d(Y))^{2d}$. Then
\begin{enumerate}
\item $\Sigma_d(A) = \Sigma_d(A^*)$,
\item $K_d^{-1}\sigma_d(A) \leq \sigma_d(A^*) \leq K_d \sigma_d(A)$.
\end{enumerate}
\end{proposition}

The following lemma shows that the projective distortion $\Sigma_d(X)$, equation 
\eqref{equation:projectiveDistortion},   may not be equal to one and that the 
Jacobian may not be multiplicative. This anomaly disappears when the spaces are Hilbert.

\begin{proposition} \label{proposition:MultiplicativityJacobian}
Let $X,Y,Z$ be Banach spaces, $A \in  \mathcal{B}(X,Y)$, 
$B\in \mathcal{B}(Y,Z)$, $d\geq1$, and $K_d :=  \bar\Delta_d(X)^d$. Then
\begin{enumerate}
\item \label{item1:MultiplicativityJacobian} $1 \leq \Sigma_d(X) \leq K_d$,
\item \label{item2:MultiplicativityJacobian} 
$\Sigma_d(BA) \leq \Sigma_d(B) \Sigma_d(A)$,
\item if $U$ is  a subspace of dimension $d$, $\Sigma_d(B|AU) \Sigma_d(A|U) \leq \Sigma_d(X) \Sigma_d(BA)$.
\end{enumerate}
In the case  $X,Y$ are  Hilbert spaces, the previous inequalities are equalities.
\end{proposition}

The following theorem is the main result of this appendix. The existence of 
singular vectors depends on a small parameter $\epsilon>0$ that can be as 
small as we want. We do not assume that the operators are compact nor
asymptotically compact, and there is thus no reason to find true eigenvectors even in Hilbert spaces. 
The parameter $\epsilon$ measures the discrepancy between a true and an 
approximate eigenvector. The estimates depend moreover in Banach spaces 
on the volume distortion introduced   in the definition \ref{definition:volumeDistortion}. Although the following result is certainly well known to specialists, we did not find a good reference adapted to our needs.

\begin{theorem}[Approximate singular value decomposition] \label{theorem:PolarDecomposition}
Let $X,Y$ be Banach spaces, $A\in\mathcal{B}(X,Y)$, and $d\geq1$. 
Assume $\sigma_d(A)>0$ and choose $\epsilon>0$. Define
\begin{gather*}
\Delta_d= \max(\bar\Delta_d(X),\bar\Delta_d(Y)), \quad
C_{\epsilon,d}(X,Y) := (1+\epsilon) \Delta_d^{6d^2+15d+4} \ \Delta_2^{3d^2+4d+4}.
\end{gather*}
 Then $A$ admits an approximate singular value decomposition of index 
 $d$ and distortion $C_{\epsilon,d}=C_{\epsilon,d}(X,Y)$,  defined in the following way:
\begin{itemize}
\item there exist two $C_{\epsilon,d}$-orthogonal splittings 
$X = U \oplus V, \ Y = \tilde U \oplus \tilde V$,
\item there exist   $C_{\epsilon,d}$-Auerbach bases,  $(e_1,\ldots,e_d)$  
of $U$ and $(\phi_1,\ldots,\phi_d)$ of  $V^\perp$  dual to each over, 
such that $U = \Vect(e_1,\ldots,e_d)$ and $V =\Vect(\phi_1,\ldots,\phi_d)^\Perp$,
\item  there exist  $C_{\epsilon,d}$-Auerbach bases, $(\tilde e_1,\ldots,\tilde e_d)$ 
of $\tilde U$ and $(\tilde \phi_1,\ldots,\tilde \phi_d)$ of  $\tilde V^\perp$  
dual to each over, such that $\tilde U = \Vect(\tilde e_1,\ldots,\tilde e_d)$ 
and  ${\tilde V} = \Vect(\tilde \phi_1,\ldots,\tilde \phi_d)^\Perp$,
\end{itemize}
satisfying the following properties, for every $i=1, \ldots, d$,
\begin{enumerate}
\item \label{item1:PolarDecomposition} $AU=\tilde U$, 
$AV\subset \tilde V $, $A^* \tilde V^\perp = V^\perp$, 
$A^* \tilde U^\perp \subset U^\perp$, \ $\dim(U)= \dim(\tilde U)=d$,
\item \label{item2:PolarDecomposition}   
$Ae_i = \sigma_i(A) \tilde e_i$, \ \  $A^*\tilde \phi_i = \sigma_i(A) \phi_i$,
\item \label{item3:PolarDecomposition} 
$C_{\epsilon,d}^{-1} \sigma_i(A) \leq \sigma_i(A  |  U) \leq \sigma_i(A) $,
\item \label{item4:PolarDecomposition} 
$C_{\epsilon,d}^{-1} \sigma_i(A)  \leq \sigma_i(A^*  |  \tilde V^\perp) \leq \sigma_{i}(A)$,
\item \label{item5:PolarDecomposition} 
$\sigma_{d+1}(A) \leq \| A  |  V\| \leq C_{\epsilon,d} \, \sigma_{d+1}(A) $
\item \label{item6:PolarDecomposition} 
$\sigma_{d+1}(A) \leq \| A^* | \tilde U^\perp \| \leq C_{\epsilon,d} \, \sigma_{d+1}(A) $,
\item \label{item7:PolarDecomposition} 
$\gamma(U,V), \ \gamma(V,U), \ \gamma(\tilde U,\tilde V), 
\ \gamma(\tilde V, \tilde U) \geq C_{\epsilon,d}^{-1}$.
\end{enumerate}
If $X$ is a Hilbert space, one may choose  $C_{\epsilon,d}=1+\epsilon$. If $X,Y$ are of 
finite dimension, one may choose   $\epsilon=0$. If $X,Y$ are Hilbert spaces of finite 
dimension, one may choose $V=U^\perp$, $\tilde V = \tilde U^\perp$,  
$C_{\epsilon,d}=1$, $e_i=\phi_i$, $\tilde e_i = \tilde \phi_i$, 
$(e_1,\ldots,e_d)$ and $(\tilde e_1,\ldots,\tilde e_d)$ are orthonormal bases.
\end{theorem}

\subsection{Exterior product} \label{appendix:ExteriorProduct}

The algebraic exterior product $\bigwedge^d X$ is defined canonically of the
following procedure. We first consider the space  of almost null functions of 
$X^d \to \mathbb{R}$,  
\[
\mathcal{F} := \Big\{ \sum _{w\in X^d} \lambda_w \delta_w : 
\lambda_w \in \mathbb{R}, \ \card\{ w : \lambda_w \not=0 \} < +\infty \Big\}
\]
where $\delta_w : X^d \to \mathbb{R}$ is the Dirac function at $w\in X^d$. 
We next consider the subspace $\mathcal{G}$ of $\mathcal{F}$ defined by 
\begin{multline*}
\mathcal{G} := \Vect \Big\{ \delta_{(\lambda w_1+\mu w'_1,w_2,\ldots,w_d)} - 
\lambda \delta_{(w_1,w_2,\ldots,w_d)} - \mu \delta_{(w'_1,w_2,\ldots,w_d)}, \\
\delta_{(w_1,\ldots,w_{i-1},w'_i,w'_{i+1},w_{i+2}, \ldots,x_d)} + 
\delta_{(w_1,\ldots,w_{i-1},w'_{i+1},w'_i,w_{i+2}, \ldots,w_d)} : \\
1 \leq i \leq d-1, \  w_1,\ldots,w_d,w'_1,\ldots,w'_d  \in X^d,
\ \lambda,\mu \in \mathbb{R} \Big\}.
\end{multline*}
The {\it algebraic exterior product} the  vector space of equivalent classes 
\begin{equation*}
\textstyle{\bigwedge}^d X := \mathcal{F}/\mathcal{G} = \{w+\mathcal{G} : w \in \mathcal{F} \}
\end{equation*}
We define the {\it canonical injection} $X^d \to \bigwedge^dX$ into the quotient space  by
\[
 (w_1,\ldots,w_d) \in X^d \mapsto w_1\wedge\ldots\wedge 
 w_d := \delta_{(w_1,\ldots,w_d)}+\mathcal{G} \in \textstyle{\bigwedge}^d X 
\]
It is then easy to check 
that $\bigwedge^d X$ is spanned by {\it simple vectors}, vectors of the 
form $w_1\wedge\ldots\wedge w_d$.  The canonical  map 
$(w_1,\ldots,w_d) \mapsto w_1\wedge\ldots\wedge w_d$ is 
multilinear alternating, and its image generates $\bigwedge^d X$. Moreover 
$\bigwedge^d X$ satisfies the universal property: every multilinear and 
alternating function $f : X^d \to Y$, where $Y$ is any vector space, 
factorizes uniquely through a linear map  $F : \bigwedge^d X \to Y$ by 
$F(w_1\wedge\ldots\wedge w_d) = f(w_1,\ldots,w_d)$.

Several norms may be chosen for the exterior  product. In the case where $X$ is a
Banach space, we choose the projective norm defined in the following way.  Every 
$w\in \bigwedge^d X$ is a finite sum of vectors of the form 
$w_1^\alpha\wedge\ldots\wedge w_d^\alpha$ where $\alpha$ is an index. As this representation is not 
unique, we introduce the {\it projective norm} of $\|w\|$ defined by
\begin{equation} \label{equation:ProjectiveNorm}
\|w\| := \inf \big\{ \sum_\alpha \textstyle{\prod_{i=1}^d} \|w_i^\alpha\| : w = \sum_\alpha
w_1^\alpha\wedge\ldots\wedge w_d^\alpha \big\}.
\end{equation}
It is easy to check that $\| \cdot \|$ is a genuine norm: 
$w\not=0 \Rightarrow \|w\| \not=0$. In the case $X$ is a Hilbert space, 
we choose instead the {\it Euclidean norm} associated to the scalar product 
defined by extending by bilinearity to $\bigwedge^d X \times \bigwedge^d X$
\[
\langle w_1\wedge\ldots\wedge w_d | w'_1 \wedge \ldots \wedge w'_d \rangle
:= \det([\langle w_i | w'_j  \rangle ]_{1\leq i,j \leq d} ).
\]
The projective norm and the Euclidean norm are not equal in general when $X$ is a
Hilbert space. We call the completion of the algebraic exterior product with 
respect to the chosen norm, the {\it normed exterior product},  and we denote it by  
$\bigwedge^d X$. We point out that $\bigwedge^d (X^*)$ 
denotes the normed exterior product of $X^*$ and not the dual of 
$\bigwedge^d X$. If  $X$ is a Hilbert space, $X^* = X$ and 
$\bigwedge^d (X^*) = \bigwedge^d X = (\bigwedge^d X)^*$.

We define a {\it canonical duality} between $\bigwedge^d (X^*)$ and 
$\bigwedge^d X$ by extending by linearity for every $\theta_i \in X^*$ and $w_j \in X$,
\begin{equation} \label{equation:definitionDualityWedgeProduct}
\langle \theta_1\wedge\ldots\wedge\theta_d | w_1\wedge\ldots\wedge w_d \rangle
:= \det \big([\langle \theta_i | w_j \rangle]_{1\leq i,j \leq d} \big).
\end{equation}
We notice that the canonical linear map $\bigwedge^d( X^*) \to (\bigwedge^d X)^*$ 
is injective but may have a norm $\Sigma_d(X)$  greater than one 
(see \ref{proposition:MultiplicativityJacobian} for a bound from 
above of $\Sigma_d(X)$),
\begin{equation} \label{equation:normDualityWedgeProduct}
\begin{split}
&\forall \theta \in \textstyle{\bigwedge}^d (X^*), \ \forall w \in 
\textstyle{\bigwedge}^d X, \quad  | \langle \theta | w \rangle | \leq 
\Sigma_d(X)  \|\theta\| \|w\|, \\
&\forall w_j \in X, \quad \sup_{\|\theta_i\|=1} \langle 
\textstyle{\bigwedge}_{i=1}^d \theta_i | 
\textstyle{\bigwedge}_{j=1}^d w_j \rangle \geq  \| \textstyle{\bigwedge}_{j=1}^d w_j \|.
\end{split}
\end{equation}
In particular, for every Auerbach family $(u_1,\ldots,u_d)$ of $X$,
\begin{equation}
 \Sigma_d(X)^{-1} \leq \|u_1 \wedge\ldots\wedge u_d \| \leq 1.
\end{equation}

Let  $(u_1,\ldots,u_d)$ be a linearly independent family of $X$, 
$U = \Vect(u_1,\ldots,u_d)$, and $1\leq r\leq d$. For every sequence 
$I=(i_1,\ldots,i_r)$ of $r$ ordered elements  in $\{1,\ldots,d\}$, we denote   
$u_I := u_{i_1}\wedge \ldots \wedge u_{i_r}$. Then $\{u_I\}_I$ is a basis of 
$\bigwedge^r X$ spanning   $\bigwedge^r U$. The following lemma gives an 
estimate on the  volume distortion of this basis in $\bigwedge^r X$.

\begin{lemma} \label{lemma:basisVolumicDistortion}
Let $X$ be a Banach space, $1 \leq r \leq d$,  $(u_1,\ldots,u_d)$ be a 
$C$-Auerbach family of $X$ dual to a $C$-Auerbach family $(\eta_1,\ldots,\eta_d)$ 
of $X^*$. Then $\{u_I\}_I$ and $\{\eta_I\}_I$  are a $C^r\Sigma_r(X)$-Auerbach 
families dual to each other of $\bigwedge^r X$ and $\bigwedge^r X^*$ respectively.
\end{lemma}

Let $0 \leq r \leq d$. We denote by $(w,w') \in 
\bigwedge^r X \times \bigwedge^{d-r} X \mapsto w\wedge w' \in \bigwedge^d X$ 
the canonical  bilinear map extending 
\[
(w_1 \wedge \ldots \wedge w_r )\wedge ( w_{r+1} \wedge \ldots \wedge w_d) 
= w_1 \wedge\ldots\wedge w_d.
\]
\begin{lemma}
If $X$ is a Banach space and $\| \cdot \|$ is the projective norm, or if $X$ is a
Hilbert space and $\| \cdot \|$ is the Euclidean norm, then for every $0 \leq r \leq d$
\[
\forall w \in \textstyle{\bigwedge}^r X, \ \forall w' \in \textstyle{\bigwedge}^{d-r} X, 
\quad \|w \wedge w' \| \leq \| w\| \|w'\|.
\]
\end{lemma}

The following lemma extends the volume distortion estimate of lemma 
\ref{lemma:basisVolumicDistortion}.

\begin{lemma}
Let $X$ be a Banach space, $d\geq1$, $C\geq1$, $X=U \oplus V$ be a 
$C$-orthogonal splitting of closed subspaces with $\dim(U)=d$. 
Let $(u_1,\ldots,u_d)$ and $(\eta_1,\ldots,\eta_d)$ be $C$-Auerbach 
bases dual to each other spanning $U$ and $V^\perp$. Let $V'\subset V$ 
be a subspace of $V$ of dimension $d'\geq0$ and $X' := U \oplus V'$. Define 
\[
K_d := \Sigma_d(X)\bar\Delta_{\binom{d+d'}{d}}({\textstyle{\bigwedge}}^d X)^2 
\max_{0 \leq r \leq d'} \Big(\Sigma_r(X) \bar \Delta_{\binom{d'}{r}}
({\textstyle{\bigwedge}}^r X)^2 \Big) \bar\Delta_2(X)^{8d}\bar\Delta_d(X)^{8d}. 
\]
Then every $w \in \bigwedge^d X'$ admits a unique decomposition  
$w=\sum_I u_I \wedge v_I$ where the summation is realized over every ordered  
sequence $I = (i_1,\ldots,i_r)$ of $\{1,\ldots,d\}$, 
$u_I = u_{i_1}\wedge \cdots \wedge u_{i_r}$,  $v_I \in \bigwedge^{d-r} V'$ is any vector, and 
$0\leq r \leq d$. Moreover
\[
C^{-2d} K_d^{-1}  \big( \sum_I \| v_I \|^2 \big)^{1/2} \leq \|w\| \leq C^{2d} 
K_d  \big( \sum_I \| v_I \|^2 \big)^{1/2}.
\]
\end{lemma}

Non-zero simple vectors in $\bigwedge^d X$ are in one-to-one correspondence 
with subspaces of $X$ of dimension $d$. We introduce the following notations 
to clarify this correspondence.

\begin{definition} \label{definition:extendedSplitting}
Let $X$ be a vector space and $d\geq1$.
\begin{enumerate}
\item If $U$ is a subspace of $X$ of dimension $d$, we call
\[
\widehat U := \Vect \{ {\textstyle\bigwedge}_{i=1}^dw_i : \forall\ i, \ w_i \in U \} \subset \textstyle{\bigwedge}^dX.
\]
\item If $V$ is a subspace of codimension $d$, we call
\[
\widecheck V := \Vect\{ {\textstyle \bigwedge_{i=1}^d}w_i : \exists\ i, \ w_i \in V, \ \forall i, \  w_i \in X \}  \subset \textstyle{\bigwedge}^dX.
\]
\end{enumerate}
Then $\dim(\widehat U) = 1$ and $\codim(\widecheck V)=1$.
\end{definition}
\noindent If $X = U \oplus V$  with $\dim(U)=d$, then 
$\bigwedge^d X = \widehat U \oplus \widecheck V$.
If $(\eta_1,\ldots,\eta_d)$ are linearly independent and  
$V = \Vect(\eta_1,\ldots,\eta_d)^\Perp$, then $\widecheck V$ is the kernel of a 
simple linear form of $\bigwedge^d X$,
\[
\widecheck V = \{ w \in {\textstyle \bigwedge}^d X : \langle \eta_1 \wedge 
\ldots \wedge \eta_d | w \rangle = 0 \} = 
\Vect ( {\textstyle \bigwedge}^d_{i=1}\eta_i )^\Perp.
\]
 
The following lemma compares the angle between $U$ and $V$ and the angle 
between $\widehat U$ and $\widecheck V$. Using  equation 
\eqref{equation:comparisonAngleNorm}, we also obtain a comparison 
between $\| \pi_{U | V} \|$ and $\| \pi_{\widehat U | \widecheck V} \|$, (see \eqref{equation:canonicalProjector} for the definition of $\pi_{U|V}$).

\begin{lemma} \label{lemma:comparisonAngleExteriorProduct}
Let $X$ be a Banach space, $d\geq1$, $X=U \oplus V$ be a splitting of closed 
subspaces with $\dim(U)=d$ and $K_d := \bar\Delta_2(X)^4\bar\Delta_d(X)^3$. 
Then  $\bigwedge^d X = \widehat U \oplus \widecheck V$ and
\begin{gather*}
K_d^{-d} \gamma(\widehat U,\widecheck V) \leq \gamma(U,V) \leq 
K_d \gamma(\widehat U,\widecheck V)^{1/d}, \\
K_d^{-1} \| \pi_{\widehat U | \widecheck V} \|^{1/d} \leq  
\| \pi_{U | V} \| \leq K_d^d \| \pi_{\widehat U | \widecheck V} \|.
\end{gather*}
\end{lemma}

In the case the splitting $X = U \oplus V$ is $C$-orthogonal, using lemma 
\ref{lemma:projectorDistortionBasic}, the norm of the two projectors admits a simpler estimate.

\begin{lemma}
Let $X$ be a Banach space, $d\geq1$, $C\geq1$,   $X = U \oplus V$ be a 
$C$-orthogonal splitting with $\dim U = d$ and  $K_d := 
\bar \Delta_2(\bigwedge^d X)^3$. Then
\[
\| \pi_{\widehat U | \widecheck V} \| \leq \Sigma_d(X) C^{2d}, 
\ \ \text{and} \ \ \| \pi_{\widecheck V | \widehat U} \| \leq \Sigma_d(X) K_d C^{2d}.
\]
\end{lemma}

Angles between subspaces can also be measured by the norm of some graphs 
over a reference splitting as  in lemma \ref{lemma:comparisonAngleGraph}. 
Consider a splitting $X=U_0 \oplus V_0$ with $\dim(U_0)=d$ and a 
subspace $V \in \mathcal{N}(U_0)$. Then $V= \Graph(\Theta)$ for 
some operator $\Theta \in \mathcal{B}(V_0,U_0)$ or equivalently, 
as explained in lemma \ref{lemma:GraphcoGraph}, 
$V^\perp = \Graph(\Theta^\perp)$ for some 
$\Theta^\perp \in \mathcal{B}(V_0^\perp,U_0^\perp)$. Lemma 
\ref{lemma:comparisonAngleExteriorProduct} implies 
\[
{\textstyle \bigwedge}^d X = \widehat U_0 \oplus \widecheck V_0 = 
\widehat U_0 \oplus \widecheck V,
\] 
and in particular $\widecheck V \in \mathcal{N}(\hat U_0)$  is equal to 
the graph of some $\widehat \Theta  \in \mathcal{B}(\widecheck V_0, \widehat U_0)$. 
The following lemma gives an estimate of $\| \widehat \Theta \|$ with 
respect to $\| \Theta^\perp \|$.

\begin{lemma} \label{lemma:normExtendedGraph}
Let $X$ be a Banach space, $d\geq1$, $C\geq1$, and 
$X = U_0 \oplus V_0$ be a $C$-orthogonal splitting of closed subspaces 
with $\dim(U_0)=d$. Let $(u_1,\ldots,u_d)$ and $(\eta_1,\ldots,\eta_d)$ 
be $C$-Auerbach families in $X$ and $X^*$ respectively, dual to each over, 
such that $U_0 = \Vect(u_1,\ldots,u_d)$ and 
$V_0 = \Vect(\eta_1,\ldots,\eta_d)^\Perp$. 

Let  
$\Theta^\perp \in \mathcal{B}(V_0^\perp,U_0^\perp)$ and   
$V = \Graph(\Theta^\perp)^\Perp$. Then 
\begin{itemize}
\item $\widecheck V = 
\Vect({\textstyle \bigwedge}^d_{i=1}(\Id \oplus \Theta^\perp)\eta_i)^\Perp 
= \Graph(\widehat\Theta)$ for some 
$\widehat\Theta \in \mathcal{B}(\widecheck V_0,\widehat U_0)$,
\item $\forall w \in \widecheck V_0, \quad \widehat \Theta (w) = -\langle 
\textstyle{\bigwedge}_{i=1}^d (\eta_i + \Theta^\perp \eta_i) | w \rangle 
\textstyle{\bigwedge}_ {i=1}^d u_i$,
\item $\| \widehat \Theta \| \leq C^{2d} \Sigma_d(X)
\| \Theta^\perp \| (1+\| \Theta^\perp\|)^{d-1}$.
\end{itemize}
\end{lemma}

The next theorem shows that the approximate singular value decomposition 
of index $d$ of a bounded operator $A \in\mathcal{B}(X,Y)$ admits a 
particular form when the operator is considered in the exterior product. Let
\[
\widehat A := {\textstyle \bigwedge}^d A \in \mathcal{B}
({\textstyle \bigwedge}^d X, {\textstyle \bigwedge}^d Y).
\]

\begin{theorem} \label{theorem:codimensionOneSingularValueDecomposition}
Let $X,Y$ be Banach spaces, $d\geq1$, $\epsilon>0$, and
 $A\in\mathcal{B}(X,Y)$ satisfying $\sigma_d(A)>0$. Let 
 $X = U \oplus V$ and $Y = \tilde U \oplus \tilde V$, be the  
 approximate singular value decomposition of index $d$ and 
 distortion $C_{\epsilon,d}$ given  in theorem \ref{theorem:PolarDecomposition}. Let
\[
\widehat C_{\epsilon,d} := C_{\epsilon,d}^{17d}
\Sigma_d(X) (\bar\Delta_{\binom{2d}{d}}
({\textstyle{\bigwedge}}^d X))^2 \max_{0 \leq r \leq d} 
\Big(\Sigma_r(X) (\bar \Delta_{\binom{d}{r}}
({\textstyle{\bigwedge}}^r X))^2 \Big) 
\bar\Delta_2(X)^{24d}\bar\Delta_d(X)^{28d}. 
\]
Then

\begin{enumerate}
\item $({\textstyle \bigwedge}^d_{i=1}e_i)$ and 
$({\textstyle \bigwedge}^d_{i=1}\phi_i)$ are 
$\widehat C_{\epsilon,d}$-orthogonal bases dual to each over, 
\[
\widehat U = \Vect({\textstyle \bigwedge}^d_{i=1}e_i), 
\quad \widecheck V = \Vect({\textstyle \bigwedge}^d_{i=1}\phi_i)^\Perp,
\]
\item $({\textstyle \bigwedge}^d_{i=1}\tilde e_i)$ and 
$({\textstyle \bigwedge}^d_{i=1}\tilde \phi_i)$ are 
$\widehat C_{\epsilon,d}$-orthogonal bases dual to each over, 
\[
\widehat{\tilde U} = \Vect({\textstyle \bigwedge}^d_{i=1}\tilde e_i), 
\quad \widecheck{\tilde V} = \Vect({\textstyle \bigwedge}^d_{i=1}\tilde \phi_i)^\Perp,
\]
\item ${\textstyle \bigwedge}^d X = \widehat U \oplus \widecheck V$, 
\ ${\textstyle \bigwedge}^d Y  = \widehat{\tilde U} \oplus \widecheck{\tilde V}$, 
\ $\dim(\widehat U) = \dim(\widehat{\tilde U}) = 1$,
\item $\widehat A \widehat U = \widehat{\tilde U}$, 
\ $\widehat A \widecheck V \subset \widecheck{\tilde V}$, 
\  $\widehat A^* \widecheck{\tilde V}^\perp = 
\widecheck V^\perp$, \ $\widehat A^* \widehat{\tilde U}^\perp 
\subset \widehat U^\perp$,
\item $\widehat C_{\epsilon,d}^{-1} \, \prod_{i=1}^d \sigma_i(A) 
\leq  \|\widehat A | \widehat U\| \leq \| \widehat A \| 
\leq  \widehat C_{\epsilon,d} \, \prod_{i=1}^d \sigma_i(A)$,
\item $\widehat C_{\epsilon,d}^{-1} \, \prod_{i=1}^d \sigma_i(A) 
\leq  \| \widehat A^* | \widecheck{\tilde V}^\perp \| 
\leq \| \widehat A^* \|  \leq  \widehat C_{\epsilon,d} \, \prod_{i=1}^d \sigma_i(A)$,
\item $\sigma_2(\widehat A) \leq \| \widehat A | \widecheck V \| 
\leq \widehat C_{\epsilon,d}  \,  \sigma_1(A) \  
\cdots \ \sigma_{d-1}(A) \sigma_{d+1}(A)$,
\item $\sigma_2(\widehat A) \leq \| \widehat A^* | \widecheck{\tilde U}^\perp \|
 \leq \widehat C_{\epsilon,d}  \,  \sigma_1(A) 
 \ \cdots \ \sigma_{d-1}(A) \sigma_{d+1}(A)$,
\item $\gamma(\widehat U, \widecheck V) \geq \widehat C_{\epsilon,d}^{-1}$, 
\ $\gamma(\widecheck V, \widehat U) \geq \widehat C_{\epsilon,d}^{-1}$.
\end{enumerate}
\end{theorem}

In the following lemma we consider a product $BA$ of two operators and the 
relative position of the approximate singular value decomposition of $A$ and $B$.

\begin{lemma} \label{lemma:boundBelowFBI}
Let $X,Y,Z$ be three Banach spaces, $A\in \mathcal{B}(X,Y)$, 
$B\in\mathcal{B}(Y,Z)$, $d\geq1$, and $\epsilon >0$. Assume 
$\sigma_d(A)>0$ and $\sigma_d(B)>0$.  Let 
\[
{\textstyle \bigwedge}^dX = \widehat U_A \oplus \widecheck V_A, 
\ \quad {\textstyle \bigwedge}^d Y  = \widehat{\tilde U}_A \oplus 
\widecheck{\tilde V}_A = \widehat U_B \oplus \widecheck V_B, 
\quad {\textstyle \bigwedge}^d Z  = \widehat{\tilde U}_B \oplus \widecheck{\tilde V}_B,
\]
be the  two approximate singular value decompositions of index 1 and 
distortion $\widehat C_{\epsilon,d}$ of $\widehat A$ and $\widehat B$
obtained in  theorem \ref{theorem:codimensionOneSingularValueDecomposition}. Then
\[
\prod_{i=1}^d \frac{\sigma_i(BA)}{\sigma_i(A) \sigma_i(B)} \geq 
\widehat C_{\epsilon,d}^{-3} \ \gamma(\widehat{\tilde U}_A, \widecheck V_B).
\]
\end{lemma}


\begin{thebibliography}{99}

\bibitem{Blumenthal2015_1}
A. Blumenthal, A volume-based approach to the multiplicative ergodic 
theorem on Banach spaces, Discrete and Continuous Dynamical Systems, Vol. 36, No.  (2016).

\bibitem{BlumenthalMorris2015_1}
A. Blumenthal, I. D. Morris, Characterization of dominated splittings 
for operator cocycles acting on Banach spaces, preprint (2015).

\bibitem{BochiGourmelon2009_1}
J. Bochi, N. Gourmelon, Some characterization of domination, 
Math. Zeit\-schrift, Vol. 263 (2009), 221--231.

\bibitem{FroylandGonzalezTokmanQuas2014}
G. Froyland, C. Gonz\'alez-Tokman, A. Quas,  Stochastic stability of 
Lyapunov exponents and Oseledets splitting for semi-invertible matrix 
cocycles, Communications on Pure and Applied Mathematic, Vol. 68 (2015), 2052--2081.

\bibitem{GohbergKrein1957}
I.C. Gohberg, M.G. Krein, The basic propositions on defect numbers, root numbers and
indices of linear operator, Uspehi Mat. Nauk. 12 (74), 43--118; English translation, 
Amer. Math. Soc. Transl. (2) 13 (1960), 185--264. 

\bibitem{GohbergMarcus1959}
I.C. Gohberg, A.S. Marcus, Two theorems on the opening of subspaces of 
Banach space, Uspehi Mat. Nauk, Vol. 14, No. 5 (1959), 135--140.

\bibitem{GohbergKrein1969}
I.C. Gohberg, M.G Krein, Introduction to the Theory of Linear Nonselfadjoint Operators, 
Translations of Mathematical Monographs, Vol. 18 (1969), American Mathematical Society.


\bibitem{GonzalezTokmanQuas2014_1}
C. Gonz\'alez-Tokman, A. Quas, A concise proof of the multiplicative ergodic 
theorem on Banach spaces, Journal of Modern Dynamics, Vol. 9 (2015), 237--255.


\bibitem{Kato1958_01}
T. Kato, Perturbation theory for nullity, deficiency and other quantities of linear 
operators, J. Analyse Math. Vol. 6 (1958), 261--322.

\bibitem{Kato1995}
T. Kato, Perturbation Theory for Linear Operators, Springer-Verlag, Classics
in Mathematics (1995), reprint of the 1980 edition.

\bibitem{LianLu2010_1}
Z. Lian, K. Lu, Lyapunov Exponents and Invariant Manifolds for Random Dynamical 
Systems in A Banach space, Memoirs of the AMS, Vol. 206 (2010).

\bibitem{Pietsch1987}
A. Pietsch, Eigenvalues and s-numbers, Cambridge studies in advanced mathematics, Vol. 13, Cambridge University Press (1987).

\bibitem{Raghunathan1979_1}
M.S. Raghunathan, A proof of Oseledec's multiplicative ergodic theorem, Israel 
Journal of Mathematics, Vol. 32, No. 4 (1979), 356--362.


\bibitem{Ruelle1982_1}
D. Ruelle, Characteristic Exponents and Invariant Manifolds in Hilbert Space, 
Annals of Mathematics, Second Series, Vol. 115, No. 2 (1982), 243--290.

\end{thebibliography}
\end{document}